\documentclass[11pt]{amsart}
\usepackage{amsmath}
\usepackage{amssymb}
\usepackage{amsthm}
\usepackage{latexsym}
\usepackage{graphicx}
\usepackage{hyperref}
\usepackage{enumerate}
\usepackage[all]{xy}

\newtheorem{thm}{Theorem}[section]
\newtheorem{cor}[thm]{Corollary}
\newtheorem{lem}[thm]{Lemma}
\newtheorem{prop}[thm]{Proposition}

\newtheorem{thmintro}{Theorem}

\newtheorem{conj}[thmintro]{Conjecture}

\newcommand{\N}{\mathbb N}
\newcommand{\Z}{\mathbb Z}
\newcommand{\Q}{\mathbb Q}
\newcommand{\R}{\mathbb R}
\newcommand{\C}{\mathbb C}
\newcommand{\mf}{\mathfrak}
\newcommand{\mc}{\mathcal}
\newcommand{\mb}{\mathbf}
\newcommand{\mh}{\mathbb}
\def\Irr{{\rm Irr}}
\newcommand{\mr}{\mathrm}
\newcommand{\enuma}[1]{\begin{enumerate}[\textup{(}a\textup{)}] {#1} \end{enumerate}}
\newcommand{\Fr}{\mathrm{Frob}}
\newcommand{\Sc}{\mathrm{sc}}
\newcommand{\ad}{\mathrm{ad}}
\newcommand{\unr}{\mathrm{unr}}
\newcommand{\cusp}{\mathrm{cusp}}
\newcommand{\nr}{\mathrm{nr}}
\newcommand{\Wr}{\mathrm{wr}}

\newcommand{\Rep}{\mathrm{Rep}}
\newcommand{\af}{\mathrm{aff}}
\def\tor{{\rm tor}}
\newcommand{\unip}{\mathrm{unip}}
\newcommand{\der}{\mathrm{der}}
\newcommand{\matje}[4]{\left(\begin{smallmatrix} #1 & #2 \\ 
#3 & #4 \end{smallmatrix}\right)}
\newcommand{\fdeg}{\mathrm{fdeg}}
\newcommand{\vol}{\mathrm{vol}}
\newcommand{\Hom}{\mathrm{Hom}}

\newcommand{\temp}{\mathrm{temp}}
\newcommand{\bdd}{\mathrm{bdd}}

\begin{document}

\title{On formal degrees of unipotent representations}
\author{Yongqi Feng}
\address{Korteweg-de Vries Institute for Mathematics\\
Universiteit van Amsterdam, Science Park 105-107\\ 
1098 XG Amsterdam, The Netherlands}
\email{Yongqi.Feng@science.ru.nl}
\author{Eric Opdam}
\address{Korteweg-de Vries Institute for Mathematics\\
Universiteit van Amsterdam, Science Park 105-107\\ 
1098 XG Amsterdam, The Netherlands}
\email{e.m.opdam@uva.nl}
\author{Maarten Solleveld}
\address{Institute for Mathematics, Astrophysics and Particle Physics\\
Radboud Universiteit, Heyendaalseweg 135\\
6525AJ Nijmegen, the Netherlands}
\email{m.solleveld@science.ru.nl} 
\date{\today}
\thanks{The author is supported by a NWO Vidi grant "A Hecke algebra approach to the 
local Langlands correspondence" (nr. 639.032.528).}
\maketitle
\vspace{5mm}

\begin{abstract}
Let $G$ be a reductive $p$-adic group which splits over an unramified extension
of the ground field. Hiraga, Ichino and Ikeda conjectured that the formal degree
of a square-integrable $G$-representation $\pi$ can be expressed in terms of the adjoint
$\gamma$-factor of the enhanced L-parameter of $\pi$. A similar conjecture was posed
for the Plancherel densities of tempered irreducible $G$-representations.

We prove these conjectures for unipotent $G$-representations. 
We also derive explicit formulas for the involved adjoint $\gamma$-factors.
\end{abstract}

\vspace{5mm}

\tableofcontents


\section*{Introduction}

Let $\mc G$ be a connected reductive group defined over a non-archimedean local field $K$, 
and write $G = \mc G(K)$. We are interested in irreducible $G$-representations, always tacitly 
assumed to be smooth and over the complex numbers. The most basic example of such 
representations are the unramified or spherical representations \cite{Mac,Sat} of $G$, which 
play a fundamental role in the Langlands correspondence by virtue of the Satake isomorphism.

By a famous result of Borel \cite{Bor1,Cas}, the smallest block of the category of 
of smooth representations of $G$ which contains the spherical representations is the 
abelian subcategory generated by the unramified minimal principal series representations. 
The objects in this block are 
smooth representations which are generated by the vectors which are fixed by an Iwahori
subgroup $I$ of $G$. The study of such Iwahori-spherical representations is a classical topic, 
about which a lot is known.

The local Langlands correspondence for Iwahori-spherical representations was established 
by Kazhdan and Lusztig \cite{KL}, for $\mc G$ split simple of adjoint type. It parameterizes the 
irreducible Iwahori-spherical representations with enhanced unramified Deligne--Langlands parameters 
for $G$, where a certain condition is imposed on the enhancements. The category 
of representations of $G$ which naturally completes this picture (by the lifting the restriction
on the enhancements) is the category of so-called unipotent representations, as envisaged by Lusztig. 
An irreducible smooth representation of $G$ is called unipotent if its restriction 
to some parahoric subgroup $P_{\mf f}$ of $G$ contains a unipotent representation of $P_{\mf f}$. 
In the special case that $P_{\mf f}$ is an Iwahori subgroup of $G$, 
we recover the Iwahori-spherical representations.

Unipotent representations of simple adjoint groups over $K$ were classified by Lusztig
\cite{LusUni1,LusUni2}. Beyond such groups, the theory works best if $G$ splits over an unramified
extension of $K$, so we tacitly assume that throughout the introduction. The authors exhibited
a local Langlands correspondence for supercuspidal unipotent representations of reductive groups
over $K$ in \cite{FeOp,FOS}. Next the second author generalized this to a Langlands parametrization
of all tempered unipotent representations in \cite{Opd18}. Finally, with different methods the third
author constructed a local Langlands correspondence for all unipotent representations of reductive
groups over $K$ \cite{SolLLCunip}. In Theorem \ref{thm:1.1} we show that the approaches from 
\cite{Opd18} and \cite{SolLLCunip} agree, and we derive some extra properties of these instances
of a local Langlands correspondence.

Hiraga, Ichino and Ikeda \cite{HII} suggested that, for any irreducible tempered representation
$\pi$ of a reductive $p$-adic group, there is a relation between the Plancherel density of
$\pi$ and the adjoint $\gamma$-factor of its L-parameter. In fact, they conjectured an explicit
formula, to be sketched below in terms of a (tentative) enhanced L-parameter of $\pi$.

Let ${}^L G$ be the Langlands dual group of $G$, with identity component $G^\vee$. 
Let $\pi \in \Irr (G)$ be square-integrable modulo centre and suppose that $(\phi_\pi, \rho_\pi)$ 
is its enhanced L-parameter (so we need to assume that a local Langlands correspondence has been
worked out for $\pi$). To measure the size of the L-packet we use the group 
\[
S_{\phi_\pi}^\sharp := \pi_0 \big( Z_{(G / Z(G)_s)^\vee} (\phi_\pi) \big) ,
\] 
where $Z(G)_s$ denotes the maximal $K$-split central torus in $G$.
Let $\mr{Ad}_{G^\vee}$ denote the adjoint representation of ${}^L G$ on
\[
\mr{Lie}(G^\vee) \big/ \mr{Lie} \big( Z(G^\vee)^{\mb W_K} \big) \cong 
\mr{Lie} \big( (G / Z(G)_s)^\vee \big) .
\]
We refer to \eqref{eq:A.1} for the definition of the adjoint $\gamma$-factor
$\gamma (s,\mr{Ad}_{G^\vee} \circ \phi, \psi)$.

We endow $K$ with the Haar measure that gives its ring of integers volume 1 and we 
normalize the Haar measure on $G$ as in \cite[(1.1) and Correction]{HII}. 
It was conjectured in \cite[\S 1.4]{HII} that
\begin{equation}\label{eq:HII1}
\fdeg (\pi) = \dim (\rho_\pi) |S_{\phi_\pi}^\sharp|^{-1} | 
\gamma (0,\mr{Ad}_{G^\vee} \circ \phi_\pi,\psi)| .
\end{equation}
More generally, let $\mc P = \mc M \mc U$ be a parabolic $K$-subgroup of $\mc G$,
with Levi factor $\mc M$ and unipotent radical $\mc U$. Let $\pi \in \Irr (M)$
be square-integrable modulo centre and let $X_\unr (M)$ be the group of unitary
unramified characters of $M$. Let $\mc O = X_\unr (M) \pi \subset \Irr (M)$
be the orbit in $\Irr (M)$ of $\pi$, under twists by $X_\unr (M)$. We define
a Haar measure of d$\mc O$ on $\mc O$ as in 
\cite[p. 239 and 302]{Wal}. This also provides a Haar measure on the family
of (finite length) $G$-representations $I_P^G (\pi')$ with $\pi' \in \mc O$.

Denote the adjoint representation of ${}^L M$ on
Lie$\big( G^\vee) / \mr{Lie}(Z(M^\vee)^{\mb W_K} \big)$ by $\mr{Ad}_{G^\vee,M^\vee}$.

\begin{conj}\textup{\cite[\S 1.5]{HII}} \
\label{conj:HII}
Suppose that the enhanced L-parameter of $\pi$ is $(\phi_\pi,\rho_\pi) \in \Phi_e (M)$. 
Then the Plancherel density at $I_P^G (\pi) \in \mr{Rep}(G)$ is
\[
c_M \dim (\rho_\pi) |S_{\phi_\pi}^\sharp|^{-1} | 
\gamma (0,\mr{Ad}_{G^\vee,M^\vee} \circ \phi_\pi,\psi)| \, \textup{d}\mc O (\pi) ,
\]
for some constant $c_M \in \R_{>0}$ independent of $K$ and $\mc O$.
\end{conj}

We point out that the validity of \eqref{eq:HII1} and of Conjecture \ref{conj:HII} does not
depend on the choice of the additive character $\psi : K \to \C^\times$. For another choice 
of $\psi$ the adjoint $\gamma$-factors will be different \cite[Lemma 1.3]{HII}. But also 
the normalization of the Haar measure on $G$ has to be modified, which changes the formal
degrees \cite[Lemma 1.1]{HII}. These two effects precisely compensate each other.

We note that representations of the form $I_P^G (\pi)$ are tempered \cite[Lemme III.2.3]{Wal}
and that almost all of them are irreducible \cite[Proposition IV.2.2]{Wal}. Every 
irreducible tempered $G$-representation appears as a direct summand of $I_P^G (\pi_M)$, 
for suitable choices of the involved objects \cite[Proposition III.4.1]{Wal}. 
Moreover, if $I_P^G (\pi_M)$ is reducible, its decomposition can be analysed quite 
explicitely in terms of R-groups \cite{Sil1}. In this sense Conjecture \ref{conj:HII}
provides an expression for the Plancherel densities of all tempered irreducible $G$-representations.\\

The HII-conjectures were proven for supercuspidal unipotent representations in 
\cite{Re1,FeOp,FOS}, for unipotent representations of simple adjoint groups in \cite{Opd2}
and for tempered unipotent representations in \cite{Opd18}. However, in the last case
the method only sufficed to establish the desired formulas up to a constant. Of course
the formal degree of a square-integrable representation is just a number, so a priori
one gains nothing from knowing it up to a constant. Fortunately, the formal degree
of a unipotent square-integrable representation can be considered as a rational function
of the cardinality $q$ of the residue field of $K$ \cite{Opd2}. Then "up to a constant"
actually captures a substantial part of the information. The main result of this
paper is a complete proof of the HII-conjectures for unipotent representations:

\begin{thmintro}\label{thm:B}
Let $\mc G$ be a connected reductive $K$-group which splits over an unramified extension
and write $G = \mc G (K)$. Use the local Langlands correspondence for unipotent
$G$-representations from Theorem \ref{thm:1.1}.
\enuma{
\item The HII-conjecture \eqref{eq:HII1} holds for all unipotent, square-integrable 
modulo centre $G$-representations.
\item Conjecture \ref{conj:HII} holds for tempered unipotent $G$-representations, 
in the following slightly stronger form:
\[
\textup{d} \mu_{Pl}(I_P^G (\pi)) = \pm \dim (\rho_\pi) |S_{\phi_\pi}^\sharp|^{-1}  
\gamma (0,\mr{Ad}_{G^\vee,M^\vee} \circ \phi_\pi,\psi) \, \textup{d}\mc O (\pi) .
\]
}
\end{thmintro}

In the appendix we work out explicit formulas for the above adjoint $\gamma$-factors,
in terms of a maximal torus $T^\vee \subset G^\vee$ and the root system of $(G^\vee,T^\vee)$
(Lemma \ref{lem:A.4} and Theorem \ref{thm:A.2}). These expressions can also be interpreted
with $\mu$-functions for a suitable affine Hecke algebra \cite{Opd-Sp}. The calculations
entail in particular that all involved adjoint $\gamma$-factors are real numbers 
(Lemma \ref{lem:A.5}). \\

Our proof of Theorem \ref{thm:B} proceeds stepwise, in increasing generality.
The most difficult case is unipotent square-integrable representations of semisimple groups. 
The argument for that case again consists of several largely independent parts.
First we recall (\S \ref{par:adj}) that \eqref{eq:HII1} has already been proved for 
square-integrable representations of adjoint groups \cite{Opd2,FOS}.

Our main strategy is pullback of representations along the adjoint quotient map
$\eta : \mc G \to \mc G_\ad$. The homomorphism of $K$-rational points $\eta : G \to G_\ad$
need not be surjective, so this pullback operation need not preserve irreducibility of
representations. For $\pi_\ad \in \Irr (G_\ad)$ the computation of the length of
$\eta^* (\pi_\ad)$ has two aspects. On the one hand we determine in \S \ref{sec:pullback}
how many Bernstein components for $G$ are involved. On the other hand, we study the
decomposition within one Bernstein component in \S \ref{sec:AHA}. The latter is done
in terms of affine Hecke algebras, via the types and Hecke algebras from 
\cite{Mor1,Mor2,LusUni1}. Considerations with affine Hecke algebras also allow us to find
the exact ratio between fdeg$(\pi_\ad)$ and the formal degree of any irreducible 
constituent of $\eta^* (\pi_\ad)$, see Theorem \ref{thm:3.4}.

On the Galois side of the local Langlands correspondence, the comparison between $G$ and $G_\ad$
is completely accounted for by results from \cite{SolFunct}. In Lemma \ref{lem:3.2} we put
those in the form that we actually need. With all these partial results at hand, we finish
the computation of the formal degrees of unipotent square-integrable representations of
semisimple in Theorem \ref{thm:3.10}.\\

The generalization from semisimple groups to square-integrable
modulo centre representations of reductive groups (\S \ref{par:red}) is not difficult,
because unipotent representations of $p$-adic tori are just weakly unramified characters.
That proves part (a) of Theorem \ref{thm:B}. 

To get part (b) for square-integrable modulo centre representations (so with $M = G$), we
need to carefully normalize the involved Plancherel measures (\S \ref{par:norm}).
In \S \ref{par:par} we establish part (b) for any Levi subgroup $M \subset G$. This involves
a translation to Plancherel densities for affine Hecke algebras, via the aforementioned types. 
In the final stage we use that Theorem \ref{thm:B} was already known up to constants \cite{Opd18}.

\section{Background on unipotent representations}
\label{sec:background}

Let $K$ be a non-archimedean local field with ring of integers $\mf o_K$ and uniformizer
$\varpi_K$. Let $k = \mf o_K / \varpi_K \mf o_K$ be its residue field, of cardinality 
$q = q_K$.

Let $K_s$ be a separable closure of $K$. Let $\mb W_K \subset \mr{Gal}(K_s / K)$ be the 
Weil group of $K$ and let $\Fr$ be an arithmetic Frobenius element. Let $\mb I_K$ be the
inertia subgroup of Gal$(K_s/K)$, so that $\mb W_K / \mb I_K \cong \Z$ is generated by $\Fr$. 

Let $\mc G$ be a connected reductive $K$-group. Let $\mc T$ be a maximal torus 
of $\mc G$, and let $\Phi (\mc G, \mc T)$ 
be the associated root system. We also fix a Borel subgroup $\mc B$ of $\mc G$ 
containing $\mc T$, which determines a basis $\Delta$ of $\Phi (\mc G, \mc T)$. 

Let $\Phi (\mc G,\mc T)^\vee$ be the dual root system of $\Phi (\mc G,\mc T)$,
contained in the cocharacter lattice $X_* (\mc T)$. The based root datum of $\mc G$ is
\[
\big( X^* (\mc T), \Phi (\mc G,\mc T), X_* (\mc T), \Phi (\mc G,\mc T)^\vee, \Delta \big) .
\]
Let $\mc S$ be a maximal $K$-split torus in $\mc G$. By \cite[Theorem 13.3.6.(i)]{Spr} 
applied to $Z_{\mc G}(\mc S)$, we may assume that $\mc T$ is defined over $K$ and 
contains $\mc S$. Then $Z_{\mc G}(\mc S)$ is a minimal $K$-Levi subgroup of $\mc G$. Let
\[
\Delta_0 := \{ \alpha \in \Delta : \mc S \subset \ker \alpha \}
\]
be the set of simple roots of $(Z_{\mc G}(\mc S), \mc T)$. 

Recall from \cite[Lemma 15.3.1]{Spr} that the root system $\Phi (\mc G, \mc S)$ is the 
image of $\Phi (\mc G, \mc T)$ in $X^* (\mc S)$, without 0. The set of simple roots of
$(\mc G, \mc S)$ can be identified with $(\Delta \setminus \Delta_0 ) / \mu_{\mc G}(\mb W_K)$,
where $\mu_{\mc G}$ denotes the action of Gal$(K_s / K)$ on $\Delta$ determined by
$(\mc B,\mc T)$.

We write $G = \mc G(K)$ and similarly for other $K$-groups. 
Let $\mc G^\vee$ be the split reductive group with based root datum
\[
\big( X_* (\mc T), \Phi (\mc G,\mc T)^\vee, X^* (\mc T), 
\Phi (\mc G,\mc T), \Delta^\vee \big) . 
\]
Then $G^\vee = \mc G^\vee (\C)$ is the complex dual group of $\mc G$. Via the choice of
a pinning, the action $\mu_{\mc G}$ of $\mb W_K$ on the root datum of $\mc G$ determines 
an action of $\mb W_K$ of $G^\vee$. That action stabilizes
the torus $T^\vee = X^* (\mc T) \otimes_\Z \C^\times$ and the Borel subgroup
$B^\vee$ determined by $T^\vee$ and $\Delta^\vee$. The Langlands dual group (in the 
version based on $\mb W_K$) of $\mc G (K)$ is ${}^L G := G^\vee \rtimes \mb W_K$.

Define the abelian group 
\[
\Omega = X_* (\mc T)_{\mb I_K} / (\Z \Phi (\mc G, \mc T)^\vee)_{\mb I_K} .
\] 
Then $Z(G^\vee)$ can be identified with $\Irr (\Omega) = \Omega^*$, and $\Omega$ is 
naturally isomorphic to the group $X^* (Z(G^\vee))$ of algebraic characters of 
$Z(G^\vee)$. In particular
\begin{equation}\label{eq:1.6}
\Omega^{\mb W_K} \cong X^* \big( Z(G^\vee) \big)^{\mb W_K} = 
X^* \big( Z(G^\vee)_{\mb W_K} \big) .
\end{equation}
To indicate the group and to reconcile the notations from \cite{FOS} and 
\cite{SolLLCunip} we write
\[
\Omega_G = \Omega^{\mb W_K}. 
\]
Kottwitz defined a natural, surjective group homomorphism $\kappa_G : G \to \Omega_G$.
The action of $\ker (\kappa_G)$ on the Bruhat--Tits building preserves the types
of facets (i.e. preserves a coloring of the vertices). Further, the kernel of $\kappa_G$ 
contains the image (in $G$) of the simply connected cover of the derived group of $G$, see
\cite[Appendix]{HR}. We say that a character of $G$ is weakly unramified if it is 
trivial on $\ker (\kappa_G)$. Thus the weakly unramified characters of $G$ can
be identified with the characters of $\Omega_G$.

Let $Z(\mc G)_s$ be the maximal $K$-split torus in $Z(\mc G)$. 
As $H^1 (K, Z(\mc G)_s) = 1$, there is a short exact sequence
\begin{equation}\label{eq:1.1}
1 \to Z(\mc G)_s (K) \to \mc G (K) \to (\mc G / Z (\mc G)_s)(K) \to 1 . 
\end{equation}
In view of the naturality of the Kottwitz homomorphism $\kappa_G$, this induces a 
short exact sequence
\begin{equation}\label{eq:1.12}
1 \to \Omega_{Z(G)_s} \to \Omega_G \to \Omega_{G / Z(G)_s} \to 1.
\end{equation}
Recall \cite[Part 3]{Lus-Che} that an 
irreducible representation of a reductive group over a finite field is called unipotent if it 
appears in the Deligne--Lusztig series associated to the trivial character of a maximal
torus in that group. An irreducible representation of a linear algebraic group over $\mf o_K$ 
is called unipotent if its arises, by inflation, from a unipotent representation of the maximal 
finite reductive quotient of the group.

We call an irreducible smooth $G$-representation $\pi$ unipotent if there exists a parahoric 
subgroup $P_{\mf f} \subset G$ such that $\pi |_{P_{\mf f}}$ contains an irreducible unipotent 
representation of $P_{\mf f}$. Then the restriction of $\pi$ to some smaller parahoric subgroup
$P_{\mf f'} \subset G$ contains a cuspidal unipotent representation of $P_{\mf f'}$, as required 
in \cite{LusUni1}. An arbitrary smooth $G$-representation is unipotent if it 
lies in a product of Bernstein components, all whose cuspidal supports are unipotent. 

The category of unipotent $G$-representations can be described in terms of types and
affine Hecke algebras. For a facet $\mf f$ of the Bruhat--Tits building $\mc B (\mc G,K)$
of $G$, let $\mc G_{\mf f}$ be the smooth affine $\mf o_K$-group scheme from \cite{BrTi2}, 
such that $\mc G_{\mf f}^\circ$ is a $\mf o_K$-model of $\mc G$ and
$\mc G_f^\circ (\mf o_K)$ equals the parahoric subgroup $P_{\mf f}$ of $G$. Then $
\hat{P}_{\mf f} := \mc G_{\mf f}(\mf o_K)$ is the pointwise stabilizer of $\mf f$ in $G$. 
Let $\overline{\mc G_{\mf f}}$ be the maximal reductive quotient of the $k$-group
scheme obtained from $\mc G_{\mf f}$ by reduction modulo $\varpi_K$. Thus 
\[
\overline{\mc G_{\mf f}}(k) = \hat{P}_{\mf f} / U_{\mf f} \quad \text{and} \quad
\overline{\mc G_{\mf f}^\circ}(k) = P_{\mf f} / U_{\mf f} ,
\]
where $U_{\mf f}$ is the pro-unipotent radical of $P_{\mf f}$. By \cite[\S 5.1]{DeRe}
our normalized Haar measure on $G$ satisfies
\begin{equation}\label{eq:1.3}
\vol (P_{\mf f}) = | \overline{\mc G_{\mf f}^\circ} (k) | \;
q^{- \dim \overline{\mc G_{\mf f}^\circ} / 2} .
\end{equation}
By the counting formulas for reductive groups over finite fields \cite[Theorem 9.4.10]{Car1}, 
$| \overline{\mc G_{\mf f}^\circ} (k) |$ can be considered as a polynomial in $q = |k|$.
The above notions behave best when $\mc G$ splits over an unramified extension of $K$, so we
assume that in most of the paper.

Replacing the involved objects by a suitable $G$-conjugate, we can achieve that $\mf f$ lies 
in the closure of a fixed "standard" chamber $C_0$ of the apartment of $\mc B (\mc G,K)$ 
associated to $\mc S$. Since $\mc G$ splits over an unramified extension, the group
$\Omega_G = \Omega^{\mb W_K}$ from \eqref{eq:1.6} equals $\Omega^\Fr$. It acts naturally on
$\overline{C_0}$, and we denote the setwise stablizer of $\mf f$ by $\Omega_{G,\mf f}$
and the pointwise stabilizer of $\mf f$ by $\Omega_{G,\mf f,\mr{tor}}$.
It was noted in \cite[(32)]{SolLLCunip} that 
\begin{equation}\label{eq:1.7}
\hat{P}_{\mf f} / P_{\mf f} \cong \Omega_{G,\mf f,\mr{tor}} .
\end{equation}
Suppose that $(\sigma, V_\sigma)$ is a cuspidal unipotent representation of 
$\overline{\mc G_{\mf f}^\circ}(k)$ (in particular this includes that it is irreducible).
It was shown in \cite[\S 6]{MoPr2} and \cite[Theorem 4.8]{Mor2} that $(P_{\mf f},\sigma)$ 
is a type for $G$. Let $\Rep (G)_{(P_{\mf f},\sigma)}$ be the corresponding direct factor
of $\Rep (G)$. By \cite[1.6.b]{LusUni1} 
\begin{equation}\label{eq:1.11}
\begin{array}{ll}
\Rep (G)_{(P_{\mf f},\sigma)} = \Rep (G)_{(P_{\mf f'},\sigma')} & 
\text{if } g \mf f' = \mf f, \mr{Ad}(g)^* \sigma = \sigma'
\text{ for some } g \in G \hspace{-5mm} \\
\Rep (G)_{(P_{\mf f},\sigma)} \cap \Rep (G)_{(P_{\mf f'},\sigma')} = \{ 0 \} & \text{otherwise.}
\end{array}
\end{equation}
By \cite[\S 1.16]{LusUni1} and \cite[Lemma 15.7]{FOS} $\sigma$ can be extended (not uniquely) 
to a representation of $\overline{\mc G_{\mf f}}(k)$, which we inflate to an irreducible 
representation of $\hat{P}_{\mf f}$ that we denote by $(\hat \sigma, V_\sigma)$.
It is known from \cite[Theorem 4.7]{Mor2} that $(\hat{P}_{\mf f},\hat \sigma)$ is
a type for a single Bernstein block $\mr{Rep}(G)^{\mf s}$. Conversely, every Bernstein
block consisting of unipotent $G$-representations is of this form. We note that
$\Rep (G)_{(P_{\mf f},\sigma)}$ is the direct sum of the $\Rep (G)^{\mf s}$ associated
to the different extensions of $\sigma$ to $\hat{P}_{\mf f}$.

To $(\hat{P}_{\mf f},\hat \sigma)$ Bushnell and Kutzko \cite{BuKu} associated an algebra 
$\mc H (G,\hat{P}_{\mf f},\hat \sigma )$, such that there is an equivalence of categories
\begin{equation}\label{eq:1.4}
\begin{array}{ccc}
\mr{Rep}(G)^{\mf s} & \to & \mr{Mod}(\mc H (G,\hat{P}_{\mf f},\hat \sigma ) ) \\
\pi & \mapsto & \Hom_{\hat{P}_{\mf f}}( \hat \sigma, \pi)
\end{array} .
\end{equation}
It turns out that $\mc H (G,\hat{P}_{\mf f},\hat \sigma )$ is an affine Hecke algebra,
see \cite[\S 1]{LusUni1} and \cite[\S 3]{SolLLCunip}. Moreover a finite length
representation in $\mr{Rep}(G)^{\mf s}$ is tempered (resp. essentially square-integrable)
if and only if the associated $\mc H (G,\hat{P}_{\mf f},\hat \sigma )$-module is
tempered (resp. essentially discrete series) \cite[Theorem 3.3.(1)]{BHK}.

The affine Hecke algebra $\mc H (G,\hat{P}_{\mf f},\hat \sigma )$ comes with the
following data:
\begin{itemize}
\item a lattice $X_{\mf f}$ and a complex torus $T_{\mf f} = \Irr (X_{\mf f})$;
\item a root system $R_{\mf f}$ in $X_{\mf f}$, with a basis $\Delta_{\mf f}$;
\item a Coxeter group $W_\af = W(R_{\mf f}) \ltimes \Z R_{\mf f}$ in 
$W(R_{\mf f}) \ltimes X_{\mf f}$;
\item a set $S_{\mf f,\af}$ of affine reflections, which are Coxeter generators of $W_\af$;
\item a parameter function $q^{\mc N} : W_\af \to \R_{>0}$.
\end{itemize}
Furthermore it has a distinguished basis $\{ N_w : w \in W(R_{\mf f}) \ltimes X_{\mf f} \}$,
an involution * and a trace $\tau$. Thus $\mc H (G,\hat{P}_{\mf f},\hat \sigma )$ has the
structure of a Hilbert algebra, and one can define a Plancherel measure and formal degrees
for its representations. The unit element $N_e$ of $\mc H (G,\hat{P}_{\mf f},\hat \sigma )$ 
is the central idempotent $e_{\hat \sigma}$ (in the group algebra of $\hat{P}_{\mf f}$) 
associated to $\hat \sigma$. The trace $\tau$ is normalized so that
\begin{equation}\label{eq:1.10}
\tau (N_w) = \left\{ \begin{array}{ll}
e_{\hat \sigma}(1) = \dim (\hat \sigma) \vol (\hat{P}_{\mf f})^{-1} & w = e \\
0 & w \neq e
\end{array} \right. .
\end{equation}
It follows from \cite[Theorem 3.3.(2)]{BHK} that, with this normalization, the equivalence of 
categories \eqref{eq:1.4} preserves Plancherel measures and formal degrees. 
For affine Hecke algebras, these were analysed in depth in \cite{Opd-Sp,OpSo,CiOp}. 

Consider a discrete series representation $\delta$ of $\mc H (G,\hat{P}_{\mf f},\hat \sigma )$,
with central character $W(R_{\mf f}) r \in T_{\mf f} / W(R_{\mf f})$. By \cite{Opd-Sp} its
formal degree can be expressed as 
\begin{equation}\label{eq:1.5}
\fdeg (\delta) = \pm \dim (\hat \sigma) \vol (\hat{P}_{\mf f})^{-1}  
d_{\mc H,\delta} m (q^{\mc N})^{(r)} ,
\end{equation}
where $d_{\mc H,\delta} \in \Q_{>0}$ is computed in \cite{CiOp} (often it is just 1).
The factor $m (q^{\mc N})$ is a rational function in $r \in T_{\mf f}$ and the 
parameters $q^{\mc N}(s_\alpha)^{1/2}$ with $s_\alpha \in S_{\mf f,\af}$, while the
superscript $(r)$ indicates that we take its residue at $r$. We refer to \eqref{eq:6.10} and
\eqref{eq:A.33} for the explicit definition of $m (q^{\mc N})$.

\section{Langlands parameters}
\label{sec:Langlands}

Recall that a Langlands parameter for $G$ is a homomorphism 
\[
\phi : \mb W_K \times SL_2 (\C) \to {}^L G = G^\vee \rtimes \mb W_K ,
\]
with some extra requirements. In particular $\phi |_{SL_2 (\C)}$ has to be algebraic, 
$\phi (\mb W_K)$ must consist of semisimple elements and $\phi$ must respect the
projections to $\mb W_K$. 

We say that a L-parameter $\phi$ for $G$ is 
\begin{itemize}
\item discrete if it does not factor through the L-group of any proper Levi subgroup of $G$;
\item bounded if $\phi (\Fr) = (s,\Fr)$ with $s$ in a bounded subgroup of $G^\vee$;
\item unramified if $\phi (w) = (1,w)$ for all $w \in \mb I_K$.
\end{itemize}
Let ${G^\vee}_\ad$ be the adjoint group of $G^\vee$, and let ${G^\vee}_\Sc$ be its
simply connected cover. Let $\mc G^*$ be the unique $K$-quasi-split inner form of $\mc G$. 
We consider $\mc G$ as an inner twist of $\mc G^*$, so endowed with a $K_s$-isomorphism
$\mc G \to \mc G^*$. Via the Kottwitz isomorphism $\mc G$ is labelled by a character
$\zeta_{\mc G}$ of $Z({G^\vee}_\Sc)^{\mb W_K}$ (defined with respect to $\mc G^*$).
We choose an extension $\zeta$ of $\zeta_{\mc G}$ to $Z({G^\vee}_\Sc)$. As explained
in \cite[\S 1]{FOS}, this is related to the explicit realization of $\mc G$ as an
inner twist of $\mc G^*$.

Both ${G^\vee}_\ad$ and ${G^\vee}_\Sc$ act on $G^\vee$ by conjugation. As
\[
Z_{G^\vee}(\text{im } \phi) \cap Z(G^\vee) = Z(G^\vee)^{\mb W_K} ,
\]
we can regard $Z_{G^\vee}(\text{im } \phi) / Z(G^\vee)^{\mb W_K}$ as a subgroup of 
${G^\vee}_\ad$.  Let $Z^1_{{G^\vee}_\Sc}(\text{im } \phi)$ be its inverse image in
${G^\vee}_\Sc$ (it contains $Z_{{G^\vee}_\Sc}(\text{im } \phi)$ with finite index). 
A subtle version of the component group of $\phi$ is
\[
\mc A_\phi := \pi_0 \big( Z^1_{{G^\vee}_\Sc}(\text{im } \phi) \big) .
\]
An enhancement of $\phi$ is an irreducible representation $\rho$ of $\mc A_\phi$.
Via the canonical map $Z({G^\vee}_\Sc) \to \mc A_\phi$, $\rho$ determines
a character $\zeta_\rho$ of $Z({G^\vee}_\Sc)$. 
We say that an enhanced L-parameter $(\phi,\rho)$ is relevant for $G$ if $\zeta_\rho =
\zeta$. This can be reformulated with $G$-relevance of $\phi$ in terms of Levi
subgroups \cite[Lemma 9.1]{HiSa}. To be precise, in view of \cite[\S 3]{Bor}
there exists an enhancement $\rho$ such that $(\phi,\rho)$ is $G$-relevant if and only if 
every L-Levi subgroup of ${}^L G$ containing the image of $\phi$ is $G$-relevant. 
The group $G^\vee$ acts naturally on the collection of $G$-relevant enhanced 
L-parameters, by 
\[
g \cdot (\phi,\rho) = (g \phi g^{-1},\rho \circ \mr{Ad}(g)^{-1}) .
\]
We denote the set of $G^\vee$-equivalence classes of $G$-relevant (resp. enhanced) 
L-parameters by $\Phi (G)$, resp. $\Phi_e (G)$. A local Langlands correspondence 
for $G$ (in its modern interpretation) should be a bijection between $\Phi_e (G)$ and 
the set $\Irr (G)$ of irreducible smooth $G$-representations, with several nice properties.\\

We denote the set of irreducible unipotent $G$-representations by $\Irr_\unip (G)$. 
The next theorem is a combination of the main results of \cite{FOS,SolLLCunip,Opd18}.

\begin{thm}\label{thm:1.1}
Let $\mc G$ be a connected reductive $K$-group which splits over an \\
unramified extension. There exists a bijection
\[
\begin{array}{ccc}
\Irr_\unip (G) & \longrightarrow & \Phi_{\nr,e}(G) \\
\pi & \mapsto & (\phi_\pi, \rho_\pi) \\
\pi (\phi,\rho) & \text{\rotatebox[origin=c]{180}{$\mapsto$}} & (\phi,\rho)
\end{array}
\]
with the following properties.
\begin{enumerate}[(a)]
\item Compatibility with direct products of reductive $K$-groups.
\item Equivariance with respect to the canonical actions of the group $X_\Wr (G)$ of
weakly unramified characters of $G$.
\item The central character of $\pi$ equals the character of $Z(G)$ determined by $\phi_\pi$.
\item $\pi$ is tempered if and only if $\phi_\pi$ is bounded.
\item $\pi$ is essentially square-integrable if and only if $\phi_\pi$ is discrete.
\item $\pi$ is supercuspidal if and only if $(\phi_\pi, \rho_\pi)$ is cuspidal.
\item The analogous bijections for the Levi subgroups of $G$ and the cuspidal
support maps form a commutative diagram
\[
\begin{array}{ccc}
\Irr_\unip (G) & \longrightarrow & \Phi_{\nr,e}(G) \\
\downarrow & & \downarrow \\
\bigsqcup_M \Irr_{\cusp,\unip}(M) \big/ N_G (M)  & 
\longrightarrow & \bigsqcup_M \Phi_{\nr,\cusp}(M) \big/ N_{G^\vee} (M^\vee \rtimes \mb W_K) 
\end{array} .
\]
Here $M$ runs over a collection of representatives for the conjugacy classes of Levi subgroups 
of $G$. See \cite[\S 2]{SolLLCunip} for explanation of the notation in the diagram.
\item Suppose that $P = M U$ is a parabolic subgroup of $G$ and that 
$(\phi,\rho^M) \in \Phi_{\nr,e}(M)$ is bounded. Then the normalized parabolically induced 
representation $I_P^G \pi (\phi,\rho^M)$ is a direct sum of representations $\pi (\phi,\rho)$, 
with multiplicities $[\rho^M : \rho]_{\mc A_{\phi}^M}$.
\item Compatibility with the Langlands classification for representations of reductive groups
and the Langlands classification for enhanced L-parameters.
\item Compatibility with restriction of scalars of reductive groups over 
non-archimedean local fields.
\item Let $\tilde{\mc G}$ be a group of the same kind as $\mc G$, and let 
$\eta : \tilde{\mc G} \to \mc G$ be a homomorphism of $K$-groups such that the kernel of 
$\textup{d}\eta : \mr{Lie}(\tilde{\mc G}) \to \mr{Lie}(\mc G)$ is central and the cokernel of 
$\eta$ is a commutative $K$-group. Let ${}^L \eta : {}^L \tilde G \to {}^L G$ be the dual 
homomorphism and let $\phi \in \Phi_\nr (G)$. 

Then the L-packet $\Pi_{{}^L \eta \circ \phi}(\tilde G)$ consists precisely of the constituents 
of the completely reducible $\tilde G$-representations $\eta^* (\pi)$ with $\pi \in \Pi_\phi (G)$.
\item Conjecture \ref{conj:HII} holds for tempered unipotent $G$-representations,
up to some rational constants that depend only on an orbit $\mc O$.
\end{enumerate}
Moreover the above properties determine the surjection 
\[
\Irr_\unip (G) \to \Phi_\nr (G): \pi \mapsto \phi_\pi
\]
uniquely, up to twisting by weakly unramified characters of $G / Z(G)_s$.
\end{thm}
\textbf{Remark.} We regard this as a local Langlands correspondence for unipotent representations.
We point out that for simple adjoint groups Theorem \ref{thm:1.1} differs somewhat from
the main results of \cite{LusUni1,LusUni2} -- which do not satisfy (d) and (e).
\begin{proof}
A bijection satisfying the properties (a)--(i) was exhibited in \cite[\S 5]{SolLLCunip}.
The construction involves some arbitrary choices, we will fix some of those here.

For property (j) see \cite[Lemma A.3]{FOS} and \cite[Lemma 2.4]{SolLLCunip}.
For property (k) we refer to \cite[Corollary 5.8 and \S 7]{SolFunct}.

Denote the set of tempered irreducible smooth $G$-representations by $\Irr_\temp (G)$ and 
let $\Phi_\bdd (G)$ be the collection of bounded L-parameters for $G$. It was shown in 
\cite[Theorem 4.5.1]{Opd18} that there exists a "Langlands parametrization"
\begin{equation}\label{eq:Lparamet}
\phi_{HII} : \Irr_{\unip,\temp}(G) \to \Phi_{\nr,\bdd}(G)
\end{equation}
which satisfies the above property (l) and is essentially unique. Notice that the image of 
$\phi_{HII}$ consists of L-parameters, not enhanced as before.
For supercuspidal representations both $\phi_{HII}$ and \cite{SolLLCunip} boil down to the
same source, namely \cite{FeOp,FOS}. There it is shown that, on the cuspidal level for a
Levi subgroup $M$ of $G$, the bijection 
\begin{equation}\label{eq:2.2}
\Irr_{\unip,\cusp}(M) \to \Phi_{\nr,\cusp}(M) 
\end{equation}
is unique up to twisting by $X_\Wr (M / Z(M)_s)$. For use in \cite{SolLLCunip} we may
pick any instance of \eqref{eq:2.2} from \cite[Theorem 2]{FOS}. For use in \cite[4.5.1]{Opd18}
there are some extra conditions, related to the existence of suitable spectral transfer
morphisms. We fix a set $\mf{Lev}(G)$ of representatives for the conjugacy classes of
Levi subgroups of $G$. For every $M \in \mf{Lev}(G)$ we choose a bijection \eqref{eq:2.2}
which satisfies all the requirements from \cite{Opd18}. In this way we achieve that
\begin{equation}\label{eq:2.4}
\phi_\pi = \phi_{HII}(\pi) \in \Phi_{\nr,\bdd}(M) \quad 
\text{for every tempered } \pi \in \Irr_{\unip,\cusp}(M) .
\end{equation}
To prove property (l), we will show that
\begin{equation}\label{eq:2.3}
\phi_\pi = \phi_{HII}(\pi) \in \Phi_{\nr,\bdd}(G) 
\quad \text{for all } \pi \in \Irr_{\unip,\temp}(G) .
\end{equation}
The infinitesimal (central) character of an L-parameter $\phi$ is defined as
\[
\mr{inf.ch.}(\phi) = G^\vee\text{-conjugacy class of }
\phi \Big( \Fr, \matje{q^{-1/2}}{0}{0}{q^{1/2}} \Big) \in G^\vee \Fr .
\]
By the definition of L-parameters this is a semisimple adjoint orbit, and by 
\cite[Lemma 6.4]{Bor} it corresponds to a unique $W(\mc G^\vee,\mc T^\vee)^\Fr$-orbit
in $T^\vee_\Fr$. That in turn can be interpreted as a central character of the 
Iwahori--Hecke algebra $\mc H (G^*,I^*)$ of the quasi-split inner form $G^*$ of $G$.

By \cite[Theorems 3.8.1 and 4.5.1]{Opd18} the Langlands parametrization 
$\phi_{HII}$ is completely characterized by the map
\begin{equation}\label{eq:2.infch}
\mr{inf.ch.} \circ \phi_{HII} : \Irr_{\unip,\temp}(G) \to G^\vee \Fr / G^\vee\text{-conjugacy} .
\end{equation}
Hence \eqref{eq:2.3} is equivalent to:
\begin{equation}\label{eq:1.13}
\mr{inf.ch.}(\phi_\pi) = \mr{inf.ch.}(\phi_{HII} (\pi)) \qquad \text{for all }
\pi \in \Irr_{\unip,\temp}(G). 
\end{equation}
By construction the cuspidal support map for enhanced L-parameters preserves infinitesimal
characters, see \cite[Definition 7.7 and (108)]{AMS1}. Then property (g) says
that inf.ch.$(\phi_\pi)$ does not change if we replace $\pi$ by its supercuspidal support. 

The map \eqref{eq:2.infch} is constructed in \cite{Opd18} in three steps:
\begin{itemize}
\item Let $\mc H_{\mf s}$ be the Hecke algebra associated to a Bushnell-Kutzko type for 
the Bernstein block $\Rep (G)^{\mf s}$ that contains $\pi$, as in \eqref{eq:1.4}.
Consider the image $\pi_{\mc H}$ of $\pi$ in $\Irr (\mc H_{\mf s})$.
\item Compute the central character of $\pi_{\mc H}$, an orbit for the 
Weyl group $W_{\mf s}$ acting on the complex torus $T_{\mf s}$ --
both attached to $\mc H_{\mf s}$ as described after \eqref{eq:1.4}.
\item Apply a spectral transfer morphism $\mc H_{\mf s} \leadsto \mc H (G^*,I^*)$ and the
associated map $T_{\mf s} \to T^\vee_\Fr / K_L^n$ -- see the definitions in \cite[\S 5.1]{Opd1}. 
This map sends the central character of $\pi_{\mc H}$ to a unique 
$W(\mc G^\vee,\mc T^\vee)^\Fr$-orbit in $T^\vee_\Fr$, which we interpret as a semisimple 
$G^\vee$-orbit in $G^\vee \Fr$.
\end{itemize}
For irreducible $\mc H_{\mf s}$-modules, the central character map corresponds to restriction 
to the maximal commutative subalgebra $\mc O (T_{\mf s})$ of $\mc H_{\mf s}$. There is a Levi 
subgroup $M$ of $G$ with a type, covered by the type for $\Rep (G)^{\mf s}$, whose Hecke 
algebra is $\mc O (T_{\mf s})$. The equivalence of categories \eqref{eq:1.4} is compatible 
with normalized parabolic induction and Jacquet restriction \cite[Lemma 4.1]{SolComp}, so the 
central character map for $\mc H_{\mf s}$ corresponds to the supercuspidal support map for 
$\Rep (G)^{\mf s}$. 

As in \cite[\S 3.1.1]{Opd2}, $\mc H_{\mf s} \leadsto \mc H (G^*,I^*)$ can be restricted to 
a spectral transfer morphism $\mc O (T_{\mf s}) \leadsto \mc H (M^*,I^*)$, where the Levi 
subgroup $M^*$ of $G^*$ is the quasi-split inner form of $M$. Up to adjusting by an element of 
$W(\mc G^\vee,\mc T^\vee)^\Fr$, these two spectral transfer morphisms are represented by the
same map $T_{\mf s} \to T^\vee_\Fr / K_L^n$. Consequently \eqref{eq:2.infch} does not change 
if the input $\pi$ is replaced by its supercuspidal support. 
These considerations reduce \eqref{eq:1.13} and \eqref{eq:2.3} to \eqref{eq:2.4}. 

Now we have the bijection of the theorem and all its properties, except for the asserted
uniqueness. The L-parameters for $\Irr_{\unip,\temp}(G)$ completely determine the L-parameters
for all (not necessarily tempered) irreducible unipotent $G$-representations, that follows
from the compatibility with the Langlands classification \cite[Lemma 5.10]{SolLLCunip}. 
Hence it suffices to address the essential uniqueness for tempered representations and bounded 
L-parameters. For adjoint groups it was shown in \cite[Theorems 4.4.1.c and 4.5.1.b]{Opd18}.

The case where $Z(\mc G)$ is $K$-anisotropic is reduced to the adjoint case in the proof
of \cite[Theorem 4.5.1]{Opd18}. This proceeds by imposing compatibility of the Langlands
parametrization $\phi_{HII}$ with the isogeny $\mc G \to \mc G_\ad \times \mc G / \mc G_\der$, 
in the sense that: 
\begin{itemize}
\item every irreducible tempered unipotent representation of $G$ should be "liftable"
in an essentially unique way to one of $G_\ad \times (\mc G / \mc G_\der)(K)$, 
\item that should determine the L-parameters. 
\end{itemize}
In this way one concludes essential uniqueness in \cite[Theorem 4.5.1.b]{Opd18}, 
but in a weaker sense than we want. However, the compatibility of 
$\mc G \to \mc G_\ad \times \mc G / \mc G_\der$ with L-parameters actually is a requirement,
it is an instance of property (k).
If we invoke that, the argument for \cite[Theorem 4.5.1]{Opd18} shows that the non-uniqueness
(when $Z(\mc G)$ is $K$-anisotropic) is the essentially the same as in the adjoint case.
That is, the parametrization is unique up to twists by the image of $X_\Wr (G_\ad) \cong 
Z({G_\ad}^\vee)^\Fr$ in ${}^L G$, which is just $X_\Wr (G)$.

Finally we consider the case where $\mc G$ is reductive and the maximal $K$-split central 
torus $Z(\mc G)_s$ is nontrivial. Then $G / Z(G)_s = (\mc G / Z(\mc G)_s)(K)$ does have
$K$-anisotropic centre. The Langlands correspondence for $\Irr_\unip (G)$ is deduced from 
that for $\Irr_\unip (G / Z(G)_s)$, see \cite[\S 15]{FOS} and \cite[p.35]{Opd18}. 
What happens for $Z(G)_s$ is determined by the natural LLC for tori, which we have included 
as a requirement via property (c). As noted on \cite[p. 38]{FOS}, this renders a LLC for 
$\Irr_\unip (G)$ precisely as canonical as for $\Irr_\unip (G/Z(G)_s)$. In view of the cases 
considered above, the only non-uniqueness comes from twisting by $X_\Wr (G / Z(G)_s)$.
\end{proof}

Next we recall some results from \cite{SolFunct} about the behaviour of unipotent 
representations and enhanced L-parameters under isogenies of reductive groups. We will 
formulate them for quotient maps, because we will only need them for such isogenies. 

Let $\mc Z$ be a central $K$-subgroup of $\mc G$ and consider the quotient map
\[
\eta : \mc G \to \mc G' := \mc G / \mc Z . 
\]
The dual homomorphism $\eta^\vee : G'^\vee \to G^\vee$ gives rise to maps
\[
{}^L \eta : {}^L G' \to {}^L G \quad \text{and} \quad 
\Phi (\eta) : \Phi (G') \to \Phi (G) .
\]
For $\phi' \in \Phi (G')$ and $\phi = \Phi (\eta) \phi' \in \Phi (G)$, $\mc A_{\phi'}$ is
a normal subgroup of $\mc A_\phi$ and $\mc A_\phi / \mc A_{\phi'}$ is abelian
\cite[Lemma 4.1]{SolFunct}.

The map between groups of $K$-rational points $\eta : G \to G'$ need not be surjective,
but in any case its cokernel is compact and commutative. This implies that the pullback
functor
\[
\eta^* : \Rep (G') \to \Rep (G)
\]
preserves finite length and complete reducibility \cite{Sil}. It is easily seen, for
instance from \cite[Proposition 7.2]{SolFunct}, that $\eta^*$ maps one Bernstein block
$\Rep (G')^{\mf s'}$ into a direct sum of finitely many Bernstein blocks $\Rep (G)^{\mf s}$.

\begin{thm}\label{thm:3.1} 
\textup{\cite[Theorem 3 and Lemma 7.3]{SolFunct}}\\
Let $\mc G$ be a connected reductive $K$-group which splits over an unramified extension.
Let $(\phi',\rho') \in \Phi_{\nr,e}(G')$ and let $\pi (\phi',\rho') \in \Irr (G')$
be associated to it in Theorem \ref{thm:1.1}. Then, with $\phi = \Phi (\eta) \phi'$:
\[
\eta^* \pi (\phi',\rho') = 
\bigoplus_{\rho \in \Irr (\mc A_\phi)} \hspace{-3mm} \Hom_{\mc A_\phi} 
\Big( \mr{ind}_{\mc A_{\phi'}}^{\mc A_\phi} \rho', \rho \Big) \otimes \pi (\phi,\rho) =
\bigoplus_{\rho \in \Irr (\mc A_\phi)} \hspace{-3mm} 
\Hom_{\mc A_{\phi'}} (\rho', \rho) \otimes \pi (\phi,\rho).
\]
\end{thm}

Let us work out a few more features of this result.

\begin{lem}\label{lem:3.2}
\enuma{
\item All irreducible constituents of the $G$-representation $\eta^* \pi (\phi',\rho')$
have the same Plancherel density and appear with the same multiplicity.
This multiplicity is one if $\pi (\phi',\rho')$ is supercuspidal.
\item All $\rho \in \Irr (\mc A_\phi)$ with $\Hom_{\mc A_{\phi'}} (\rho', \rho) \neq 0$
have the same dimension. 
\item For any such $\rho$, the length of the $G$-representation $\eta^* \pi (\phi',\rho')$ is 
\[
\dim (\rho') [\mc A_\phi : \mc A_{\phi'}] \dim (\rho)^{-1}.
\]
}
\end{lem}
\begin{proof}
(a) We abbreviate $\pi' = \pi (\phi',\rho')$. Since this $G'$-representation is irreducible,
all irreducible subrepresentations of $\eta^* (\pi')$ are equivalent under the action of
$G'$ on $\Irr (G)$. Conjugation with $g' \in G'$ defines a unimodular automorphism of $G$,
so $\mr{Ad}(g')^*$ preserves the Plancherel density on $\Irr (G)$.

Similarly, all isotypic components of $\eta^* (\pi')$ are $G'$-associate. As already shown
in \cite[Lemma 2.1]{GeKn}, this implies that every irreducible constituent of $\eta^* (\pi')$
appears with the same multiplicity. By \cite[Lemma 7.1]{SolFunct} this multiplicity is
one if $\pi'$ is supercuspidal.\\
(b) We briefly recall how to construct irreducible representations of $\mc A_\phi$ that contain
$\rho'$. Let $(\mc A_\phi )_{\rho'}$ be the stabilizer of $\rho'$ in $\mc A_\phi$ (with
respect to the action of $\mc A_\phi$ on $\Irr (\mc A_{\phi'})$ coming from conjugation).
The projective action of $(\mc A_\phi )_{\rho'}$ on $V_{\rho'}$ gives rise to a 2-cocycle
$\kappa_{\rho'}$ and a twisted group algebra $\C [ (\mc A_\phi )_{\rho'},\kappa_{\rho'}]$.
Clifford theory (in the version \cite[Proposition 1.1]{AMS1}) says that: 
\begin{itemize}
\item for every $(\tau,V_\tau) \in \Irr \big (\C [ (\mc A_\phi )_{\rho'},\kappa_{\rho'}] \big)$,
$\tau \ltimes \rho := \mr{ind}_{(\mc A_\phi )_{\rho'}}^{\mc A_\phi} (V_\tau \otimes
V_{\rho'})$ is an irreducible $\mc A_\phi$-representation containing $\rho'$;
\item every irreducible $\mc A_\phi$-representation containing $\rho'$ is of the form
$\tau \ltimes \rho'$.
\end{itemize}
For $\rho = \tau \ltimes \rho'$ we see that 
\[
\Hom_{\mc A_{\phi'}}(\rho',\rho) = \Hom_{\mc A_{\phi'}} \big( V_{\rho'},
\mr{ind}_{(\mc A_\phi )_{\rho'}}^{\mc A_\phi} (V_\tau \otimes V_{\rho'}) \big) \cong
\Hom_{\mc A_{\phi'} } (V_{\rho'}, V_\tau \otimes V_{\rho'}) \cong V_\tau.
\]
We can compute the dimension of $\rho = \tau \ltimes \rho'$ in these terms:
\begin{equation}\label{eq:3.6}
\dim (\rho) = [\mc A_\phi : (\mc A_\phi )_{\rho'}] \dim (V_\tau) \dim (V_{\rho'}) =
[\mc A_\phi : (\mc A_\phi )_{\rho'}] \dim (\rho') \dim \Hom_{\mc A_{\phi'}}(\rho',\rho) .
\end{equation}
By Theorem \ref{thm:3.1}
\[
\Hom_G (\pi (\phi,\rho), \eta^* (\pi')) \cong \Hom_{\mc A_{\phi'}}(\rho',\rho) .
\]
By part (a) this space is independent of $\rho$ (as long as it is nonzero). With \eqref{eq:3.6}
we conclude that dim$(\rho)$ is the same for all such $\rho$.\\
(c) By Frobenius reciprocity 
\[
\Hom_{\mc A_\phi} \Big( \mr{ind}_{\mc A_{\phi'}}^{\mc A_\phi} \rho', \rho \Big) 
\cong \Hom_{\mc A_{\phi'}}(\rho',\rho) .
\]
Hence $\mr{ind}_{\mc A_{\phi'}}^{\mc A_\phi} \rho'$ is a direct sum of irreducible 
subrepresentations of common dimension $\dim (\rho)$. Then its length is
\[
\dim \big( \mr{ind}_{\mc A_{\phi'}}^{\mc A_\phi} \rho' \big) \dim (\rho)^{-1} =
\dim (\rho') [\mc A_\phi : \mc A_{\phi'}] \dim (\rho)^{-1}.
\]
By Theorem \ref{thm:3.1} that is also the length of $\eta^* (\pi')$.
\end{proof}

\section{Affine Hecke algebras}
\label{sec:AHA}

As before, $\mc G$ denotes a connected reductive $K$-group which splits over an 
unramified extension. Let $\mc G_\ad = \mc G / Z(\mc G)$ be its adjoint group.
We intend to investigate the behaviour of the formal degrees with respect to 
the quotient map $\eta : \mc G \to \mc G_\ad$. As preparation, we consider the analogous
question for the affine Hecke algebras from Section \ref{sec:background}.

This means that we focus on one Bernstein component $\Rep (G)_{(\hat P_{\mf f},\hat \sigma)}$
for $G$ and one Bernstein component $\Rep (G_\ad)_{(\hat P_{\mf f_\ad},\hat \sigma_\ad)}$ for 
$G_\ad$, such that the pullback of the latter has nonzero components in the former. 
As already noted in \cite[\S 13]{FOS} and \cite[\S 3.3]{SolLLCunip}, we may assume that
$\mf f_\ad = \mf f$ and that underlying cuspidal unipotent representations 
$\sigma$ and $\sigma_\ad$ are essentially the same. That is, they are defined on the same 
vector space $V_\sigma$ and $\sigma$ is the pullback of $\sigma_\ad$ via the natural map 
$\overline{\mc G_f^\circ}(k) \to \overline{\mc G_{\ad,f}^\circ}(k)$. More precisely, we may
even assume that $\hat \sigma$ is the pullback of $\hat \sigma_\ad$ along 
$\eta : \hat P_{\mf f} \to \hat P_{\mf f,\ad}$.

In this setting $\eta$ induces an inclusion
\begin{equation}\label{eq:3.12}
\eta_{\mc H} : \mc H (G,\hat{P}_{\mf f},\hat \sigma) \to
\mc H (G_\ad,\hat{P}_{\ad,\mf f},\hat \sigma_\ad) ,
\end{equation}
which we need to analyse in more detail. Let $X_{\mf f,\ad}$ denote the lattice $X_{\mf f}$ 
for $G_\ad$. From \cite[Proposition 3.1 and Theorem 3.3.b]{SolLLCunip} we see that $X_{\mf f}$ 
can be regarded as a sublattice of $X_{\mf f,\ad}$, and that
\begin{equation}\label{eq:3.9}
X_{\mf f,\ad} / X_{\mf f} \cong 
\big( \Omega_{G_\ad,\mf f} \big/ \Omega_{G_\ad,\mf f,\tor} \big) \big/ 
\big( \Omega_{G,\mf f} \big/ \Omega_{G,\mf f,\tor} \big) .
\end{equation}
To make sense of the right hand side, we remark that by construction the natural map 
$\Omega_{G,\mf f} \to \Omega_{G_\ad,\mf f}$ is injective. The group \eqref{eq:3.9} 
is finite if and only if $Z(G)$ is compact.

We recall from \cite[\S 1.20]{LusUni1} and \cite[(42)]{SolLLCunip} that
$\Omega_{G_\ad,\mf f} / \Omega_{G_\ad,\mf f,\tor}$ acts on 
$\mc H (G,\hat{P}_{\mf f},\hat \sigma)$ by algebra automorphisms, and that
\begin{equation}\label{eq:3.10}
\mc H (G,\hat{P}_{\mf f},\hat \sigma) \cong 
\mc H_\af (G,P_{\mf f},\sigma) \rtimes \Omega_{G,\mf f} / \Omega_{G,\mf f,\tor} .
\end{equation}
By \cite[Lemma 3.5]{SolLLCunip} $\mc H_\af (G,P_{\mf f},\sigma)$ and all the data for that
algebra are the same for $G$ and for $G_\ad$. So the difference between \eqref{eq:3.10}
and its analogue for $G_\ad$ lies only in the finite group $\Omega_{G,\mf f} / 
\Omega_{G,\mf f,\tor}$. The inclusion \eqref{eq:3.12} is the identity on 
$\mc H_\af (G,P_{\mf f},\sigma)$.

Let $\tau$ and $\tau_\ad$ denote the normalized traces of the affine Hecke algebras $\mc H 
(G,\hat{P}_{\mf f},\hat \sigma)$ and $\mc H (G_\ad,\hat{P}_{\ad,\mf f},\hat \sigma_\ad)$.
By \eqref{eq:1.10} and \eqref{eq:1.7}
\begin{equation}\label{eq:3.8}
\frac{\tau (N_e)}{\tau_\ad (N_e)} = 
\frac{|\Omega_{G_\ad,\mf f,\tor}|}{|\Omega_{G,\mf f,\tor}|} .
\end{equation}
Both \eqref{eq:3.9} and \eqref{eq:3.8} contribute to the difference between the Plancherel
measures for $\mc H (G,\hat{P}_{\mf f},\hat \sigma)$ and for
$\mc H (G_\ad,\hat{P}_{\ad,\mf f},\hat \sigma_\ad)$. For the latter that is clear, for
the former we compute the effect below.

We abbreviate $A = \Omega_{G_\ad,\mf f} / \Omega_{G_\ad,\mf f,\tor}$, 
$\mc H_\ad = \mc H_\af (G,P_{\mf f},\sigma) \rtimes A$ and 
\[
\mc H = \mc H_\af (G,P_{\mf f},\sigma) \rtimes \Omega_{G,\mf f} / \Omega_{G,\mf f,\tor}.
\]
Since the abelian group $A$ acts on $\mc H_\af (G,P_{\mf f},\sigma)$ and (trivially) on
$\Omega_{G,\mf f} / \Omega_{G,\mf f,\tor}$, \eqref{eq:3.10} shows that it acts on $\mc H$
by algebra automorphisms. 

\begin{lem}\label{lem:4.4}
Let $V$ be any irreducible $\mc H_\ad$-module. All the constituents of $\eta_{\mc H}^* (V)$ 
have the same dimension and the same Plancherel density, and they appear with the same multiplicity.
\end{lem}
\begin{proof}
If $V_{\mc H}$ is any irreducible submodule of $\eta_{\mc H}^* (V)$, \eqref{eq:3.10} shows that 
\begin{equation}\label{eq:3.24}
V = \sum\nolimits_{\omega \in A} N_\omega \cdot V_{\mc H} .
\end{equation}
As $N_\omega$ normalizes the subalgebra $\mc H$ of $\mc H_\ad$, each $N_\omega \cdot V_{\mc H}$
is an irreducible $\mc H$-submodule of $V$. Consequently
\begin{equation}\label{eq:3.25}
\text{every constituent of } \eta_{\mc H}^* (V) \text{ is isomorphic to }
\mr{Ad}(N_\omega)^* V_{\mc H} \text{ for some } \omega \in A .
\end{equation}
Taking into account that conjugation by $N_\omega$ is a trace-preserving automorphism
of $\mc H$, \eqref{eq:3.25} shows that all the constituents of $\eta_{\mc H}^* (V)$ have the same 
dimension and the Plancherel density.
Further, we see from \eqref{eq:3.24} that any two $\mc H$-isotypic submodules of $V$ are in 
bijection, via multiplication with a suitable $N_\omega$. Hence all constituents of 
$\eta_{\mc H}^* (V)$ appear with the same multiplicity in that $\mc H$-module.
\end{proof}

This rough analysis of $\eta_{\mc H}^*$ does not yet suffice, we need more precise results
from Clifford theory. Write 
\[
C = \Irr ( X_{\mf f,\ad} / X_{\mf f} ) .
\]
By \eqref{eq:3.9}, $C$ can also be regarded as the character group of
\[
\big( \Omega_{G_\ad,\mf f} \big/ \Omega_{G_\ad,\mf f,\tor} \big) \big/ 
\big( \Omega_{G,\mf f} \big/ \Omega_{G,\mf f,\tor} \big) .
\]
Using \eqref{eq:3.10}, every $c \in C$ determines an automorphism of $\mc H_\ad$, namely
\[
c \cdot (h \otimes N_\omega) = h \otimes c(\omega) N_\omega 
\qquad h \in \mc H_\af (G,P_{\mf f},\sigma), \omega \in A .
\]
We note that $\mc H_\ad^C = \mc H$.

The restriction of modules from $\mc H_\ad$ to $\mc H_\ad^C$ was investigated in 
\cite[Appendix]{RaRa}. Let $C_V$ be the stabilizer (in $C$) of the isomorphism class of
$V \in \Irr (\mc H_\ad)$. For every $c \in C$ there exists an isomorphism of 
$\mc H$-modules
\[
i_c : V \to c^* V .
\]
By Schur's lemma $i_c$ is unique up to scalars, and thus the $i_c$ furnish a projective 
action of $C$ on $V$. Our particular situation is favourable because the action of $C$ on 
$\mc H_\ad$ is free, in the sense that it acts freely on a vector space basis). This can be 
exploited to analyse the intertwining operators $i_c$. 

\begin{lem}\label{lem:4.5}
The group $C_V$ acts linearly on $V$, by $\mc H$-module automorphisms. 
\end{lem}
\begin{proof}
We normalize $i_c$ by requiring that it restricts to the identity on
$V_{\mc H}$. For any $\omega \in A, c \in C_V$ and $v \in V_{\mc H}$ we have
\begin{equation}\label{eq:3.26}
i_c (N_\omega \cdot v) = c(N_\omega) \cdot i_c (v) = c(\omega) N_\omega \cdot v .
\end{equation}
In view of \eqref{eq:3.24}, this formula determines $i_c$ completely. In particular
$i_c \circ i_{c'} = i_{c c'}$ for all $c, c' \in C_V$.
\end{proof} 

In the remainder of this section we assume that $\mc G$ is semisimple, so that 
\eqref{eq:3.9} and $C$ are finite.
By \cite[Theorem A.13]{RaRa} the action from Lemma \ref{lem:4.5} gives rise to 
an isomorphism of $\mc H \times \C [C_V]$-modules
\begin{equation}\label{eq:3.27}
V \cong \bigoplus\nolimits_{E \in \Irr (C_V)} V_E \otimes E .
\end{equation}

\begin{lem}\label{lem:4.6}
For every $E \in \Irr (C_V)$ the $\mc H$-module $V_E = \Hom_{C_V} (E,V)$
is irreducible and appears with multiplicity one in $\eta_{\mc H}^* (V)$.
\end{lem}
\begin{proof}
By \cite[Theorem A.13]{RaRa} the $\mc H$-module $V_E$ is either zero or irreducible. 
Let $\Omega \subset A$ be a set of representatives of $A / \cap_{c \in C} \ker (c|_A)$,
so that $\Irr (C_V)$ is naturally in bijection with $\Omega$.
From \eqref{eq:3.24} and \eqref{eq:3.26} we see that there is a linear bijection
\[
\C \Omega \otimes \sum_{\omega \in \cap_{c \in C} \ker (c|_A)} N_\omega \cdot V_{\mc H} 
\to V : a \otimes v \to N_a \cdot v .
\]
Hence every $E \in \Irr (C_V) \cong \Omega$ appears nontrivially in the decomposition 
\eqref{eq:3.27}. The multiplicity of $V_E$ in $V$ is $\dim (E)$, which is one because
$C_V$ is abelian.
\end{proof}

For another irreducible $\mc H_\ad$-module $V'$, \cite[Theorem A.13]{RaRa} shows how
the restrictions to $\mc H$ compare:
\begin{equation}\label{eq:3.28}
\eta_{\mc H}^* (V') \left\{
\begin{array}{ll}
\cong \eta_{\mc H}^* (V) & \text{if } V' \cong c^* V \text{ for some } c \in C \\
\text{has no constituents in common with } \eta_{\mc H}^* (V) & \text{otherwise} .
\end{array}
\right.
\end{equation}
From here on we assume that $V$ is discrete series. Casselman's criterion for discrete series 
representations \cite[Lemma 2.22]{Opd-Sp} entails that $\eta_{\mc H}^* \delta'$ is direct 
sum of finitely many irreducible discrete series representations of $\mc H$. 

Endow $\mc H_\af$ and $\mc H_\ad$ with the trace $\tau'$ so that $\tau' (N_e) = 1$. 
We indicate the formal degree with respect to this renormalized trace by $\fdeg'$. 

\begin{thm}\label{thm:3.4}
Let $\mc G$ be a semisimple $K$-group which splits over an unramified extension.
Let $V$ be an irreducible discrete series representation of $\mc H_\af (G,P_{\mf f},\sigma) 
\rtimes \Omega_{G_\ad,\mf f} / \Omega_{G_\ad,\mf f,\tor}$ and let $\eta_{\mc H}^* V$ 
be its pullback to $\mc H_\af (G,P_{\mf f},\sigma) \rtimes \Omega_{G,\mf f} / 
\Omega_{G,\mf f,\tor}$ via \eqref{eq:3.12} and \eqref{eq:3.10}. Then
\[
\frac{\fdeg' (\eta_{\mc H}^* V)}{\fdeg' (V)} = |C| = \Big[ \frac{\Omega_{G_\ad,\mf f}}{
\Omega_{G_\ad,\mf f,\tor}} : \frac{\Omega_{G,\mf f}}{\Omega_{G,\mf f,\tor}} \Big] .
\]
For any irreducible constituent $V_E$ of $\eta_{\mc H}^* (V)$:
\[
\fdeg' (V_E) = [C : C_V] \, \fdeg' (V) . 
\]
Here $|C_V|$ equals the length of $\eta_{\mc H}^* (V)$.
\end{thm}
\begin{proof}
Let $C_r^* (\mc H)$ be the $C^*$-completion of $\mc H$, as in \cite[Definition 2.4]{Opd-Sp}. 
As $V$ is discrete series, we know from \cite[\S 6.4]{Opd-Sp} that $C_r^* (\mc H_\ad)$ 
contains a central idempotent $e_{V}$ such that 
\[
e_{V} C_r^* (\mc H_\ad) \cong \mr{End}_\C (V) .
\]
Then by definition
\begin{equation}\label{eq:3.29}
\tau' (e_{V}) = \dim (V) \fdeg' (V) .
\end{equation}
The $C$-orbit of $V$ in $\Irr (\mc H_\ad)$ has precisely $[C : C_{V}]$ elements,
and these are all discrete series. The central idempotent 
\[
e_{C,V} := \sum\nolimits_{c \in C / C_{V}} c \cdot e_{V} 
\]
lies in $C_r^* (\mc H_\ad)^C = C_r^* (\mc H)$ and
\[
e_{C,V} C_r^* (\mc H_\ad) \cong 
\bigoplus\nolimits_{c \in C / C_{V}} \mr{End}_\C (c^* V) .
\]
Since the action of $C$ preserves $\tau'$, we obtain
\[
\tau' (e_{C,V}) = \sum\nolimits_{c \in C / C_{V}} \dim (c^* V)
\fdeg' (c^* V) = [C : C_{V}] \dim (V) \fdeg' (V) . 
\]
With \eqref{eq:3.27} and Lemma \ref{lem:4.6} this can be expanded as
\begin{equation}\label{eq:3.30}
\begin{aligned}
\tau' (e_{C,V}) & = [C : C_{V}] \, \fdeg' (V) 
\sum\nolimits_{E \in \Irr (C_{V})} \dim (V_E) \dim (E) \\
& = [C : C_{V}] \, \fdeg' (V) |C_{V}| \dim (V_E) = 
|C| \, \fdeg' (V) \dim (V_E) .
\end{aligned}
\end{equation}
From \eqref{eq:3.27} and \eqref{eq:3.28} we see that
\[
e_{C,V} C_r^* (\mc H) \cong \bigoplus\nolimits_{E \in \Irr (C_{V})}
\mr{End}_\C (V_E) .
\]
Considering $\tau'$ as trace for $\mc H$, using Lemma \ref{lem:4.4}
and the commutativity of $C_V$, we find
\begin{equation}\label{eq:3.31}
\tau' (e_{C,V}) = \sum\nolimits_{E \in \Irr (C_{V})} 
\dim (V_E) \fdeg' (V_E) = |C_{V}| \dim (V_E) \fdeg' (V_E) .
\end{equation}
Now we compare \eqref{eq:3.30} and \eqref{eq:3.31}, for any constituent
$V_E$ of $\eta_{\mc H}^* (V)$, and we find the desired formula for $\fdeg' (V_E)$.

From that formula and \eqref{eq:3.27} we deduce:
\begin{multline*}
\!\!\! \fdeg' (\eta_{\mc H}^* V) = \sum_{E \in \Irr (C_{V})} \dim (E) 
\fdeg (V_E) = |C_{V}| [C : C_{V}] \, \fdeg (V) = |C| \, \fdeg (V) . 
\end{multline*}
For the interpretation of $|C_V|$ we refer to Lemma \ref{lem:4.6}.
\end{proof}

We note two direct consequences of Theorem \ref{thm:3.4} and \eqref{eq:3.8}:
\begin{align}
& \frac{\fdeg (\eta_{\mc H}^* V)}{\fdeg (V)} = |C| \frac{|\Omega_{G_\ad,\mf f,\tor}|}{
|\Omega_{G,\mf f,\tor}|} = [\Omega_{G_\ad,\mf f}  : \Omega_{G,\mf f}] , \\
& \label{eq:3.50} \text{length of } \eta_{\mc H}^* (V) \cdot \fdeg (V_E) = 
[\Omega_{G_\ad,\mf f}  : \Omega_{G,\mf f}] \, \fdeg (V) . 
\end{align}

\section{Pullback of representations}
\label{sec:pullback}

We would like to apply Lemma \ref{lem:3.2} to \eqref{eq:3.50},
but to do so we first have to find the relation between the length of $\eta_{\mc H}^* (V)$
and the length of the pullback of the associated $G_\ad$-representation. This involves 
the number of $G$-orbits of facets and the number of Bernstein components obtained from 
$(\hat P_{\mf f,\ad},\hat \sigma_\ad)$ under pullback along $\eta$. 

For a facet $\mf f$ of $\mc B (\mc G,K)$, let $\Rep (G)_{\mf f}$ be the sum of the 
subcategories $\Rep (G)_{(P_{\mf f},\sigma)}$,
where $\sigma$ runs over the irreducible representations of $P_{\mf f}$ inflated from
cuspidal representations of $\overline{\mc G_{\mf f}^\circ}(k)$. In Section 
\ref{sec:background} we saw that this is a direct sum of finitely many Bernstein blocks, 
which by \cite[Corollary 3.10]{Mor2} all come from supercuspidal Bernstein components 
of the same Levi subgroup of $\mc G$. By \eqref{eq:1.11}
\begin{equation}\label{eq:3.19}
\begin{array}{ll}
\Rep (G)_{\mf f} = \Rep (G)_{\mf f'} & 
\text{if } g \mf f' = \mf f \text{ for some } g \in G \\
\Rep (G)_{\mf f} \cap \Rep (G)_{\mf f'} = \{ 0 \} & \text{otherwise.}
\end{array}
\end{equation}
Let $\eta^* (\Rep (G_\ad)_{\mf f})$ the pullback of $\Rep (G_\ad)_{\mf f}$ along
$\eta : G \to G_\ad$.

\begin{lem}\label{lem:3.8}
Let $\mc G$ be a reductive $K$-group which splits over an unramified extension.
The number of different subcategories $\Rep (G)_{\mf f'}$ involved nontrivially in
$\eta^* (\Rep (G_\ad)_{\mf f})$ is 
$|\Omega_{G_\ad} \cdot \mf f| \, |\Omega_G \cdot \mf f|^{-1}$.
\end{lem}
\begin{proof}
Since $\overline{C_0}$ contains a fundamental domain for the $G$-action on $\mc B (\mc G,K)$,
it suffices by \eqref{eq:3.19} to consider facets in $\overline{C_0}$. Since the kernel of 
$\kappa_G : G \to \Omega_G$ acts type-preservingly on $\mc B (\mc G,K)$, the $G$-association classes 
of facets in $\overline{C_0}$ are precisely the $\Omega_G$-orbits of facets in $\overline{C_0}$.

Since $\Omega_G$ embeds in the abelian group $\Omega_{G_\ad}$, all $\Omega_G$-orbits in
$\Omega_{G_\ad} \cdot \mf f$ have the same length. The number of such $\Omega_G$-orbits is
$|\Omega_{G_\ad} \cdot \mf f| \, |\Omega_G \cdot \mf f|^{-1}$.

It is clear from the definitions that $\eta^* (\Rep (G_\ad)_{\mf f})$ has nonzero parts in the 
$\Rep (G)_{\mf f'}$ with $\mf f' \in \Omega_{G_\ad} \cdot \mf f$, and maps to zero in all
other subcategories $\Rep (G)_{\mf f''}$. In view of \eqref{eq:3.19}, the number of different
$\Rep (G)_{\mf f'}$ involved here equals the number of $\Omega_G$-orbits in 
$\Omega_{G_\ad} \cdot \mf f$.
\end{proof}

There exists a Levi subgroup $M = \mc M (K)$ such
that $(\hat{P}_{\mf f} \cap M, \hat \sigma)$ is a type for a supercuspidal Bernstein
block $\mf s_M$ of Rep$(M)$, covered by $(\hat{P}_{\mf f},\hat \sigma)$ 
\cite[Corollary 3.10]{Mor2}. We will often denote objects associated to $M$ with an
additional subscript, e.g. $P_{M,\mf f} = P_{\mf f} \cap M$. We note that, by
\cite[Theorem 2.1]{Mor2}, $\mf f$ is contained in a minimal facet $\mf f_M$ of
$\mc B (\mc M_\ad,K)$ and $P_{M,\mf f} = P_{\mf f_M}$ is a maximal parahoric subgroup of $M$. 

Recall from \cite[\S 1.2]{Tit} that the apartment $\mh A$ of $\mc B (\mc G,K)$ associated to 
$\mc S$ admits a canonical decomposition 
\begin{equation}\label{eq:3.18}
\mh A = \mh A_{M_\ad} \times X_* (Z(\mc M)_s) \otimes_\Z \R,
\end{equation}
where $\mh A_{M_\ad}$ is the apartment of $\mc B (\mc M_\ad ,K) = 
\mc B (\mc M / Z(\mc M)_s ,K)$ associated to $\mc S / Z(\mc M)_s$. With \eqref{eq:3.18} we can 
express $\mf f$ as $\mf f_M \times \mf f^M$, where  $\mf f_M$ is a vertex of $\mc B (\mc M_\ad,K)$ 
and $\mf f^M$ is an open subset of $X_* (Z(\mc M)_s) \otimes_\Z \R$. Now
\begin{equation}\label{eq:3.17}
\overline{\mf f} = \big( \overline{\mf f_M} \times X_* (Z(\mc M)_s) \otimes_\Z \R \big) 
\cap \overline{C_0} ,
\end{equation}
so that $\mf f$ and $\mf f_M$ determine each other.

\begin{lem}\label{lem:3.7} 
Let $\mf f' \in N_{G_\ad}(M) \cdot \mf f$ such that $\Rep (M)_{\mf f_M} \neq
\Rep (M)_{\mf f'_M}$. Then $\Rep (G)_{\mf f} \neq \Rep (G)_{\mf f'}$.
\end{lem}
\begin{proof}
Suppose that $\Rep (G)_{\mf f} = \Rep (G)_{\mf f'}$. Then any inertial equivalence class
$\mf s = [M,\pi_M]_G$ with $\Rep (G)_{\mf s} \subset \Rep (G)_{\mf f}$ equals an inertial
equivalence class $\mf s'$ with $\Rep (G)_{\mf s'} \subset \Rep (G)_{\mf f'}$. By assumption 
$\mf f'$ also admits a decomposition \eqref{eq:3.17}, with the same $M$. Hence we may 
assume that $\mf s' = [M,\pi'_M]_G$.

This means that there exist $g \in N_G (M)$ and $\chi \in X_\nr (M)$ such that 
$\pi'_M = \mr{Ad}(g)^* (\pi_M \otimes \chi)$. Since $N_G (M)/M$ only depends on $G$ up to 
isogenies, we may assume that $g$ lies in the image of $G_\Sc \to G$. In particular $g$ lies 
in the kernel of $\kappa_G$ and acts type-preservingly on $\mc B (\mc G,K)$. 

By \cite[\S 1]{LusUni1} or \cite[(1.18)]{FOS} we can write 
$\pi_M \otimes \chi = \mr{ind}_{N_M (P_{\mf f_M})}^M (\tilde \sigma)$ for some extension 
$\tilde \sigma$ of $\sigma$ to $N_M (P_{\mf f_M})$. Then
\[
\mr{Ad}(g)^* (\pi_M \otimes \chi) \cong 
\mr{ind}_{g N_M (P_{\mf f_M}) g^{-1}}^M (\mr{Ad}(g)^* \tilde \sigma) =
\mr{ind}_{N_M (P_{g \cdot \mf f_M})}^M (\mr{Ad}(g)^* \tilde \sigma) .
\]
With \eqref{eq:3.19} this implies $m g \cdot \mf f_M = \mf f'_M$ for some $m \in M$. Then 
\eqref{eq:3.17} shows that also $m g \cdot \mf f = \mf f$. Since $\mf f \cup \mf f' \subset
\overline{C_0}$ and $g \in \ker (\kappa_G)$, it follows that
$\kappa_G (m) \mf f = \mf f'$. From \eqref{eq:3.17} and the naturality of the Kottwitz
homomorphism we deduce that $\kappa_M (m) \mf f_M = \mf f'_M$. By \eqref{eq:3.19} this
contradicts the assumption of the lemma.
\end{proof}

We write $\mc M_{AD} = \mc M / Z(\mc G)$ and we restrict $\eta$ to 
$\eta_M : \mc M \to \mc M_{AD}$. Let $\mc P$ be a parabolic $K$-subgroup of $\mc G$ with 
Levi factor $\mc M$ and put $\mc P_{AD} = \mc P / Z(\mc G)$. The normalized parabolic
induction functors form a commutative diagram
\begin{equation}\label{eq:5.1}
\begin{array}{ccc}
\Rep (G_\ad) & \xrightarrow{\eta^*} & \Rep (G) \\
\uparrow I_{P_{AD}}^{G_\ad} & & \uparrow I_P^G \\
\Rep (M_{AD}) & \xrightarrow{\eta_M^*} & \Rep (M) 
\end{array} . 
\end{equation}

\begin{lem}\label{lem:5.1}
Let $\pi_\ad \in \Rep (G_\ad)_{(\hat P_{\ad,\mf f},\hat \sigma_\ad)}$ and let 
$\pi_{\ad,\mc H}$ be the associated module of $\mc H (G_\ad,\hat P_{\ad,\mf f},\hat \sigma_\ad)$.
Then the length of $\eta^* (\pi_\ad) \in \Rep (G)$ equals 
$|\Omega_{G_\ad} \cdot \mf f| \, |\Omega_G \cdot \mf f|^{-1}$ times the length of
$\eta_{\mc H}^* (\pi_{\ad,\mc H}) \in \Rep (\mc H (G,\hat P_{\mf f},\hat \sigma))$.
\end{lem}
\begin{proof}
In every subcategory $\Rep (G)_{g \mf f}$ with $g \in G_\ad$, $\eta^* (\pi_\ad)$ has a nonzero
component. These components are associated by the automorphisms Ad$(g)$ of $G$, so they all
have the same length. Lemma \ref{lem:3.8} tells us that the number of such components is
$|\Omega_{G_\ad} \cdot \mf f| \, |\Omega_G \cdot \mf f|^{-1}$.

Hence it suffices to consider the projection $\pi_{\mf f}$ of $\eta^* (\pi_\ad)$ to 
$\Rep (G)_{\mf f}$, and we have to show that its length equals that of 
$\eta_{\mc H}^* (\pi_{\ad,\mc H})$. From the commutative diagram \eqref{eq:5.1} we see that
\begin{multline}\label{eq:5.2}
\Rep (G)_{\mf f} \cap \overline{ \eta^* \big(
\Rep (G_\ad)_{(\hat P_{\ad,\mf f},\hat \sigma_\ad)} \big)} \\
= \overline{I_P^G (\Rep (M)_{\mf f_M})} \cap \overline{I_P^G \big( \eta_M^* 
(\Rep (M_{AD})_{(\hat P_{\ad,\mf f_M},\hat \sigma_\ad )} )\big) } ,
\end{multline}
where the $\overline X$ indicates that we take the sum of all Bernstein components appearing
in $X$. It is known from \cite[(7.8)]{SolFunct} that $\eta_M^* \big( \Rep (M_{AD})_{
(\hat P_{\ad,\mf f_M},\hat \sigma_\ad )} \big)$ involves just one Bernstein component of 
$\Rep (M)_{\mf f'_M}$ for every facet $\mf f'_M \in M_{AD} \cdot \mf f$, and no others.
By Lemma \ref{lem:3.7} these different Bernstein components remain different upon 
parabolic induction to $G$. Hence we can identify the right hand side of \eqref{eq:5.2} as
\begin{equation}\label{eq:5.3}
\begin{split}
\overline{I_P^G \big( \text{projection of } \eta_M^* (\Rep (M_{AD})_{(\hat P_{\ad,\mf f_M},
\hat \sigma_\ad )}) \text{ to } \Rep (M)_{\mf f_M} ) \big)} \\
= \overline{I_P^G \big(\Rep (M)_{(\hat P_{\mf f_M},\hat \sigma)}\big)} =
\Rep (G)_{(\hat P_{\mf f},\hat \sigma)}  
\end{split}
\end{equation} 
From \eqref{eq:5.2} and \eqref{eq:5.3} we see that the projection $\pi_{\mf f}$ of 
$\eta^* (\pi_\ad)$ to $\Rep (G)_{\mf f}$ equals its projection to 
$\Rep (G)_{(\hat P_{\mf f},\hat \sigma)}$. The latter category is equivalent with 
$\Rep (\mc H (G,\hat P_{\mf f},\hat \sigma))$, via \eqref{eq:1.4}. Hence $\pi_{\mf f}$ maps
to $\eta_{\mc H}^* (\pi_{\ad,\mc H})$ by \eqref{eq:1.4}, and in particular these two
representations have the same length.
\end{proof}

\section{Computation of formal degrees}

As before, $\mc G$ denotes a connected reductive $K$-group which splits over an 
unramified extension. We will compute the formal degrees of square-integrable
unipotent $G$-representations in increasing generality.

By the choice of a Haar measure on $G$, we make $C_c^\infty (G)$ into a convolution
algebra, denoted $\mc H (G)$. The Plancherel theorem asserts that there exists a unique
Borel measure $\mu_{Pl}$ on $\Irr (G)$ such that
\[
f(1) = \int_{\Irr (G)} \mr{tr} \, \pi (f) \, \textup{d} \mu_{Pl}(\pi) 
\qquad \forall f \in \mc H (G) .
\]
The support of $\mu_{Pl}$ is precisely the collection $\Irr_{\mr{temp}}(G)$ of tempered
irreducible $G$-representations. 
For a selfadjoint idempotent $e \in \mc H (G)$ we write
\[
\Irr (G)^e = \{ (\pi,V_\pi) \in \Irr (G) : e V_\pi \neq 0 \} . 
\]
If it is nonzero, $e V_\pi$ is an irreducible representation of the Hilbert algebra 
$e \mc H (G) e$. Suppose in addition that $\dim (e V_\pi) = d \in \N$ for all 
$\pi \in \Irr (G)^e$. Then \cite[Theorem 2.3 and Proposition 2.1]{BHK} tell us that
\begin{equation}\label{eq:1.8}
\mu_{Pl}( \Irr (G)^e ) = e (1) d^{-1} .
\end{equation}
For an important special case, suppose that $(\sigma, V_\sigma)$ is an irreducible 
representation of a compact open subgroup $J$ of $G$ and that $\mr{ind}_J^G (\sigma)$ 
is irreducible. This forces that $Z(G)$ is compact, that $\mr{ind}_J^G (\sigma)$ is 
supercuspidal and that $e_\sigma \mc H (G) e_\sigma \cong \mr{End}_\C (V_\sigma)$. Applying 
\eqref{eq:1.8} to the associated central idempotent $e_\sigma \in \mc H (J)$, we find
\begin{equation}\label{eq:1.9}
\mu_{Pl} \big( \mr{ind}_J^G (\sigma) \big) = \frac{e_\sigma (1)}{\dim (\sigma)} = 
\frac{\dim (\sigma)}{\mr{vol}(J)} . 
\end{equation}
Recall that a $G$-representation $(\pi ,V_\pi)$ is square-integrable 
modulo centre if $Z(G)$ acts on $V_\pi$ by a unitary character and $V_\pi$ is 
square-integrable as representation of the derived group of $G$. Such a representation
has a $G$-invariant inner product and is completely reducible.
The formal degree of an irreducible square-integrable modulo centre $G$-representation 
is defined as the unique number fdeg$(\pi) \in \R_{>0}$ such that 
\begin{equation}\label{eq:1.2}
\int_{G / Z(G)_s} \langle \pi (g) v_1, v_2 \rangle \overline{\langle \pi (g) v_3, v_4 \rangle}
\, \fdeg (\pi) \textup{d} g = \langle v_1, v_3 \rangle \overline{\langle v_2, v_4 \rangle}  
\quad \text{for all } v_i \in V_\pi.
\end{equation}
When $\pi$ is actually square-integrable (which can happen only if $Z(G)$ is compact),
\eqref{eq:1.2} entails that the formal degree of $\pi$ is its mass with respect 
to the Plancherel measure $\mu_{Pl}$ on $\Irr (G)$ \cite[Proposition 18.8.5]{Dix}. 
The formal degree can be extended to finite length square-integrable modulo centre 
representations by additivity.

The above depends on the choice of a Haar measure, we need to make that explicit.
Fix an additive character $\psi : K \to \C^\times$ of order $-1$, that is, trivial on 
$\varpi_K \mf o_K$ but nontrivial on $\mf o_K$. We endow $G$ (and all other reductive
$p$-adic groups) with the Haar measure as in \cite{HII}. Since $\psi$ has order $-1$, 
this agrees with the Haar measure in \cite{GaGr} and in \cite{DeRe}.

\subsection{Adjoint groups} \
\label{par:adj}

It is known from \cite[Theorems 3.22 and 3.29]{Lus-Che} that the cuspidal unipotent 
representations of $\overline{\mc G_{\mf f}^\circ}(k)$ depend functorially on the finite 
field $k$. This means that $\sigma$ is part of a family of representations $\sigma_{k'}$ 
of $\overline{\mc G_{\mf f}^\circ}(k')$, one for every finite field $k'$ containing $k$.
Moreover, the dimension of $\sigma_{k'}$ is a particular polynomial in $q' = |k'|$, a
product of a rational number without factors $p$, a power of $q'$ and terms 
$(q'^{n} - 1)^{\pm 1}$ with $n \in \N$.

When $K'$ is an unramified extension of $K$, $\overline{\mc G_{\mf f}^\circ}(k')$ is
the finite reductive quotient of $\mc G_{\mf f}^\circ (\mf o_{K'})$, a parahoric subgroup
of $\mc G (K')$. As already remarked after \eqref{eq:1.3}, the volume of 
$\mc G_{\mf f}^\circ (K')$ is a rational function in $q'$, with the same kind of (rational)
factors as dim$(\sigma_{k'})$. This enables one to vary $q$, while keeping $\mf f,
\mc G_{\mf f}^\circ$ and $\sigma$ essentially constant.

A similar variation is possible for the affine Hecke algebra 
\begin{equation}\label{eq:2.1}
\mc H (G,\hat{P}_{\mf f},\hat \sigma ) = \mc H (X_{\mf f}, R_{\mf f}, q^{\mc N}) .
\end{equation}
There it means that the parameter function $q^{\mc N}$ can be replaced by
$q'^{\mc N} = (q^{\mc N})^{[k':k]}$, obtaining a new Hecke algebra
$\mc H (X_{\mf f}, R_{\mf f}, q'^{\mc N})$. Any discrete series representation $\delta$
of \eqref{eq:2.1} naturally gives rise to a discrete series representation $\delta'$
of $\mc H (X_{\mf f}, R_{\mf f}, q'^{\mc N})$ and conversely, see \cite[\S 5.2]{Opd-Sp} 
and \cite[Corollary 4.2.2]{SolAHA}.
In this way one can consider fdeg$(\delta)$ as a function of $q$. 

Further, unramified L-parameters $\phi$ can be made into functions of $q$. 
Replacing $\phi$ by a $G^\vee$-conjugate, we may assume that $\phi (\Fr) = t \Fr$  with
$t \in (T^\vee)^{\mb W_F,\circ}$. For $q' = |k'|$ one takes $\phi'$ with
$\phi' (\Fr) = t^{[k':k]} \Fr$ and $\phi' = \phi$ on $\mb I_K \times SL_2 (\C)$.
It is easily seen from the explicit formulas in \cite[\S 4]{GrRe} that this makes the 
L-functions, $\epsilon$-factors and $\gamma$-factors of $\phi$ into meromorphic
functions of $q$.

\begin{thm}\label{thm:2.1}
Suppose that $\mc G_\ad$ is simple and splits over an unramified extension.
Let $\pi \in \Irr (G_\ad)$ be square-integrable and unipotent.
\enuma{
\item The HII-conjecture \eqref{eq:HII1} holds in this setting, and more precisely
\[
\fdeg (\pi) = \pm \frac{\dim (\sigma) }
{\vol (P_{\mf f}) \, [\hat{P}_{\mf f} : P_{\mf f}] } d_{\mc H,\pi} m (q^{\mc N})^{(r)} =
\pm \frac{\dim (\rho_\pi)}{|S_{\phi_\pi}^\sharp|} \gamma (0,\mr{Ad}_{G^\vee} \circ \phi_\pi, \psi) .
\]
\item The expressions $\gamma (0,\mr{Ad}_{G^\vee} \circ \phi_\pi, \psi)$, $\dim(\sigma)$,
$\vol (P_{\mf f})$ and $m (q^{\mc N})^{(r)}$ are nonzero rational functions of $q$. 
Each of them is a product of a rational constant and factors of the form $q^{m/2}$ 
with $m \in \Z$ and $(q^n - 1)^{\pm 1}$ with $n \in \N$.
}
\end{thm}
\begin{proof}
By \cite[Theorem A.1 and Lemma A.3]{FOS}, the objects in Theorem 2.1 do not change under
Weil restriction for reductive groups, with respect to finite unramified extensions.
Hence we may assume that $\mc G_\ad$ is, in addition, absolutely simple. A first expression
for $\gamma (0,\mr{Ad}_{G^\vee} \circ \phi_\pi, \psi)$ was given in \cite[(38)]{Opd2}, we provide
the proof of that in the appendix (Theorem \ref{thm:A.2}). Then we can use the results from 
\cite{Opd2}, which were obtained by classification.

For (a) see \cite[Theorem 4.11]{Opd2}. For (b) see \cite[Proposition 2.5 and (38)]{Opd2}
and \cite[\S 4.2]{GrRe}.
\end{proof}

Exactly the same argument as for \cite[Proposition 12.2]{FOS} extends 
Theorem \ref{thm:2.1} to all adjoint groups.

\begin{cor}\label{cor:2.2}
Suppose that $\mc G_\ad$ is adjoint and splits over an unramified extension of $K$.
Then \eqref{eq:HII1} holds for all irreducible square-integrable unipotent $G$-representations.

The functions $\gamma (0,\mr{Ad}_{G^\vee} \circ \phi_\pi, \psi)$, $\dim(\sigma)$,
$\vol (P_{\mf f})$ and $m(q^{\mc N})^{(r)}$ for $G_\ad$ are the products of the corresponding
functions for the $K$-simple factors of $\mc G_\ad$.
\end{cor}

\subsection{Semisimple groups} \
\label{par:ss}

We note that by \cite[Propositions 2.2 and 2.7]{Tad} every (square-integrable) 
irreducible representation of $G$ appears in the pullback of an irreducible 
(square-integrable) representation of $G_\ad$. The constructions from
\cite{Tad} can be performed entirely in categories of unipotent representations
(for $G$ and $G_\ad$). Therefore the next result applies to all square-integrable
unipotent $G$-representations.

\begin{thm}\label{thm:3.10}
Let $\delta \in \Irr_\unip (G)$ and $\delta_\ad \in \Irr_\unip (G_\ad)$ be 
square-integrable, such that $\delta$ is a constituent of $\eta^* (\delta_\ad)$.
\enuma{
\item Their formal degrees, normalized as in \cite{HII}, are related as
\[
\frac{\fdeg (\eta^* (\delta_\ad))}{\fdeg (\delta_\ad)} =
\frac{\fdeg (\delta) \cdot \mathrm{length \: of} \, \eta (\delta_\ad)}{\fdeg (\delta_\ad)} = 
\frac{|\Omega_{G_\ad}|}{|\Omega_G|} .
\]
\item $\fdeg (\delta) = \pm \dim (\rho_\delta) \, \big| S_{\phi_\delta}^\sharp \big|^{-1}
\gamma (0,\mr{Ad}_{G^\vee} \circ \phi_\delta, \psi )$.
}
\end{thm}
\begin{proof}
(a) The first equality sign is a consequence of Lemma \ref{lem:3.2}.a.

Write $V = \Hom_{\hat P_{\mf f,\ad}} (\hat \sigma_\ad, \delta_\ad)$ and 
$V_E = \Hom_{\hat P_{\mf f,\ad}} (\hat \sigma_\ad, \delta)$. Recall from \cite{BHK} 
that $\delta$ and $V_E$ have the same formal degree, and similarly for $\delta_\ad$ and $V$.
With \eqref{eq:3.50} and Lemma \ref{lem:5.1} we compute
\begin{equation}\label{eq:5.4}
\frac{\fdeg (\delta)}{\fdeg (\delta_\ad)} = \frac{\fdeg (V_E)}{\fdeg (V)} = 
\frac{[\Omega_{G_\ad,\mf f}  : \Omega_{G,\mf f}]}{\text{ length of } \eta_{\mc H}^* (V)} = 
\frac{|\Omega_{G_\ad,\mf f}| |\Omega_{G_\ad} \cdot \mf f|}{ 
|\Omega_{G,\mf f}| |\Omega_G \cdot \mf f| \text{ length of } \eta^* (\delta_\ad)} .
\end{equation}
By the orbit counting lemma and \eqref{eq:1.6} this equals
\begin{equation}\label{eq:5.5} 
\frac{|\Omega_{G_\ad}|}{|\Omega_G| \text{ length of } \eta^* (\delta_\ad)} 
= \frac{|Z ({G^\vee}_\Sc)^{\mb W_K}|}{|Z(G^\vee)^{\mb W_K}| 
\text{ length of } \eta^* (\delta_\ad)} .
\end{equation}
Rearranging \eqref{eq:5.4} and \eqref{eq:5.5} yields the desired equality.\\
(b) By Theorem \ref{thm:3.1} $\phi_\delta$ is the composition of $\phi_{\delta_\ad}$ with
the quotient map 
\[
{}^L G_\ad = {G^\vee}_\Sc \rtimes \mb W_K \to G^\vee \rtimes \mb W_K = {}^L G .
\]
From Lemma \ref{lem:3.2}.c we know that $\eta^* (\delta_\ad)$ is the direct sum of exactly
\begin{equation}\label{eq:3.22}
\dim (\rho_{\phi_\ad}) [\mc A_{\phi_\delta} : \mc A_{\phi_{\delta_\ad}}] \dim (\rho_\delta)^{-1}
\end{equation}
irreducible $G$-representations. It is known from \cite[(13.6)]{FOS} that 
\begin{equation}\label{eq:3.23}
\big[ \mc A_{\phi_\delta} : \mc A_{\phi_{\delta_\ad}} \big] = \big[ Z_{{G^\vee}_\Sc}(\phi_\delta) : 
Z_{{G^\vee}_\Sc}(\phi_{\delta_\ad}) \big] = \Big[ S_{\phi_\delta}^\sharp : 
S_{\phi_{\delta_\ad}}^\sharp \Big] \frac{|Z({G^\vee}_\Sc)^{\mb W_K}|}{|Z(G^\vee)^{\mb W_K}|} .
\end{equation}
From \eqref{eq:5.4}, \eqref{eq:5.5}, \eqref{eq:3.22} and \eqref{eq:3.23} we deduce
\begin{equation}\label{eq:5.6}
\frac{\fdeg (\delta)}{\fdeg (\delta_\ad)} = 
\frac{|Z({G^\vee}_\Sc)^{\mb W_K}|}{|Z(G^\vee)^{\mb W_K}|} \frac{\big| \mc A_{\phi_{\delta_\ad}} 
\big| \dim (\rho_\delta)}{|\mc A_{\phi_\delta}| \dim (\rho_{\phi_\delta})} 
= \frac{\big|S_{\phi_{\delta_\ad}}^\sharp \big| \, \dim (\rho_\delta)}{| S_{\phi_\delta}^\sharp |
\, \dim (\rho_{\delta_\ad})} .
\end{equation}
As Lie$({G^\vee}_\Sc) = \mr{Lie}(G^\vee)$, 
\[
\gamma (s,\mr{Ad}_{G^\vee} \circ \phi_\delta, \psi) = \gamma (s,\mr{Ad}_{{G^\vee}_\Sc} \circ
\phi_{\delta_\ad},\psi) \qquad \text{for all } s \in \C .
\]
Then \eqref{eq:5.6} says
\[
\frac{\fdeg (\delta)}{\fdeg (\delta_\ad)} = \frac{\dim (\rho_\delta) \, 
|S_{\phi_{\delta_\ad}}^\sharp|}{\dim (\rho_{\delta_\ad}) \, | S_{\phi_\delta}^\sharp |} 
\frac{\gamma (0,\mr{Ad}_{G^\vee} \circ \phi_\delta, \psi)}{
\gamma (0,\mr{Ad}_{{G^\vee}_\Sc} \circ \phi_{\delta_\ad},\psi)} . 
\]
Combining that with Theorem \ref{thm:2.1}.a, we obtain the desired formula for $\fdeg (\delta)$.
\end{proof}

\subsection{Reductive groups} \
\label{par:red}

To prove the HII conjecture for unipotent representations of a reductive group $G$,
we want to compare its representations with those of $G / Z(G)_s$ and with those of
the derived group $G_\der := \mc G_\der (K)$. (Notice that $G_\der$ may be larger than 
the derived group of $G$ as an abstract group.)

We start with some preparations for the case that $Z(\mc G)^\circ$ is $K$-anisotropic.
Let $\overline{Z(\mc G)^\circ}$ be the connected reductive $k$-group associated by
Bruhat--Tits to the unique vertex of $\mc B (Z(\mc G)^\circ,K)$.
 
For any Levi $K$-subgroup $\mc M$ of $\mc G$, $\mc M_d := \mc M \cap \mc G_\der$
is a Levi $K$-subgroup of $\mc G_\der$. Furthermore $Z(\mc M)_s \subset \mc M_d$,
for example because Lie$(Z(\mc M)_s) \subset \mr{Lie}(\mc G_\der)$. We note also that
$\mc M_d / Z(\mc M)_s$ is the derived group of $\mc M / Z(\mc M)_s$.

\begin{lem}\label{lem:4.1}
Suppose that $Z(\mc G)^\circ$ is $K$-anisotropic. The inclusion $M_d \to M$ induces:
\enuma{
\item a bijection $\Irr_{\unip,\cusp}(M) \to \Irr_{\unip,\cusp}(M_d)$;
\item for every minimal facet $\mf f$ of $\mc B (\mc M_\ad,K)$, a bijection
between the types $(\hat{P}_{\mf f},\hat \sigma)$ for $M$ and for $M_d$.
}
\end{lem}
\begin{proof}
(a) Let $X_\Wr(M)$ be the group of weakly unramified characters, i.e. those characters 
$M \to \C^\times$ that are trivial on the kernel of the Kottwitz homomorphism $\kappa_M$.
From the short exact sequence \eqref{eq:1.1} (for $\mc M$ and for $\mc M_d$) we deduce
that there are natural isomorphisms
\begin{equation}\label{eq:4.4}
X_\Wr (M) / X_\Wr (M / Z(M)_s) \cong X_\Wr (Z(M)_s) \cong X_\Wr (M_d) / X_\Wr (M_d / Z(M)_s) .
\end{equation}
By \cite[(15.6)]{FOS} every irreducible cuspidal unipotent $M$-representation is
of the form $\pi_{M / Z(M)_s} \otimes \chi_M$ with $\pi_{M / Z(M)_s} \in 
\Irr_{\unip,\cusp}(M / Z(M)_s)$ and $\chi_M \in X_\nr (M)$. Using weakly unramified characters,
we can formulate this more precisely as a bijection
\begin{equation}\label{eq:4.1}
\Irr_{\unip,\cusp} (M / Z(M)_s) \underset{X_\Wr (M / Z(M)_s)}{\times} X_\Wr (M)
\to \Irr_{\unip,\cusp}(M) .
\end{equation}
Similarly there is a bijection
\begin{equation}\label{eq:4.2}
\Irr_{\unip,\cusp} (M_d / Z(M)_s) \underset{X_\Wr (M_d / Z(M)_s)}{\times} 
X_\Wr (M_d) \to \Irr_{\unip,\cusp}(M_d) .
\end{equation}
We note that $Z(\mc M / Z (\mc M)_s)^\circ$ is isogenous to $Z(\mc G)^\circ$, and in
particular it is $K$-anisotropic. Hence we may apply \cite[Lemma 15.3]{FOS}, which tells
us that the inclusion $M_d / Z(M)_s \to M / Z(M)_s$ induces a bijection
\begin{equation}\label{eq:4.3}
\Irr_{\unip,\cusp} (M / Z(M)_s)  \to \Irr_{\unip,\cusp} (M_d / Z(M)_s) .
\end{equation}
Combining \eqref{eq:4.3} and \eqref{eq:4.4} with \eqref{eq:4.1} and \eqref{eq:4.2},
we obtain the required bijection.\\
(b) The (semisimple) Bruhat--Tits buildings of $M, M_d ,M / Z(M)_s$ and $M_d /Z(M)_s$
can be identified \cite[\S 2]{Tit}. In particular these buildings have the same collections
of facets $\mf f$. The group $\overline{\mc M_{\mf f}^\circ}$ is isogenous to the direct
product of $\overline{\mc M_{d,\mf f}^\circ}$ and the $k$-torus 
$\overline{Z(\mc G)^\circ}$. The only cuspidal unipotent representation of 
$\overline{Z(\mc G)^\circ}(k)$ is the trivial representation. The collection of cuspidal
unipotent representations of $\big( \overline{\mc M_{d,\mf f}^\circ} \times 
\overline{Z(\mc G)^\circ} \big) (k)$ does not change under isogenies of $k$-groups
\cite[\S 3]{Lus-Che}, so it is the same as for $\overline{\mc M_{\mf f}^\circ}(k)$. 
As the semisimple group $\mc M_d / Z(\mc M)_s$ is the derived group of $\mc M / Z(\mc M)_s$, 
\cite[Lemma 15.2]{FOS} says that $\Omega_{M / Z(M)_s} = \Omega_{M_d / Z(M)_s}$.
Combining that with \eqref{eq:1.12}, we find that 
\begin{equation}\label{eq:4.5}
\Omega_M = \Omega_{M_d} .
\end{equation}
With \eqref{eq:1.7} we deduce that $M_d \to M$ induces a bijection between the indicated
collections of types.
\end{proof}

The behaviour of formal degrees of supercuspidal unipotent representations under pullback
from $M$ to $M_d$ was analysed in \cite[(16.13)]{FOS}. That and Lemma \ref{lem:4.1} can be 
generalized to all (square-integrable) unipotent representations:

\begin{lem}\label{lem:4.2}
Suppose that $Z(\mc G)^\circ$ is $K$-anisotropic. Let $(\hat{P}_{\mf f,M}, \hat \sigma)$
and $(\hat{P}_{\mf f,M_d},\hat \sigma)$ be as in Lemma \ref{lem:4.1}.b.
\enuma{
\item The inclusion $G_\der \to G$ induces an algebra isomorphism
\[
\mc H (G_\der, \hat{P}_{\mf f,G_\der},\hat \sigma) \to \mc H (G,\hat{P}_{\mf f,G},\hat \sigma) .
\]
\item Suppose that $\delta \in \Irr (G)_{(\hat{P}_{\mf f,G},\hat \sigma)}$ is square-integrable,
and let $\delta_\der$ be its pullback to $G_\der$. Then $\delta_\der$ is irreducible and
\[
\frac{\fdeg (\delta)}{\fdeg (\delta_\der)} = 
\frac{ q^{\dim (\overline{Z(\mc G)^\circ}) / 2}}{|\overline{Z(\mc G)^\circ}(k) |} .
\]
}
\end{lem}
\begin{proof}
(a) By \cite[Lemma 3.5 and (42)]{SolLLCunip} these two affine Hecke algebras differ only in
the involved lattices $X_{\mf f}$. From \eqref{eq:4.5} and the proof of 
\cite[Theorem 3.3.b]{SolLLCunip} we see that
\[
X_{\mf f,G} = \Omega_{M,\mf f_M} / \Omega_{M,\mf f_M,\tor} = \Omega_{M_d,\mf f_M} /
\Omega_{M_d,\mf f_M,\tor} = X_{\mf f,G_\der} .
\]
Hence $\mc H (G_\der, \hat{P}_{\mf f,G_\der},\hat \sigma)$ can be identified with
$\mc H (G,\hat{P}_{\mf f,G},\hat \sigma)$, and the canonical map between them is an 
isomorphism. We note that nevertheless the traces of these algebras may be normalized
differently.\\
(b) Let $\delta_{\mc H}$ be the $\mc H (G, \hat{P}_{\mf f,G},\hat \sigma)$-module
associated to $\delta$ via \eqref{eq:1.4}. By Lemma \ref{lem:4.1} 
$\delta_\der \in \Rep (G_\der)_{( \hat{P}_{\mf f,G_\der},\hat \sigma)}$. By part (a)
the $\mc H (G_\der, \hat{P}_{\mf f,G_\der},\hat \sigma)$-module $\delta_{\der,\mc H}$
can be identified with $\delta_{\mc H}$, and in particular it is irreducible. 
From \cite{BHK} and \eqref{eq:1.5} we see that
\[
\frac{\fdeg (\delta)}{\fdeg (\delta_\der)} = 
\frac{\fdeg (\delta_{\mc H})}{\fdeg (\delta_{\der,\mc H})} =
\frac{\vol (\hat{P}_{\mf f,G_\der})}{\vol (\hat{P}_{\mf f,G})} .
\]
By \eqref{eq:1.7} and \eqref{eq:4.5} this equals
\begin{equation}\label{eq:4.6}
\frac{\vol (P_{\mf f,G_\der}) |\Omega_{G_\der,\mf f,\tor}|}{
\vol (P_{\mf f,G}) |\Omega_{G,\mf f,\tor} |} =
\frac{\vol (P_{\mf f,G_\der}) |\Omega_{M_\der,\mf f,\tor}|}{
\vol (P_{\mf f,G}) |\Omega_{M,\mf f,\tor} |} =
\frac{\vol (P_{\mf f,G_\der})}{\vol (P_{\mf f,G})} . 
\end{equation} 
These volumes, with respect to our normalized Haar measures, are expressed in terms of
$k$-groups in \eqref{eq:1.3}. Since $\overline{\mc G_{\mf f}^\circ}$ is isogenous to
$\overline{\mc G_{\der,\mf f}^\circ} \times \overline{Z(\mc G)^\circ}$, we have
\cite[Proposition 1.4.12.c]{GeMa}
\[
|\overline{\mc G_{\mf f}^\circ} (k)| =
|\overline{\mc G_{\der,\mf f}^\circ}(k)| \, |\overline{Z(\mc G)^\circ}(k)| .
\]
Then \eqref{eq:4.6} becomes
\begin{equation}\label{eq:5.7}
\frac{\vol (P_{\mf f,G_\der})}{\vol (P_{\mf f,G})} = \frac{|\overline{\mc G_{\der,\mf f}^\circ}(k)| 
\, q^{-\dim (\overline{\mc G_{\der,\mf f}^\circ}) / 2}}{
|\overline{\mc G_{\mf f}^\circ}(k)| \, q^{-\dim (\overline{\mc G_{\mf f}^\circ}) / 2}} =
\frac{ q^{\dim (\overline{Z(\mc G)^\circ}) / 2}}{|\overline{Z(\mc G)^\circ}(k) |} .\qedhere
\end{equation}
\end{proof}

With all preparations complete, we can prove our main result, the HII conjecture
\eqref{eq:HII1} for unipotent representations.

\begin{thm}\label{thm:4.3}
Let $\mc G$ be a connected reductive $K$-group which splits over an unramified extension.
Let $\delta \in \Irr_\unip (G)$ be square-integrable modulo centre and let $(\phi_\delta,
\rho_\delta)$ be its enhanced L-parameter via Theorem \ref{thm:1.1}. Let $\psi : K \to \C^\times$
have order -1 and normalize the Haar measure on $G$ as in \cite{HII} and \cite{GaGr}. Then
\[
\fdeg (\delta) = \pm \dim (\rho_\delta) | S_{\phi_\delta}^\sharp |^{-1} 
\gamma (0,\mr{Ad}_{G^\vee} \circ \phi_\delta, \psi).
\]
\end{thm}
\begin{proof}
For the moment we assume that $Z(\mc G)^\circ$ is $K$-anisotropic. 
Then Lemma \ref{lem:4.2} tells us that the pullback $\delta_\der$ of $\delta$
along $G_\der \to G$ is irreducible, so that Theorem \ref{thm:3.10}.b applies to 
$\delta_\der \in \Irr_\unip (G_\der)$. 

In the proof of \cite[Lemma 16.4]{FOS} it was shown that 
\[
\frac{\dim (\rho_\delta) | S_{\phi_{\delta_\der}}^\sharp | \gamma (0,\mr{Ad}_{G^\vee} 
\circ \phi_\delta, \psi)}{\dim (\rho_{\delta_\der}) |S_{\phi_\delta}^\sharp | 
\gamma (0,\mr{Ad}_{G_\der^\vee} \circ \phi_{\delta_\der}, \psi)} =
\frac{q^{\dim (\overline{Z(\mc G)^\circ}) / 2}}{|\overline{Z(\mc G)^\circ}(k) |} .
\]
By Lemma \ref{lem:4.2}.b the right hand side equals $\fdeg (\delta) \fdeg (\delta_\der)^{-1}$.
Combining that with the formula for $\fdeg (\delta_\der)$ from Theorem \ref{thm:3.10}.b,
we find the desired expression for $\fdeg (\delta)$.

Now we consider any $\mc G$ as in the statement of the theorem. The connected reductive
$K$-group $\mc G / Z(\mc G)_s$ has $K$-anisotropic connected centre. It was shown in
\cite[proof of Theorem 3 on page 42]{FOS} how the theorem for $G$ can be derived from the
theorem for $G / Z(G)_s$. Although \cite{FOS} is formulated only for supercuspidal
representations, this proof also works for square-integrable modulo centre representations
when we use the local Langlands correspondence from Theorem \ref{thm:1.1} 
(especially part (b) on the compatibility with weakly unramified characters).
\end{proof}

\section{Extension to tempered representations}

\subsection{Normalization of densities} \
\label{par:norm}

In this paragraph we study the Plancherel densities for essentially square-integrable
representations of a reductive group $G$ with non-compact centre.

We fix an essentially square-integrable unipotent $\pi \in \Irr (G)$, trivial on $Z(G)_s$.
Recall that we have canonical Haar measures and hence canonical Plancherel measures
for $G$ and for $G / Z(G)_s$. Further, Conjecture \ref{conj:HII} and \cite{Wal}
impose a measure on $\mc O = X_\unr (G) \pi \subset \Irr (G)$. Our conventions
(in particular ord$(\psi) = -1$) force us to slightly modify the latter measure.
We propose a new normalization and we check that it results in a nice formula
for the Plancherel mass of $\mc O$.

Let $G^1$ be the subgroup of $G$ generated by all compact subgroups and let 
$Z(G)_s^1$ be the unique maximal compact subgroup of $Z(G)_s$. 
We endow $X_\unr (Z(G)_s)$ with the Haar measure of total mass $\vol (Z(G)_s^1)^{-1}$. 
Following \cite[p. 302]{Wal}, we decree that the covering maps
\[
\begin{array}{ccccc}
X_\unr (Z(G)_s) & \leftarrow & X_\unr (G) & \to & \mc O \\
\chi |_{Z(G)_s} &  \text{\rotatebox[origin=c]{180}{$\mapsto$}} & \chi &
\mapsto & \chi \otimes \pi
\end{array}
\]
are locally measure preserving. We denote the associated density on $\mc O$ by d$\mc O$.
Notice that the degree of $X_\unr (G) \to X_\unr (Z(G)_s)$ equals $[G : G^1 Z(G)_s]$. Write
\[
\mc O \cap \Irr (G / Z(G)_s) = \{ \pi \otimes \chi \in \mc O : Z(G)_s \subset \ker \chi \} .
\]
Tensoring $\pi$ with $\chi$ gives a covering map
\[
\ker \big( X_\unr (G) \to X_\unr (Z(G)_s) \big) \to \mc O \cap \Irr (G / Z(G)_s) ,
\]
whose degree equals the degree of $X_\unr (G) \to \mc O$. Hence the number of elements
of any fiber of $X_\unr (G) \to \mc O$ is
\[
[G : G^1 Z(G)_s] \, |\mc O \cap \Irr (G / Z(G)_s) |^{-1} .
\]
It follows that
\begin{align}
& \vol (X_\unr (G)) = [G : G^1 Z(G)_s] \, \vol ( Z(G)_s^1)^{-1} , \\
\label{eq:3.4} & \vol (\mc O) = | \mc O \cap \Irr (G / Z(G)_s) | \, \vol ( Z(G)_s^1)^{-1} .
\end{align}
\begin{lem}\label{lem:6.1}
The Plancherel density on $\mc O$ is $\fdeg (\pi) \,\textup{d} \mc O$ and
\[
\mu_{Pl} (\mc O) = \fdeg (\pi) \vol (\mc O) .
\]
\end{lem}
\begin{proof}
Choose a test function $f \in C_r^* (G)$ such that $f$ is supported on $G^1$, 
$\mr{tr} \, \pi (f) = 1$ and $f$ acts as 0 on all irreducible $G$-representations outside 
$\mc O$. Then $f$ is $Z(G)_s^1$-invariant and tr$(\pi \otimes \chi)(f) = 1$ for all 
$\chi \in X_\unr (G)$. By definition
\begin{equation}\label{eq:6.1}
f(1) = \int_{\mc O} \mr{tr} \, \pi (f) \, \textup{d} \mu_{Pl}(\pi) = \mu_{Pl}(\mc O).
\end{equation}
Since $f$ is $Z(G)_s^1$-invariant, it defines a function $f_1$ on 
\[
G^1 / Z(G)_s^1 \cong G^1 Z(G)_s / Z(G)_s , 
\]
which we extend by zero to whole of $G / Z(G)_s$. Due to the difference
in the Haar measures, $f$ and $f_1$ act differently on representations of $G / Z(G)_s$.
Instead, the function $f_2 := \vol (Z(G)_s^1) f_1$ has the same action as $f$ on any smooth
$G / Z(G)_s$-representation. This can be seen by expressing $f$ on a small 
subset of the form $X \cong X / Z(G)_s^1 \times Z(G)_s^ 1$ as
\[
f \big|_X = f_2 \big|_{X / Z(G)_s^1} \cdot \frac{1_{Z(G)_s^1}}{\vol (Z(G)_s^1)} .
\]
In view of the construction of $f$, the function $f_2$ detects only the $G$-representations 
$\chi \otimes \pi$ with $\chi \in X_\unr (G)$ and  $Z(G)_s \subset \ker \chi$. 
All these representations have the same Plancherel density (both for $G$ and for $G / Z(G)_s$).
The Plancherel formula for $G / Z(G)_s$ gives
\begin{multline}
f(1) \vol (Z(G)_s^1) = f_2 (1) = \int_{\mc O \cap \Irr (G / Z(G)_s)} \mr{tr} \, \pi (f_2) \,
\textup{d} \mu_{Pl,G/Z(G)_s}(\pi) \\
= \mu_{Pl,G/Z(G)_s}(\mc O \cap \Irr (G / Z(G)_s)) 
= |\mc O \cap \Irr (G / Z(G)_s)| \mu_{Pl,G/Z(G)_s} (\pi) .
\end{multline}
Comparing with \eqref{eq:3.4} and \eqref{eq:6.1}, we find
\[
\mu_{Pl}(\mc O) = \vol (Z(G)_s^1)^{-1} |\mc O \cap \Irr (G / Z(G)_s)| 
\mu_{Pl,G/Z(G)_s} (\pi) = \vol (\mc O) \fdeg (\pi) .
\]
As tensoring with unramified unitary characters preserves the Plancherel density, this
means that $\fdeg (\pi) \, \textup{d} \mc O$ is the Plancherel density on $\mc O$. 
\end{proof}

Let $(\hat P_{\mf f},\hat \sigma)$ be the unipotent type such that
$\pi \in \Irr (G)_{(\hat P_{\mf f},\hat \sigma)}$. We abbreviate 
$\mc H = \mc H (G,\hat P_{\mf f},\hat \sigma)$. The representation $\hat \sigma$ is trivial 
on $Z(G)_s^1$, so $(\hat{P}_{\mf f},\hat \sigma)$ descends to a type 
$(\hat{P}_{\mf f} / Z(M)_s^1,\hat \sigma)$ for the group $G / Z(G)_s$. This type can
detect more than one Bernstein component, because $\hat P_{G / Z(G)_s,\mf f}$ can
properly contain $\hat{P}_{\mf f} / Z(M)_s^1$. Let $\sigma'$ be the (unique) extension of
$\hat \sigma$ to $\hat P_{G / Z(G)_s,\mf f}$ which is contained in $\pi$. Then
\[
\mc H_{ss} := \mc H \big( G / Z(G)_s, \hat P_{G / Z(G)_s,\mf f}, \sigma' \big)
\]
is naturally a quotient of $\mc H$, obtained by mapping the generators $N_w \in \mc H$ with
$w \in Z(G) / Z(G)_s^1$ to suitable scalars. The traces $\tau$ and $\tau_{ss}$ of $\mc H$
and $\mc H_{ss}$, normalized as in \eqref{eq:1.10}, differ at the unit element:

\begin{lem}\label{lem:6.2}
$\displaystyle{ \frac{\tau_{ss}(N_e)}{\tau (N_e)} = \frac{|\Omega_{G,\mf f,\tor}|  \,
\vol (Z(G)_s^1) }{|\Omega_{G/Z(G)_s,\mf f,\tor}|}
= \frac{|\Omega_{G,\mf f,\tor}|}{|\Omega_{G/Z(G)_s,\mf f,\tor}|} 
\Big( \frac{q-1}{q^{1/2}} \Big)^{\dim (Z(\mc G)_s)} }$.
\end{lem}
\begin{proof}
By \eqref{eq:1.10} and \eqref{eq:1.7}
\[
\frac{\tau_{ss}(N_e)}{\tau (N_e)} = \frac{\dim (\sigma') \vol ( \hat P_{G / Z(G)_s,\mf f} )^{-1}}
{\dim (\hat \sigma) \, \vol ( \hat P_{G,\mf f} )^{-1}} 
= \frac{|\Omega_{G,\mf f,\tor}| \, \vol ( P_{G,\mf f} )}{|\Omega_{G/Z(G)_s,\mf f,\tor}| 
\vol ( P_{G / Z(G)_s,\mf f} )} .
\]
Let $\overline{Z(\mc G)_s} \cong GL_1^{\dim Z(\mc G)_s}$ be the connected reductive $k$-group 
associated to the unique vertex of $\mc B (Z(\mc G)_s,K)$. Since 
$\overline{\mc G_{\mf f}^\circ}$ is isogenous to 
$\overline{(\mc G / Z(\mc G)_s)_{\mf f}^\circ} \times \overline{Z(\mc G)_s}$, a calculation
analogous to \eqref{eq:5.7} shows that
\[
\frac{\vol (P_{\mf f})}{\vol (P_{G / Z(G)_s, \mf f})} =  
\frac{|\overline{Z(\mc G)_s}(k) |}{ q^{\dim (\overline{Z(\mc G)_s}) / 2}}=
\vol (Z(G)_s^1) .
\]
Finally, we note that $|\overline{Z(\mc G)_s}(k)| = |GL_1 (k)|^{\dim Z(\mc G)_s} =
(q-1)^{\dim Z(\mc G)_s}$.
\end{proof}

As only the trivial element of $\Omega_{Z(G)_s}$ fixes any point of the standard apartment
$\mh A$ of $\mc B (\mc G,K)$, \eqref{eq:1.12} entails that the natural map
\[
\Omega_{G,\mf f,\tor} \to \Omega_{G/Z(G)_s,\mf f,\tor}
\]
is injective. However, in general it need not be surjective.

We write $T_{\mf f,ss} = \Hom (X_{G / Z(G)_s, \mf f},\C^\times)$, a subtorus of
$T_{\mf f} = \Hom (X_{\mf f},\C^\times)$. The image of $X_\nr (G)$ in $T_{\mf f}$ is
another algebraic subtorus $T_{\mf f,Z}$, which is complementary in the sense that
\[
T_{\mf f,ss} T_{\mf f,Z} = T_{\mf f} \quad \text{and} \quad 
|T_{\mf f,ss} \cap T_{\mf f,Z}|< \infty. 
\]
Let $T_{\mf f,un} = \Hom (X_{\mf f},S^1)$, the maximal compact real subtorus of
$T_{\mf f}$. We define $T_{\mf f,ss,un}$ and $T_{\mf f,Z,un}$ similarly. 
Write $\pi_{\mc H} = \Hom_{\hat P_{\mf f}}(\hat \sigma,\pi) \in \mr{Mod} (\mc H)$. 
By \eqref{eq:1.4} the map
\[
X_\unr (G) \to \mc O : \chi \mapsto \chi \otimes \pi 
\]
induces a surjection 
\[
T_{\mf f,Z,un} \to T_{\mf f,Z,un} \pi_{\mc H} = 
\{ \Hom_{\hat P_{\mf f}}(\hat \sigma,\chi \otimes \pi) : \chi \in X_\unr (G) \} . 
\]
Furthermore $T_{\mf f,Z,un} \pi_{\mc H}$ is in bijection with $\mc O$ via
\eqref{eq:1.4}. 

Let d$t_Z$ be the Haar measure on $T_{\mf f,Z,un} \pi_{\mc H}$ with total volume 1.
Since \eqref{eq:1.4} preserves Plancherel measures \cite{BHK}, $\mu_{Pl,G} 
\big|_{\mc O}$ and $\mu_{Pl,\mc H} \big|_{T_{\mf f,Z,un} \pi_{\mc H}}$ agree. 
With Lemma \ref{lem:6.1} we find that  
\begin{equation}\label{eq:6.2}
\textup{d} \mu_{Pl,\mc H}(\pi_{\mc H}) = \fdeg_{\mc H_{ss}}(\pi_{\mc H}) \,| \mc O \cap
\Irr (G / Z(G)_s) | \, \vol (Z(G)_s^1 )^{-1} \textup{d} t_Z .
\end{equation}
This can also be formulated entirely in terms of affine Hecke algebras:

\begin{lem}\label{lem:6.3}
$\textup{d} \mu_{Pl,\mc H}(\pi_{\mc H}) = \fdeg_{\mc H_{ss}}(\pi_{\mc H}) \tau (N_e) 
\tau_{ss}(N_e)^{-1} | (T_{\mf f,ss} \cap T_{\mf f,Z}) \pi_{\mc H}| \, \textup{d} t_Z $. 
\end{lem}
\begin{proof}
Consider an arbitrary extension $\sigma''$ of $\hat \sigma$ to $\hat P_{G / Z(G)_s,\mf f}$.
From \cite[\S 1.20]{LusUni1} or \cite[(40)]{SolLLCunip} we see that the number
of elements of $\mc O$ that contain $\sigma''$ equals the number of elements that contain
$\sigma'$. The number of possible extensions $\sigma''$ is 
$|\Omega_{G/Z(G)_s,\mf f,\tor}| \, |\Omega_{G,\mf f,\tor}|^{-1}$, and hence
\begin{equation}\label{eq:6.3}
\big| \mc O \cap \Irr (G / Z(G)_s) \big| = |\Omega_{G/Z(G)_s,\mf f,\tor}| \, 
|\Omega_{G,\mf f,\tor}|^{-1} \big| T_{\mf f,Z} \pi_{\mc H} \cap \Irr (\mc H_{ss}) \big| .
\end{equation}
An $\mc H$-representation $t \otimes \pi_{\mc H} \in T_{\mf f} \pi_{\mc H}$ descends to
$\mc H_{ss}$ if and only if $t \in T_{\mf f,ss}$. Therefore
\[
| T_{\mf f,Z} \pi_{\mc H} \cap \Irr (\mc H_{ss})| = 
| (T_{\mf f,ss} \cap T_{\mf f,Z}) \pi_{\mc H}| .
\]
Combine that with \eqref{eq:6.2}, \eqref{eq:6.3} and Lemma \ref{lem:6.1}.
\end{proof}

We remark that Lemma \ref{lem:6.3} is in accordance with a comparison formula for Plancherel
measures of affine Hecke algebras \cite[(4.96)]{Opd-Sp}.

\subsection{Parabolic induction and Plancherel densities} \
\label{par:par}

Let $\mc M$ be a Levi $K$-subgroup of $\mc G$ and let $\pi_M \in \Irr_\unip (M)$
be essentially square-integrable. As before we write $\mc O = X_\unr (M) \pi_M$.
Let $\mc P$ be a parabolic $K$-subgroup of $\mc G$ with Levi factor $\mc M$ and
denote the normalized parabolic induction functor by $I_P^G$.
We want to express the Plancherel density on the family of finite length 
tempered unitary $G$-representations 
\begin{equation}\label{eq:6.11}
I_P^G (\mc O) = \{ I_P^G (\chi \otimes \pi_M) : \chi \in X_\unr (M) \} .
\end{equation}
Recall that the infinitesimal character of $\phi_M$ is
\[
\mr{inf.ch.}(\phi_M) = \phi_M \big( \Fr , \matje{q^{-1/2}}{0}{0}{q^{1/2}} \big) .
\]
We abbreviate $\mc H = \mc H (G,\hat P_{\mf f},\hat \sigma)$ and
$\mc H^M = \mc H (M,\hat P_{M,\mf f},\hat \sigma)$, where
$\pi_M \in \Irr (M)_{(\hat P_{M,\mf f},\hat \sigma)}$. From the proof of Theorem 
\ref{thm:1.1}, which can be retraced to \cite[Theorem 3.18.b]{AMS3}, one sees
that the central character of 
\[
\pi_{M,\mc H} = \Hom_{\hat P_{M,\mf f}}(\hat \sigma, \pi_M) \in \mr{Mod}(\mc H^M)
\]
is completely determined by inf.ch.$(\phi_M)$. More precisely, choose a basepoint for
the appropriate Bernstein component of enhanced L-parameters as in 
\cite[Lemma 3.4]{SolLLCunip} and pick $t_M \in T^\vee$ such that inf.ch.$(\phi_M)$ equals 
$t_M$ times the basepoint. Then the central character of $\pi_{M,\mc H}$ is 
$W(R_{M,\mf f}) t_M \in T_{M,\mf f} / W(R_{M,\mf f})$. 

For $t \in Z(M^\vee)^{\theta,\circ} \cong X_\nr (M)$ the $M$-representation $t \otimes \pi_M$
corresponds to $t \otimes \pi_{M,\mc H} \in \Irr (\mc H^M)$ and its enhanced L-parameter is 
$(t \phi_M, \rho_M)$, where $t \phi_M$ is defined in \eqref{eq:A.30}. 

\begin{lem}\label{lem:6.4}
For $t \in X_\unr (M)$:
\[
\textup{d} \mu_{Pl} \big( I_P^G (t \otimes \pi_M) \big) = 
\pm m^M_{\mf s} (t t_M) \gamma (0,\mr{Ad}_{M^\vee} \circ t\phi_M,\psi)
\dim (\rho_M) | S^\sharp_{\phi_M} |^{-1} \textup{d} \mc O (t \otimes \pi_M) .
\]
The factor $m^M_{\mf s} (t t_M)$, which depends on the Bernstein component 
$\Rep (G)^{\mf s}$ containing $I_P^G (\pi_M)$, is defined in 
\cite[(3.57)]{Opd-Sp} and \cite[(2.17)]{Opd18}.
\end{lem}
\begin{proof}
Notice that $t \otimes \pi_M$ is still essentially square integrable, for 
that property is stable under tensoring by unitary characters.
By Lemma \ref{lem:6.3}, \eqref{eq:6.2} and the expression for  Plancherel densities in 
affine Hecke algebras \cite[(4.96)]{Opd-Sp}:
\begin{multline}\label{eq:6.4a}
\mu_{Pl,\mc H} \big( \mr{ind}_{\mc H^M}^{\mc H} (t \otimes \pi_{M,\mc H}) \big) =
m^M (t t_M) \mu_{Pl,\mc H^M}(t \otimes \pi_{M,\mc H}) \\
= m^M_{\mf s} (t t_M) \fdeg_{\mc H_{ss}}(\pi_{M,\mc H}) \, | \mc O \cap
\Irr (G / Z(G)_s) | \, \vol (Z(G)_s^1 )^{-1} \textup{d} t_Z .
\end{multline}
It is known from \cite[Lemma 4.1]{SolComp} that normalized parabolic induction commutes
with the functor \eqref{eq:1.4}. As \eqref{eq:1.4} preserves Plancherel densities, 
\eqref{eq:6.4a} and \eqref{eq:3.4} yield
\[
\mu_{Pl} \big( I_P^G (t \otimes \pi_M) \big) = 
m^M_{\mf s} (t t_M) \, \fdeg (t \otimes \pi_M) \, \textup{d} \mc O (t \otimes \pi_M) .  
\]
Applying Theorem \ref{thm:4.3} to $t \otimes \pi_M \in \Irr_\unip (M)$, 
we obtain the required formula.
\end{proof}

On the other hand, we already know from Theorem \ref{thm:2.1}.i that 
Conjecture \ref{conj:HII} holds up to some constant $C_{\mc O} \in \Q_{>0}$, that is:
\begin{equation}\label{eq:6.8}
\mu_{Pl} \big( I_P^G (t \otimes \pi_M) \big) = 
\pm C_{\mc O} \, \gamma (0,\mr{Ad}_{G^\vee,M^\vee} \circ t\phi_M,\psi) \dim (\rho_M) 
| S^\sharp_{\phi_M} |^{-1} \textup{d} \mc O (t \otimes \pi_M) .
\end{equation}

\begin{thm}\label{thm:6.5}
Let $\mc G$ be a connected reductive $K$-group, which spits over an unramified extension.
Let $\mc M$ be a Levi $K$-subgroup of $\mc G$ and let $\pi_M \in \Irr_\unip (M)$ be
square-integrable modulo centre. Let $\mc O = X_\unr (M) \pi_M$ be the associated
orbit in $\Irr_\unip (M)$ and define $I_P^G (\mc O)$ as in \eqref{eq:6.11}.

Let $(\phi_M,\rho_M) \in \Phi_{\nr,e}(M)$ be the enhanced L-parameter of $\pi_M$,
as in Theorem \ref{thm:1.1}. With the normalization from \eqref{eq:3.4}, the Plancherel 
density on $I_P^G (\mc O)$ is
\[
\pm \dim (\rho_M) |S_{\phi_M}^\sharp|^{-1} 
\gamma (0,\mr{Ad}_{G^\vee,M^\vee} \circ \phi_M,\psi) \, \textup{d}\mc O (\pi_M) . 
\]
That is, Conjecture \ref{conj:HII} holds for $\Irr_\unip (G)$, with $c_M = 1$.
\end{thm}
\begin{proof}
We may assume that $\mc M$ is a standard Levi subgroup, that is, $\mc M$ contains the
standard maximal $K$-split torus $\mc S$ and the standard maximal $K$-torus $\mc T$.
Let $\mc G^*$ be the quasi-split inner form of $\mc G$. We may identify $\mc T$ with a
maximal $K$-torus of $\mc G^*$. Let $\mc M^* \subset \mc G^*$ be the Levi subgroup such that 
$\Phi (\mc M^*,\mc T) = \Phi (\mc M,\mc T)$. 

Write inf.ch.$(\phi_M) = r_M \theta$ with $r_M \in \hat{T}^{\circ,\theta}$ (which can
be achieved by replacing $\phi_M$ with an equivalent L-parameter). 
By Lemma \ref{lem:A.3} and \eqref{eq:A.35}
\begin{equation}\label{eq:6.6}
\gamma (0, \mr{Ad}_{G^\vee,M^\vee} \circ t \phi_M, \psi) =
\pm \gamma (0,\mr{Ad}_{M^\vee} \circ t \phi_M,\psi) \, m^{M^*} (t r_M) .
\end{equation}
Comparing Lemma \ref{lem:6.4}, \eqref{eq:6.8} and \eqref{eq:6.6}, we see that
\begin{equation}\label{eq:6.9}
m^M_{\mf s} (t t_M) = \pm C_{\mc O} \, m^{M^*}(t r_M) \qquad \forall t \in X_\unr (M) .
\end{equation}
Let $R^{\mf s}_0$ denote the root system associated with $\mc{H}$, and let 
$q_{\mf s}$ denote the parameter function of $\mc{H}$, as in \cite[Section 2]{Opd2}. 
Let $m^{\mf s}_\pm$ be the corresponding parameter functions on $R^{\mf s}_0$. 
Let $w_{\mf s}^M\in W_{\mf s}$ denote the shortest length representative of 
the coset $w_{\mf s,0} W_{\mf s,M}$ in $W_{\mf s} / W_{\mf s,M}$ of the longest element 
$w_{\mf s,0}$. Like in Appendix \ref{sec:A2}, these parameters can be used to define 
$\mu$-functions. From \cite[Proposition 3.27(ii)]{Opd-Sp} we see that for $t \in X_\unr (M)$:
\begin{equation}\label{eq:6.10}
m^M_{\mf s}(t t_M) = q^{-1}_{\mf s} (w_{\mf s}^M) \prod_{a \in R^\mb{s}_0\backslash R^{\mf s}_{M,0}} 
\frac{(1-\gamma_a^{-2}(t t_M))}{(1+q^{-m_{-}^{\mf s}(\gamma_a)}\gamma_a^{-1}(t t_M))
(1-q^{-m_{+}^{\mf s}(\gamma_a)}\gamma_a^{-1}(t t_M))} .
\end{equation}
This is analogous to the formula \eqref{eq:A.33} for $m^{M^*}(t r_M)$. The differences
are that for $M^*$ the product runs over more roots and that the parameters 
$m_\pm^{\mf s} (\gamma_a)$ need not equal $m_\pm (\gamma_a)$. 

Since $\pi_M$ is essentially square-integrable, $\pi_{M,\mc H}$ is essentially discrete series. 
Together with \cite[Lemma 3.31 and Proposition A.4]{Opd-Sp} this implies that $X_\nr (M) t_M$ 
is a residual coset (of minimal dimension) for $\mc H^M$. Moreover $m_{\pm}^{\mf s}(\gamma_a) 
\in \Z$, so the value $\gamma_a (t_M) \in \C^\times$ is a root of unity times an integral power 
of $q$ \cite[Theorem A.7]{Opd-Sp}. In particular $\lim_{q \to 1} \gamma_a (t_M)$ is well-defined,
and a root of unity in $\C$.

The discrete unramified L-parameter for $M^*$ determines an L-packet
of essentially square-integrable $M^*$-representations. The Iwahori-spherical members of that
packet correspond to a finite set of essentially discrete series representations of the parabolic
subalgebra $\mc H (M^*, I^* \cap M^*)$ of $\mc H (G^*,I^*)$, with central character 
\[
W(M^\vee,T^\vee)^\theta r_M \in T^\vee_\theta / W(M^\vee,T^\vee)^\theta .
\]
As above for $t_M$, $X_\nr (M^*) r_M$ is a residual coset (of minimal dimension) for \\
$\mc H (M^*, I^* \cap M^*)$. It follows as above that
the values $\gamma_a (r_M)$, with $a \in (\Phi \setminus \Phi_{M^*}) / \theta$ as in 
\eqref{eq:A.33}, are products of roots of unity and integral powers of $q$.

Taking this dependence of $t_M$ and $r_M$ on $q$ into account, we regard both sides of 
\eqref{eq:6.9} as rational functions in $t$ and in $q$.
Fix $t \in X_\unr (M)$ such that both $t t_M$ and $t r_M$ are in generic position with 
respect to all the involved roots. Then \eqref{eq:6.10} and \eqref{eq:A.33} entail
\[
\lim_{q \to 1} m^M_\mf {s} (t t_M) = 1 \qquad \text{and} \qquad \lim_{q \to 1} m^{M^*}(t r_M) = 1 .
\]
Combining that with \eqref{eq:6.9}, we find $c_{\mc O} = 1$ and $m^M_{\mf s} (t t_M) = 
\pm m^{M^*}(t r_M)$ for all $t \in X_\unr (M)$. Then \eqref{eq:6.6} becomes
\begin{equation*}
\gamma (0, \mr{Ad}_{G^\vee,M^\vee} \circ t \phi_M, \psi) = \pm m^M (t t_M)
\gamma (0,\mr{Ad}_{M^\vee} \circ t \phi_M,\psi) .
\end{equation*}
Hence the expression in Lemma \ref{lem:6.4} equals \eqref{eq:6.8} 
with $c_{\mc O} = 1$, as required.
\end{proof}

\appendix
\section{Adjoint $\gamma$-factors}

Let $(\rho,V)$ be a finite dimensional Weil--Deligne representation over $\C$,
that is, a semisimple representation of $\mb W_K$ on $V$ together with a nilpotent 
operator $N \in \mr{End}_\C (V)$, such that 
\[
\rho (w) N \rho (w)^{-1} = \| w \| N \quad \text{for all } w \in \mb W_K .
\]
The contragredient of $(\rho,V)$ is the contragredient $(\rho^\vee,V^\vee)$ as 
$\mb W_K$-representation, together with the nilpotent operator 
$N^\vee \in \mr{End}_\C (V^\vee)$ which sends $\lambda$ to $-\lambda \circ N$. 
We write $V_N = \ker (N)$ and we fix an additive character $\psi : K \to \C^\times$.
Recall from \cite[\S 4.1.6]{Tat} that the local factors of $(\rho,V)$ are defined, 
as meromorphic functions of $s \in \C$, by:
\begin{equation}\label{eq:A.1}
\begin{aligned}
& L(s,\rho) = \det \big( 1 - q^{-s} \rho (\Fr) | V_N^{\mb I_K} \big)^{-1} , \\
& \epsilon (s,\rho,\psi) = \epsilon (s,\rho_0,\psi) 
\det \big( -q^{-s} \rho (\Fr) | V^{\mb I_K} / V_N^{\mb I_K} \big) ,\\
& \gamma (s,\rho,\psi) = \epsilon (s,\rho,\psi) L(1-s,\rho^\vee) L(s,\rho)^{-1} .
\end{aligned}
\end{equation}
The $\epsilon$-factor can be described further in terms of \cite[Theorem 3.4.1]{Tat} and
the Artin conductor $a(V)$:
\[
\epsilon (s,\rho_0,\psi) = \epsilon (\rho_0 \otimes \| \, \|^{1/2},\psi) q^{a(V) (1/2 - s)} .
\]
It is well-known, for instance from \cite[Proposition 2.2]{GrRe}, that $\rho$ gives rise
to a semisimple representation 
\begin{equation}\label{eq:A.13}
\tilde \rho : \mb W_K \times SL_2 (\C) \to \mr{Aut}_\C (V) \text{, such that }
N = \textup{d}\tilde \rho |_{SL_2 (\C)} \matje{0}{1}{0}{0}.
\end{equation}
Such a $\tilde \rho$ is unique up to conjugacy in $\mr{Aut}_\C (V)$, and it determines $\rho$.
We say that $(\rho,V)$ is self-dual if it is isomorphic to its contragredient. This 
is equivalent to self-duality of $\tilde \rho$.

\subsection{Independence of the nilpotent operator} \
\label{sec:A1}

We define a new Weil--Deligne representation $(\rho_0,V)$, by decreeing that as 
$\mb W_K$-representation it is the same as $(\rho,V)$, but with nilpotent operator
$N_0 = 0$. In view of the known properties of $\gamma$-factors for representations of
$GL_n (K)$ \cite[(2.7.3)]{Jac}, one can expect a relation between the $\gamma$-factors
of $\rho$ and of $\rho_0$.

\begin{prop}\label{prop:A.1}
Let $(\rho,V)$ be a finite dimensional self-dual Weil--Deligne representation over $\C$.
Then
\[
\gamma (0,\rho,\psi) = \pm \gamma (0,\rho_0,\psi) . 
\]
That is, up to a sign the $\gamma$-factor of $\rho$ at $s=0$ does not depend on the 
nilpotent operator $N$.
\end{prop}
\begin{proof}
Let $(\rho',V')$ be the sum of the irreducible nontrivial $\mb I_K$-subrepresentations
of $(\rho,V)$. We denote the irreducible $SL_2 (\C)$-representation of dimension $n+1$
by $(\sigma_n,\mr{Sym}^n)$ and we write
\[
V_n := \Hom_{\mb I_K \times SL_2 (\C)} (\mr{triv} \otimes \sigma_n, \tilde \rho).
\]
We can decompose the $\mb W_K \times SL_2 (\C)$-representation $\tilde \rho$ as
\begin{equation}\label{eq:A.2}
V = V' \oplus \bigoplus\nolimits_{n=0}^\infty V_n \otimes \mr{Sym}^n . 
\end{equation}
In view of the additivity of the local factors \eqref{eq:A.1}, it suffices to prove the
proposition for each of the direct summands in \eqref{eq:A.2} separately. It follows 
directly from the definitions that
\[
\gamma (s,\rho',\psi) = \epsilon (s, \rho'_0,\psi) = \gamma (s,\rho'_0,\psi) ,
\]
so we only have to consider $V_n \otimes \mr{Sym}^n$ for a fixed (but arbitrary) 
$n \in \Z_{\geq 0}$. Since $\mb I_K$ is normal in $\mb W_K$, $\tilde \rho$ induces an 
action of $\mb W_K / \mb I_K \cong \Z$ on $V_n$. We decompose it \nolinebreak as 
\begin{equation}
V_n = \bigoplus\nolimits_{\chi \in \Irr (\mb W_K / \mb I_K)} \C_\chi^{m_\chi} ,  
\end{equation}
where $m_\chi$ denotes the multiplicity of $\chi$ in $V_n$.

Since $(\tilde \rho,V)$ is self-dual and $\mr{Sym}^n$ is self-dual (as 
$SL_2 (\C)$-representation), the $\mb W_K$-representation $V_n$ is also self-dual. Hence 
\begin{equation}\label{eq:A.3}
m_{\chi^{-1}} = m_\chi \quad \text{for all } \chi \in \Irr (\mb W_K / \mb I_K).
\end{equation}
To simplify the notation, we assume from now on that $V = V_n \otimes \mr{Sym}^n$,
and in particular that $V^{\mb I_K} = V$. The relation between $\rho$ and $\tilde \rho$
entails that
\begin{equation}\label{eq:A.4}
\rho (\Fr) = \tilde \rho (\Fr) \otimes \tilde \rho \big( 1, \matje{q^{-1/2}}{0}{0}{q^{1/2}} \big)
= \tilde \rho (\Fr) \otimes \sigma_n \matje{q^{-1/2}}{0}{0}{q^{1/2}} .
\end{equation}
From \eqref{eq:A.3} we see that
\begin{equation}\label{eq:A.5}
\det (\tilde \rho (\Fr) | V_n ) = 
\prod\nolimits_{\chi \in \Irr (\mb W_K / \mb I_K)}  \chi (\Fr)^{m_\chi} = (-1)^{m_{\chi_-}} ,
\end{equation}
where $\chi_-$ denotes the unique quadratic character of $\mb W_K / \mb I_K$. As
$\sigma_n (SL_2 (\C)) \subset SL_{n+1}(\C)$, \eqref{eq:A.4} and \eqref{eq:A.5} yield
\begin{equation}\label{eq:A.6}
\det (\rho (\Fr) | V) = \det (\tilde \rho (\Fr) | V_n )^{n+1} = (-1)^{(n+1) m_{\chi_-}} .
\end{equation}
Since $\mr{Sym}^n_N$ is one-dimensional with $\matje{q^{-1/2}}{0}{0}{q^{1/2}}$ acting as 
$q^{-n/2}$, 
\begin{equation}\label{eq:A.7}
\det \Big( \sigma_n \matje{q^{-1/2}}{0}{0}{q^{1/2}} \big| \mr{Sym}^n / \mr{Sym}^n_N \Big) = q^{n/2} .
\end{equation}
With \eqref{eq:A.4}--\eqref{eq:A.7}, we can express the $\epsilon$-factor as
\begin{align}
\nonumber \epsilon (s,\rho,V) & = \epsilon (s,\rho_0,\psi) 
\det \big( -q^{-s} \rho (\Fr) | V_n \otimes \mr{Sym}^n / \mr{Sym}^n_N \big) \\
\label{eq:A.8} & 
= \epsilon (s,\rho_0,\psi) (-q^{-s})^{\dim (V_n) n} (-1)^{n m_{\chi_-}} q^{\dim (V_n) n /2} .
\end{align}
Using self-duality, the definitions \eqref{eq:A.1} and \eqref{eq:A.8}, we compute
\begin{equation}\label{eq:A.9}
\begin{aligned}
\frac{\gamma (s,\rho,\psi)}{\gamma (s,\rho_0,\psi)} & =
\frac{(-1)^{n m_{\chi_-}}q^{\dim (V_n) n /2}}{ (-q^s )^{\dim (V_n) n}}
\frac{\det \big( 1 - q^{s-1} \rho (\Fr) | V_n \otimes \mr{Sym}^n / \mr{Sym}^n_N  \big)}{
\det \big( 1 - q^{-s} \rho (\Fr) | V_n \otimes \mr{Sym}^n / \mr{Sym}^n_N  \big)} \\
& = \frac{(-1)^{n m_{\chi_-}}q^{\dim (V_n) n /2}}{ (-q^s )^{\dim (V_n) n}}
\prod_{\chi} \prod_{k=1}^n \left( \frac{1 - q^{s-1} \chi (\Fr) q^{k-n/2}}{
1 - q^{-s} \chi (\Fr) q^{k - n/2}} \right)^{m_\chi} \\
& = \frac{(-1)^{n m_{\chi_-}}q^{\dim (V_n) n /2}}{ (-1 )^{\dim (V_n) n}}
\prod_{\chi} \prod_{k=1}^n \left( \frac{1 - q^s \chi (\Fr) q^{k-1-n/2}}{
q^s - \chi (\Fr) q^{k - n/2}} \right)^{m_\chi} .
\end{aligned}
\end{equation}
When $s$ goes to 0, the products over $k$ in \eqref{eq:A.9} attain telescopic behaviour,
and all terms in the numerator (except $k=1$) cancel against all terms in the denominator
(except $k=n$). This is obvious when $\chi (\Fr) q^{k-n/2} \neq 1$, while we pick up an
extra factor $-1$ if $\chi (\Fr) q^{k-n/2} = 1$. Collecting all factors $-1$ in one
symbol $\pm$, \eqref{eq:A.9} yields
\begin{equation}\label{eq:A.10}
\begin{aligned}
\frac{\gamma (0,\rho,\psi)}{\gamma (0,\rho_0,\psi)} & = \pm q^{\dim (V_n) n /2}
\lim_{s \to 0} \prod_{\chi} \left( \frac{1 - q^s \chi (\Fr) q^{-n/2}}{
q^s - \chi (\Fr) q^{n/2}} \right)^{m_\chi} \\
& = \pm \lim_{s \to 0} \prod_{\chi} \chi (\Fr)^{m_\chi} \left( \frac{q^{n/2} 
\chi(\Fr)^{-1} - q^s}{q^s - \chi (\Fr) q^{n/2}} \right)^{m_\chi} 
\end{aligned}
\end{equation}
In view of \eqref{eq:A.5}, all the terms $\chi (\Fr)^{m_\chi}$ together just contribute
a sign, so we may omit them (or rather, put them into $\pm$). When $\chi (\Fr) \neq q^{\pm n/2}$,
\eqref{eq:A.3} shows that the terms in \eqref{eq:A.10} associated to $\chi$ will cancel
against the terms associated to $\chi^{-1}$, up to a sign. Thus only the characters 
$\chi^{\pm 1}$ with $\chi (\Fr) = q^{n/2}$ remain in \eqref{eq:A.9} upon taking the limit 
$s \to 0$, and for those we compute:
\begin{equation}
\frac{\gamma (0,\rho,\psi)}{\gamma (0,\rho_0,\psi)} = 
\pm \lim_{s \to 0} \left( \frac{1 - q^s}{q^s - q^n} \right)^{m_\chi} 
\left( \frac{q^n - q^s}{q^s - 1} \right)^{m_\chi} = \pm (-1)^{2 m_\chi} = \pm 1 .
\end{equation}
This concludes the proof, and we note that by retracing the various steps one can
find an explicit (but involved) formula for the sign.
\end{proof}

The adjoint $\gamma$-factor of a L-parameter $\phi$ for $G = \mc G (K)$ comes from a 
Weil--Deligne representation on $\mr{Lie}(G^\vee) / \mr{Lie}(Z(G^\vee)^{\mb W_K})$, which is 
self-dual with respect to the Killing form \cite[\S 3.2]{GrRe}. Proposition \ref{prop:A.1}
says that the $\gamma$-factor of $\tilde \rho = \mr{Ad}_{G^\vee} \circ \phi$ equals the 
$\gamma$-factor of $\rho_0$ (both at $s=0$ and up to a sign). We note that 
\begin{equation}\label{eq:respt}
\rho_0 (\Fr) = 
\mr{Ad}_{G^\vee} \Big( \phi \big( \Fr, \matje{q^{-1/2}}{0}{0}{q^{1/2}} \big) \Big) ,
\end{equation}
where we recognize the right hand side as the adjoint representation $\mr{Ad}_{G^\vee}$ 
applied to the infinitesimal character of $\phi$. In these terms, Proposition \ref{prop:A.1} 
says that \\ $\gamma ( 0,\mr{Ad}_{G^\vee} \circ \phi, \psi)$
depends only on $\phi |_{\mb I_K}$ and the infinitesimal character of $\phi$.

\subsection{Relation with $\mu$-functions} \
\label{sec:A2}

The goal of this paragraph to relate adjoint $\gamma$-factors of unramified L-parameters
to $\mu$-functions for Iwahori--Hecke algebras. The desired equality was already claimed in
\cite[(38)]{Opd2}, we take this opportunity to work out the proof.

From now we assume that $\mathcal{G}$ is unramified over $K$, that is, $\mc G$ is quasi-split and
splits over an unramified extension of $K$. Fix a pinning of the Lie algebra $\mr{Lie}(G^\vee)$ 
and let $\theta$ denote the pinned automorphism of $\mr{Lie}(G^\vee)$ induced by $\textup{Frob}$. 
The quotient $\mathcal{G}/Z(\mathcal{G})_s$ defines 
a $\theta$-stable reductive subgroup $\hat{G}\subset G^\vee$ with Lie algebra 
\[
\hat{\mf g} := \mr{Lie}(\hat{G}) \cong \mr{Lie}(G^\vee) / \mr{Lie}(Z(G^\vee)^\theta). 
\]
Let $\mr{Ad}_{G^\vee}$ denote the adjoint action of ${}^L\mathcal{G}$ on $ \hat{\mf g}$. 
Let us denote the distinguished Cartan subgroup of $\hat{G}$ by $\hat{T}$, 
with corresponding Cartan subalgebra $\hat{\mf t}:=\mr{Lie}(\hat{T})$ of 
$\hat{\mf g}$. Clearly 
\begin{equation}\label{eq:A.12}
T^\vee = Z(G^\vee)^{\theta}\hat{T} = Z(G^\vee)^{\theta,\circ}\hat{T} ,
\end{equation}
where $Z(G^\vee)^{\theta,\circ}$ is the identity component of $Z(G^\vee)^{\theta}$.
By \cite[Section 3.3]{Re}, the Lie algebra $\hat{\mf g}^\theta$ is semisimple. 

\begin{lem}\label{lem:A.4}
Let $\psi_0 : K \to \C^\times$ be a character of order 0. Let $\phi_T$ be an unramified
L-parameter for $T$ and write $\phi_T (\Fr) = r \theta$. Then
\[
\gamma (s, \mr{Ad}_{G^\vee} |_{\hat{\mf t}} \circ \phi_T, \psi_0) =
\frac{\det(1-q^{-s} \mr{Ad}_{G^\vee}(r\theta)|_{\hat{\mf t}})}
{\det(1-q^{s-1} \mr{Ad}_{G^\vee}(r\theta)|_{\hat{\mf t}})} .
\]
For $s$ near 0 this can be expressed as
\[
\gamma (s, \mr{Ad}_{G^\vee} |_{\hat{\mf t}} \circ \phi_T, \psi_0) =
 s^{|\Delta / \theta|} \frac{n_1 \log (q)^{|\Delta / \theta|}
\prod\nolimits_{a \in \Delta / \theta} |a \cap \Delta|}
{\det(1-q^{-1} \mr{Ad}_{G^\vee}(r\theta)|_{\hat{\mf t}})} + \mc O (s^{|\Delta / \theta| + 1}) .
\]
Here $n_1$ is a positive integer which reduces to 1 if $Z(\mc G)^\circ$ is $K$-split.
\end{lem}
\begin{proof}
For any unramified representation $\rho$ of $\mathbf{W}_K$ and an additive character $\psi_0$ 
of order 0, \cite[(3.2.6) and (3.4.2)]{Tat} say that
\begin{equation}\label{eq:A.14}
\epsilon (s,\rho,\psi_0) = 1 \qquad \text{for all } s \in \C .
\end{equation}
This applies to $\rho = \mr{Ad}_{G^\vee} |_{\hat{\mf t}} \circ \phi$, and moreover $\rho$ is 
self-dual with respect to the Killing form. Knowing that, the definitions \eqref{eq:A.1} yield 
the asserted formula for $\gamma (s, \mr{Ad}_{G^\vee} |_{\hat{\mf t}} \circ \phi_T, \psi_0)$.

The finite order map $\mr{Ad}_{G^\vee}(r\theta)|_{\hat{\mf t}}$ cannot have an eigenvalue
$q \in \R_{>1}$. Hence the denominator $\det(1-q^{s-1} \mr{Ad}_{G^\vee}(r\theta)|_{
\hat{\mf t}})$ is regular at $s=0$, and behaves as expected.

The numerator $\det(1-q^{-s} \mr{Ad}_{G^\vee}(r\theta)|_{\hat{\mf t}})$ can be analysed
by splitting 
\begin{equation}\label{eq:A.17}
\hat{\mf t} = (1-\theta)Z(\hat{\mf g}) \oplus ( \hat{\mf t} \cap \hat{\mf g}_\der ) .
\end{equation}
On the first summand of \eqref{eq:A.17} we get
\begin{equation}\label{eq:A.18}
\lim_{s \to 0} \det(1-q^{-s} \mr{Ad}_{G^\vee}(r\theta)|_{(1 - \theta) Z(\hat{\mf g})})
= \det (1 - \theta |_{(1 - \theta) Z(\hat{\mf g})} ) .
\end{equation}
Identifying $(1 - \theta) Z(\hat{\mf g})$ with the Lie algebra of the complex dual 
group of $Z(\mc G)^\circ / Z(\mc G)_s$, we see that \eqref{eq:A.18} can be computed as
the determinant of a linear transformation of a (co)character lattice, so in particular
it is an integer. More precisely, as $\theta$ has finite order but no eigenvalues 1 on
the involved lattice, \eqref{eq:A.18} equals the natural number
\[
n_1 := \det \big(1 - \theta |_{X^* (Z(\mc G)^\circ / Z(\mc G)_s)} \big) \in \N .
\]
If $Z(\mc G)^\circ$ is $K$-split, then $Z(\mc G)^\circ / Z(\mc G)_s = 1$ and $n_1 = 1$.

The basis of $\hat{\mf t} \cap \hat{\mf g}_\der$ consisting of the simple coroots
is permuted by $\theta$, with orbits of length $|a \cap \Delta|$. For the second summand
in \eqref{eq:A.17} we find a contribution of
\begin{equation}\label{eq:A.19}
\lim_{s \to 0} \det \big( 1-q^{-s} \mr{Ad}_{G^\vee}(r\theta)|_{\hat{\mf t} \cap 
\hat{\mf g}_\der} \big) =
\lim_{s \to 0} \prod\nolimits_{a \in \Delta / \theta} (1 - q^{-s |a \cap \Delta|}) .
\end{equation}
The leading order term of \eqref{eq:A.19} for $s$ near 0 is
\begin{equation}\label{eq:A.23}
\prod\nolimits_{a \in \Delta / \theta} \big( s \, |a \cap \Delta| \, \log (q) \big) =
s^{|\Delta / \theta|} \log (q)^{|\Delta / \theta|}
\prod\nolimits_{a \in \Delta / \theta} |a \cap \Delta| . \qedhere
\end{equation}
\end{proof}

We start the definition of the $\mu$-functions for the relevant Hecke algebras.
Let $\Phi/\theta$ be the set of equivalence classes in the root system $\Phi$ of 
$(\hat{\mf g}, \hat{\mf t})$ as defined in \cite[Section 3.3]{Re}, and $\Delta/\theta$ 
the set of equivalence classes of the basis $\Delta$ of $\Phi$. For each $a\in \Phi/\theta$ put 
\[
\gamma_a := \sum\nolimits_{\alpha\in a} \alpha|_{\hat{\mf t}^\theta}.
\]
Then $\Phi_\theta=\{\gamma_a\}_{a\in\Phi/\theta}$ is a reduced root system on 
$\hat{\mf t}^\theta$. With $\Phi_\theta$ we also consider its (untwisted) affine extension 
$\Phi_\theta^{(1)}= \Phi_\theta \times \Z$, naturally indexed by 
$\Phi/\theta\times\mathbb{Z}$, that is, we will denote the affine root $(\gamma_a,n)$  
with $a\in\Phi/\theta$ and $n\in \Z$ by $\gamma_{(a,n)}$. 

Recall the Kac root system $\hat\Phi_\theta=\{\beta_a\mid a\in\Phi/\theta\}$, where 
$\beta_a = \alpha |_{\hat{\mf t}^\theta}$ for an $\alpha \in a$ such that $\beta_a / 2$ is
not of this form. This root system has a twisted affine extension with ``Kac diagram'' 
$\mathcal{D}(\hat{\mf g},\theta)$ \cite[Section 3.4]{Re}. 
By \cite[Section 3]{Re}, $\hat\Phi_\theta$ is the root system of $\hat{\mf g}^\theta$.
For each $a\in \Phi/\theta\times\mathbb{Z}$ there exists a positive integer $f_a$ such that 
$\gamma_a = f_a \beta_a$. If $\gamma_a\in\Phi_\theta$ is a minimal root, then by construction 
$f_{(a,1)}$ is the order of $\theta$ on the union of the components of $\Phi$ which intersect $a$. 

We say that $a \in \Phi / \theta$ (or $\alpha \in a$) has:
\begin{itemize}
\item type I if the $\theta$-orbit of $\alpha$ consists of mutually orthogonal roots;
\item type II if $a$ contains a triple $\{\alpha_1, \alpha_2, \alpha_1 + \alpha_2\}$ 
with $\alpha_2 \in \langle \theta \rangle \alpha_1$.
\end{itemize}
Type II only occurs if some irreducible component of $\Phi$ has type $A_{2n}$ and a power of
$\theta$ acts on it by the nontrivial diagram automorphism.

From \cite[Table 2]{Re} we see that for every root of type I and every $e \in \Z$: 
\[
f_{(a,e)} = f_a = |a|.
\]
On the other hand, for $a\in\Phi/\theta$ of type II :
\[
f_{(a,e)} = \left\{ \begin{array}{cc} 
f_a = 4|a|/3 & \text{if } e \text{ is even,}\\
f_a /2 = 2|a|/3 & \text{if } e \text{ is odd.}
\end{array}\right.
\]
Recall that $\mathcal{S}\subset \mathcal{G}$ denotes a maximal $K$-split torus of $\mathcal{G}$
contained in $\mathcal{T}$.
Let $\Phi(\mathcal{G},\mathcal{S})_0$ (resp. $\Phi(\mathcal{G},\mathcal{S})_1$) be the set
of indivisible (resp. non-multipliable) roots of $\Phi(\mathcal{G},\mathcal{S})$.
From \cite[(26)]{Re} we conclude that $\Phi_\theta^\vee=\Phi(\mathcal{G},\mathcal{S})_0$
and $\Phi_\theta = \Phi(\mathcal{G},\mathcal{S})_0^\vee$.

Let $I\subset G$ be an Iwahori subgroup and let $\mathcal{H}(G, I)$ be the Iwahori--Hecke 
algebra of $G$. We write the underlying root datum as $(R_0,X_*(\mathcal{S}), 
R_0^\vee,X^*(\mathcal{S}))$ and the parameter functions on $R_0$ as $m_\pm$. 
Then $R_m$ (in the sense of \cite[Subsection 2.3.3]{Opd2}) is equal to 
$\Phi_\theta$ (cf. \cite[Section 4.2]{FeOp}) or equivalently, 
$R_0^\vee=\Phi(\mathcal{G},\mathcal{S})_1$. We identify the roots of $R_0$ with 
$\{\gamma_a\}_{a\in \Phi/\theta}$.


If $a \in \Phi / \theta$ is of type I, then $\rho_0(\textup{Frob})= \mr{Ad}_{G^\vee}(r\theta)$ 
the parameters of $\mathcal{H}(G,I)$ on $\Phi_\theta$ are given by:
\begin{equation}\label{eq:A.27}
\begin{aligned}
& m_{+}(\gamma_a) = f_a, \\
& m_{-}(\gamma_a) = 0 ,
\end{aligned} 
\end{equation}
while if $a\in\Phi/\theta$ is of type II, then these parameters are given by:
\begin{equation}\label{eq:A.28}
\begin{aligned}
& m_{+}(\gamma_a) = f_{(a,0)}/2 = 2 |a| / 3, \\
& m_{-}(\gamma_a) = f_{(a,1)}/2 = |a| / 3.
\end{aligned} 
\end{equation}
Another way of expressing this is that the linearly extended parameter function $m^\vee_R$ 
on the affine Kac roots $\mathcal{D}(\hat{\mf g},\theta)$ is constant and equal to $1$, 
see \cite[Proposition 4.2.1]{FeOp}.

For use in Paragraph \ref{par:par} we need a $\mu$-function of $\mc H (G,I)$ relative to 
a Levi subgroup (or equivalently, relative to a parabolic subalgebra). 
Let $\mc P$ be a standard parabolic $K$-subgroup of $\mc G$,
with standard Levi factor $\mc M$. Let $\Phi_M \subset \Phi$ be the corresponding 
parabolic root subsystem. We recall from \cite[(3.57) and (4.96)]{Opd-Sp} that
\begin{align}
\nonumber m^M (t) & =  q^{(\dim \hat{\mf g} - \dim \hat{\mf m}) / 2}
\prod_{a\in (\Phi \setminus \Phi_M) /\theta}  \frac{(\gamma_a^{2}(t) - 1)}
{(q^{m_{-}(\gamma_a)}\gamma_a (t) + 1) (q^{m_{+}(\gamma_a)}\gamma_a (t) - 1)} \\
\label{eq:A.33} & = q^{(\dim \hat{\mf m} - \dim \hat{\mf g}) / 2}  
\prod_{a\in (\Phi \setminus \Phi_M) /\theta} \frac{(1-\gamma_a^{-2}(t))}
{(1+q^{-m_{-}(\gamma_a)}\gamma_a^{-1}(t))(1-q^{-m_{+}(\gamma_a)}\gamma_a^{-1}(t))} ,
\end{align}
a rational function of $t \in T^\vee / (1 - \theta) T^\vee$. We note that for $M$
equal to the maximal torus $T$, $m^T (t)$ involves all roots from $\Phi / \theta$.

We denote the adjoint representation of ${}^L M$ on $\mr{Lie}(G^\vee) / \mr{Lie}(M^\vee) 
= \hat{\mf g} / \hat{\mf m}$ by $\mr{Ad}_{G^\vee} |_{\hat{\mf g} / \hat{\mf m}}$.
Let $\phi_M$ be an unramified L-parameter for $M = \mc M (K)$ and write 
\[
r_M \theta = \phi_M \big( \Fr, \matje{q^{-1/2}}{0}{0}{q^{1/2}} \big) . 
\]
Upon replacing $\phi_M$ by an equivalent L-parameter we may assume that 
$r_M \in \hat{T}^{\theta,0}$ \cite[Lemma 3.2]{Re}.
For $z \in Z(M^\vee)^{\theta,\circ} \cong X_\nr (M)$ we define another 
unramified L-parameter $z \phi_M \in \Phi (M)$ by 
\begin{equation}\label{eq:A.30}
(z \phi_M) = \phi_M \text{ on } \mb I_K \times SL_2 (\C), \quad
(z \phi_M)(\Fr) = z (\phi_M (\Fr)) .
\end{equation}
By the additivity of $\gamma$-factors 
\begin{equation}\label{eq:A.35}
\gamma (s, \mr{Ad}_{G^\vee,M^\vee} \circ z \phi_M, \psi) =
\gamma (s,\mr{Ad}_{M^\vee} \circ z \phi_M,\psi) \,
\gamma (s,\mr{Ad}_{G^\vee} |_{\hat{\mf g} / \hat{\mf m}} \circ z \phi_M,\psi) .
\end{equation}

\begin{lem}\label{lem:A.3}
With the above notations:
\[
\gamma (0, \mr{Ad}_{G^\vee} |_{\hat{\mf g} / \hat{\mf m}} \circ z\phi_M, \psi) = \pm m^M (z r_M) 
\]
as rational functions of $z \in Z(M^\vee)^{\theta,\circ}$.
\end{lem}
\begin{proof}
Let $\rho, \rho_0$ be the associated self-dual Weil--Deligne representations as in 
\eqref{eq:A.13}, and let $\psi_0 : K \to \C^\times$ be an additive character of order 0. 
By \cite[Lemma 1.3]{HII} and Proposition \ref{prop:A.1}
\begin{equation}\label{eq:selfdual}
\gamma (0,\rho,\psi) = q^{(\dim \hat{\mf m} - \dim \hat{\mf g})/2} \gamma(0,\rho,\psi_0) =
\pm q^{(\dim \hat{\mf m} - \dim \hat{\mf g})/2} \gamma(0,\rho_0,\psi_0) .
\end{equation}
Using the definitions \eqref{eq:A.1} we plug \eqref{eq:A.14} into \eqref{eq:selfdual}, and
we obtain
\begin{equation}\label{eq:A.15}
\gamma (0, \mr{Ad}_{G^\vee} |_{\hat{\mf g} / \hat{\mf m}} \circ \phi, \psi) = 
\pm q^{(\dim \hat{\mf m} - \dim \hat{\mf g})/2} \lim_{s \to 0} L(1-s,\rho_0) L(s,\rho_0)^{-1} .
\end{equation}
For $a\in \Phi/\theta$ let $\hat{\mf g}_a \subset \hat{\mf g}$ be the subspace 
$\sum_{\alpha\in a}\hat{\mf g}_\alpha$, so that we have a 
$\mr{Ad}_{G^\vee}(z r_M \theta)$-stable decomposition
\begin{equation}\label{eq:A.29}
\hat{\mf g} = \hat{\mf m} \oplus \bigoplus\nolimits_{a \in (\Phi \setminus \Phi_M) / \theta} 
\hat{\mf g}_a .
\end{equation}
Using \eqref{eq:A.29} and \eqref{eq:A.15} we see that 
\begin{align}
\nonumber \gamma (0, \mr{Ad}_{G^\vee} |_{\hat{\mf g} / \hat{\mf m}} & \circ z \phi_M, \psi) 
= \pm q^{(\dim \hat{\mf m} - \dim \hat{\mf g})/2} \lim_{s\to 0}\frac{\det \big( 1 - q^{-s}  
\mr{Ad}_{G^\vee}(z r_M \theta)  |_{\hat{\mf g} / \hat{\mf m}} \big)}
{\det \big( 1 - q^{s-1} \mr{Ad}_{G^\vee}(z r_M \theta)  |_{\hat{\mf g} / \hat{\mf m}} \big)}\\
\label{eq:A.20} & = \pm q^{(\dim \hat{\mf m} - \dim \hat{\mf g})/2} \lim_{s\to 0} 
\prod_{a \in (\Phi \setminus \Phi_M) /\theta} 
\frac{\det(1-q^{-s} \mr{Ad}_{G^\vee}(z r_M \theta)|_{\hat{\mf g_a}})}
{\det(1-q^{s-1} \mr{Ad}_{G^\vee}(z r_M \theta)|_{\hat{\mf g_a}})} .
\end{align}
In \cite[Section 3.4]{Re} the characteristic polynomial of $\mr{Ad}_{G^\vee}(r\theta)$ on
$\hat{\mf g}_a$ was determined. (Strictly speaking Reeder only treats the case where $\Phi$
is irreducible, but his calculations generalize readily.) For $a\in\Phi/\theta$ of type I
this gives
\begin{equation}\label{eq:typeI}
\det(1-q^{-s} \mr{Ad}_{G^\vee}(z r_M \theta)|_{\hat{\mf g_a}})=
1-q^{-sm_{+}(\gamma_a)}\gamma_a (z r_M) ,
\end{equation}
while for $a\in\Phi/\theta$ of type II:
\begin{equation}\label{eq:typeII}
\det(1-q^{-s} \mr{Ad}_{G^\vee}(z r_M \theta)|_{\hat{\mf g_a}})=
(1+q^{-sm_{-}(\gamma_a)}\gamma_a (z r_M))(1-q^{-sm_{+}(\gamma_a)}\gamma_a (z r_M)) .
\end{equation}
With these expressions for the characteristic polynomials, \eqref{eq:A.20} becomes 
precisely $\pm m^M (z r_M)$. Finally we note that \eqref{eq:typeI} and \eqref{eq:typeII} are
regular for $z$ in a dense Zariski-open subset of $Z(M^\vee)^{\theta,\circ}$, so that 
\eqref{eq:A.20} defines a rational function on $Z(M^\vee)^{\theta,\circ}$.
\end{proof}

Consider an Iwahori subgroup $I_{ss}\subset G/Z(G)_s$. The Iwahori--Hecke algebra 
$\mathcal{H}_{ss} :=\mathcal{H}(G/Z(G)_s, I_{ss})$ of $G/Z(G)_s$ is a quotient of 
$\mathcal{H}(G, I)$. It has the same root system and the same parameter functions $m_\pm$,
and hence (essentially) the same relative $\mu$-function $m^M = m^{M / Z(G)_s}$.

Let $\overline{I_{ss}}$ be the maximal finite reductive quotient of the Iwahori subgroup
$I_{ss}\subset G/Z(G)_s$. By \cite[Proposition 3.3.5]{Car2}  
\begin{equation}\label{eq:A.11}
|\overline{I_{ss}}| = \det(q-\mr{Ad}(\theta)|_{\hat{\mf t}}).
\end{equation}
When $\mc G / Z(\mc G)_s$ is semisimple, \eqref{eq:A.11} can be expressed more explicitly
by choosing a basis of $\hat{\mf t}$ consisting of simple coroots and evaluating
the determinant:
\[
|\overline{I_{ss}}| = \prod\nolimits_{a\in \Delta/\theta}(q^{|a\cap \Delta|}-1).
\]
Given an additive character $\psi$ of $K$ with conductor $\mathfrak{p}$, 
the Haar measure $\mu_\psi$ on $G/Z(G)_s$ satisfies \cite[(1.1)]{HII} and \eqref{eq:1.3}: 
\begin{equation}\label{eq:A.16}
\vol (I_{ss}) = q^{-\mr{dim}(\hat{\mf t})/2} \det(q-\mr{Ad}(\theta)|_{\hat{\mf t}}).
\end{equation}
The $\mu$-function of the Iwahori--Hecke algebra $\mathcal{H}_{ss}$ is denoted $m_T$ in
\cite[Theorem 3.25]{Opd-Sp}. In our setting, we replace the subscript $T$ (the torus 
associated to an affine Hecke algebra) by the relevant group. 
With the above Haar measure and the normalization convention 
\cite[\S 2.4.1 and Proposition 2.5]{Opd2}, the $\mu$-function for $G / Z(G)_s$ becomes:
\begin{multline}\label{eq:mu}
m_{G / Z(G)_s} (t) = \vol (I_{ss})^{-1} m^T (t) \\
= \frac{q^{\mr{-dim}(\hat{\mf g})/2}}{\det(1-q^{-1}\mr{Ad}(\theta)|_{\hat{\mf t}})}
\prod_{a\in \Phi/\theta}\frac{(1-\gamma_a^{-2}(t))}{(1+q^{-m_{-}(\gamma_a)}\gamma_a^{-1}(t))
(1-q^{-m_{+}(\gamma_a)}\gamma_a^{-1}(t))} .
\end{multline}
Here $t$ lies in $\hat{T}/(1-\theta)\hat{T}$, the torus associated to $\mc H_{ss}$. However, 
as the roots $\gamma_a$ are trivial on $Z(G^\vee)^\theta$, \eqref{eq:A.12} entails that we may
just as well consider $m_{G / Z (G)_s}$ as a $Z(G^\vee)^\theta$-invariant function on 
$T^\vee / (1 - \theta ) T^\vee = \textup{Hom}(X_*(\mathcal{S}),\mathbb{C}^\times)$.

Recall from \cite[Proposition 3.2]{GrRe} that
$\gamma (0,\mr{Ad}_{G^\vee} \circ \phi, \psi)$ is nonzero if and only if $\phi$ is discrete.
Observe that it is a priori clear that \eqref{eq:respt}, and hence the $\gamma$-value 
\eqref{eq:A.1} for the adjoint representation $\tilde\rho$, is invariant under 
$X_\nr (G) \cong Z(G^\vee)^{\theta,\circ}$. Therefore it suffices to consider a discrete 
unramified L-parameter for $G / Z(G)_s$ in the next theorem.

\begin{thm}\label{thm:A.2} 
Fix an additive character $\psi$ of $K$ with conductor $\mathfrak{p}$.
Let $\mathcal{G}$ be an unramified reductive group over $K$. Let $\phi$ be an unramified 
discrete L-parameter for $G/Z(G)_s$, write $\tilde \rho = \mr{Ad}_{G^\vee} \circ \phi$ and 
$\rho_0 (\textup{Frob})= \mr{Ad}_{G^\vee}(r\theta)$ as in \eqref{eq:respt}. 
By \cite[Lemma 3.2]{Re} we may assume that $r\in \hat{T}^{\theta,0}$.
There exists $d \in \Q^\times$ such that, as rational functions in $q$:
\begin{multline*}
\gamma (0, \mr{Ad}_{G^\vee} \circ \phi, \psi) = d \,m_{G/Z(G)_s}^{(\{r\})} = \\
\frac{d \; q^{-\textup{dim}(\hat{\mf g})/2}}
{\det(1-q^{-1}\mr{Ad}(\theta)|_{\hat{\mf t}})}
\frac{\prod'_{a\in \Phi/\theta}(1+\gamma_a^{-1}(r))}{
\prod'_{a\in \Phi/\theta}(1+q^{-m_{-}(\gamma_a)}\gamma^{-1}_a(r))}
\frac{\prod'_{a\in \Phi/\theta}(1-\gamma_a^{-1}(r))}{
\prod'_{a\in \Phi/\theta}(1-q^{-m_{+}(\gamma_a)}\gamma^{-1}_a(r))} ,
\end{multline*} 
where $\prod'_{a\in \Phi/\theta}$ denotes the product in which zero factors are omitted.   

The constant $d$ equals $\pm 1$ if $\mc G$ is semisimple and $K$-split, while in general 
$d$ is of the form $\pm n_1 2^{n_2} 3^{n_3}$ with $\pm n_1, n_2, n_3 \in \Z$.
\end{thm}
\begin{proof}
By the additivity of $\gamma$-factors
\begin{equation}\label{eq:A.31}
\gamma (s, \mr{Ad}_{G^\vee} \circ \phi, \psi) =
\gamma (s, \mr{Ad}_{G^\vee} |_{\hat{\mf t}} \circ \phi, \psi) 
\gamma (s, \mr{Ad}_{G^\vee} |_{\hat{\mf g} / \hat{\mf t}} \circ \phi, \psi) .
\end{equation}
Define $\phi_T : \mb W_K \to {}^L T$ by $\phi_T (\Fr) = r \theta$.
By \cite[Lemma 1.3]{HII} and Proposition \ref{prop:A.1}, applied to the factor for $\hat{\mf t}$,
\eqref{eq:A.31} equals
\[
\pm q^{-\dim (\hat{\mf t}) / 2} \gamma (s, \mr{Ad}_{G^\vee} |_{\hat{\mf t}} \circ \phi_T, \psi_0) 
\gamma (s, \mr{Ad}_{G^\vee} |_{\hat{\mf g} / \hat{\mf t}} \circ \phi, \psi) .
\]
With Lemma \ref{lem:A.4} and \eqref{eq:A.20}, we find that \eqref{eq:A.31} equals
\begin{equation}\label{eq:A.32}
\pm q^{-\dim (\hat{\mf g}) / 2} \frac{\det(1-q^{-s} \mr{Ad}_{G^\vee}(r\theta)|_{\hat{\mf t}})}
{\det(1-q^{s-1} \mr{Ad}_{G^\vee}(r\theta)|_{\hat{\mf t}})} \prod_{a \in \Phi/\theta}
\frac{\det(1-q^{-s} \mr{Ad}_{G^\vee}(r \theta)|_{\hat{\mf g_a}})}
{\det(1-q^{s-1} \mr{Ad}_{G^\vee}(r \theta)|_{\hat{\mf g_a}})} .
\end{equation}
The behavior for $s \to 0$ was already analysed in Lemmas \ref{lem:A.4} and \ref{lem:A.3}.
In the current situation we can do better, by comparing the poles and the zeros.

Let $r = sc \in \hat{T}^{\theta,\circ}$ be the polar decomposition of $r$, with $s$ a torsion 
element and $c$ in the positive part of a real split subtorus.
Since $\phi$ is unramified and discrete, $H := Z_{\hat{G}}(s\theta)$ is a semisimple group and 
\[
\phi' := \phi|_{SL_2(\C)}: SL_2(\C)\to H
\]
has finite centraliser in $H$ \cite[\S 3.3]{Re}. This means that $s\theta\in \hat{G}\theta$ 
is an isolated torsion element \cite[Section 3.8]{Re}, and the root system of $H$ is a maximal 
proper subdiagram of $\mathcal{D}(\hat{\mf g},\theta)$. Moreover, $\phi'$ corresponds 
to a distinguished unipotent orbit of $H$, and 
\begin{equation}\label{eq:A.34}
c = \phi'\matje{q^{-1/2}}{0}{0}{q^{1/2}}\in \hat{T}^{\theta,\circ}\subset H.
\end{equation}
It follows \cite[Appendix A]{Opd-Sp} that the image of $r\in \hat{T}^{\theta,\circ}$ in 
$\hat{T}/(1-\theta)\hat{T}$ is a residual point for $\mathcal{H}$ (or equivalently for 
$m_{G/Z(G)_s}$).

Now we analyse the product obtained from \eqref{eq:A.32} by applying \eqref{eq:typeI}
and \eqref{eq:typeII}:
\[
\prod_{a \in \Phi/\theta}
\frac{\det(1-q^{-s} \mr{Ad}_{G^\vee}(r\theta)|_{\hat{\mf g_a}})}
{\det(1-q^{s-1} \mr{Ad}_{G^\vee}(r\theta)|_{\hat{\mf g_a}})} =
\prod_{a \in \Phi/\theta} \frac{1-q^{-sm_{+}(\gamma_a)}\gamma_a(r)}
{1-q^{(s-1)m_{+}(\gamma_a)}\gamma_a(r)} 
\frac{1 + q^{-sm_{-}(\gamma_a)}\gamma_a(r)}{1 + q^{(s-1)m_{-}(\gamma_a)}\gamma_a(r)} .
\]
The residuality of $r$ means that the pole order of this expression
at $s = 0$ is precisely $\dim (\hat{\mf t}^\theta) = |\Delta / \theta|$.
Notice that the terms with $m_- (\gamma_a) = 0$ in the numerator cancel out against
the same kind of terms in the denominator.
 
Consider a linear factor $1 \pm q^m \gamma_a (r)$, of the numerator or the denominator,
which has a zero at $s=0$. Its leading order term near $s=0$ is linear, namely
$s \log (q) m_{\pm}(\gamma_a)$. 

Let $N$ be the subset of $(a,\epsilon) \in \Phi / \theta \times \{\pm 1\}$ for which 
the corresponding term in the numerator has a pole at $s=0$, but with 
$m_\epsilon (\gamma_a) \neq 0$. Similarly we define $P$ for the denominator. Then
\begin{equation}\label{eq:A.21}
\frac{ \prod_{(a,\epsilon) \in N} 1 - \epsilon q^{-s m_\epsilon (\gamma_a)} \gamma_a (r)}
{\prod_{(a,\epsilon) \in P} 1 - \epsilon q^{(s-1) m_\epsilon (\gamma_a)} \gamma_a (r)} =
\frac{s^{-|\Delta / \theta|}}{\log (q)^{|\Delta / \theta|}}  
\frac{ \prod_{(a,\epsilon) \in N} m_\epsilon (\gamma_a)}
{\prod_{(a,\epsilon) \in P}m_\epsilon (\gamma_a)} + \mc O (s^{1-|\Delta / \theta|}) .
\end{equation}
It follows that
\begin{equation}\label{eq:A.22}
\begin{aligned}
& \prod_{a \in \Phi/\theta}
\frac{\det(1-q^{-s} \mr{Ad}_{G^\vee}(r\theta)|_{\hat{\mf g_a}})}
{\det(1-q^{s-1} \mr{Ad}_{G^\vee}(r\theta)|_{\hat{\mf g_a}})} = \\
& \frac{\prod'_{a\in \Phi/\theta}(1+\gamma_a^{-1}(r))}{
\prod'_{a\in \Phi/\theta}(1+q^{-m_{-}(\gamma_a)}\gamma^{-1}_a(r))}
\frac{\prod'_{a\in \Phi/\theta}(1-\gamma_a^{-1}(r))}{
\prod'_{a\in \Phi/\theta}(1-q^{-m_{+}(\gamma_a)}\gamma^{-1}_a(r))} \times \eqref{eq:A.21} . 
\end{aligned}
\end{equation}
From \eqref{eq:A.32}, Lemma \ref{lem:A.4} and \eqref{eq:A.22}
we conclude
\begin{multline}\label{eq:A.24}
\gamma (0, \mr{Ad}_{G^\vee} \circ \phi, \psi) = \frac{\pm n_1 q^{-\dim (\hat{\mf g}) / 2}
\prod_{a \in \Delta /\theta} |a \cap \Delta | }{\det(1-q^{-1} 
\mr{Ad}_{G^\vee}(r\theta)|_{\hat{\mf t}}) } \frac{ \prod_{(a,\epsilon) \in N} 
m_\epsilon (\gamma_a)} {\prod_{(a,\epsilon) \in P} m_\epsilon (\gamma_a)} \times \\
\frac{\prod'_{a\in \Phi/\theta}(1+\gamma_a^{-1}(r))}{
\prod'_{a\in \Phi/\theta}(1+q^{-m_{-}(\gamma_a)}\gamma^{-1}_a(r))}
\frac{\prod'_{a\in \Phi/\theta}(1-\gamma_a^{-1}(r))}{
\prod'_{a\in \Phi/\theta}(1-q^{-m_{+}(\gamma_a)}\gamma^{-1}_a(r))} .
\end{multline}
It remains to analyse the expression
\begin{equation}\label{eq:A.25}
\frac{\prod_{a \in \Delta /\theta} |a \cap \Delta | \,
\prod_{(a,\epsilon) \in N} m_\epsilon (\gamma_a)}
{\prod_{(a,\epsilon) \in P} m_\epsilon (\gamma_a)} .
\end{equation}
Since we omitted the terms with $m_\pm (\gamma_a) = 0$, \eqref{eq:A.25} is a nonzero
rational number. It factors as a product, over the irreducible components $R_i$ of 
$\Phi / \theta$, of the terms with $a \in R_i$. The restriction of $r$ to any of the $R_i$ 
is still a residual point, so there as many terms with $a \in R_i$ in numerator
as in the denominator. Let $|\theta_i|$ be the number of irreducible components
of $\Phi$ that go into $R_i$, and pick one such component $\Phi_i$. Then 
$|a| = |\theta_i| \, |a \cap \Phi_i|$ for $a \in R_i$. The factor $|\theta_i|$ appears
equally often in the numerator and in the denominator of \eqref{eq:A.25}, so it
cancels. Writing $m_{\pm,i} (\gamma_a) := m_\pm (\gamma_a) |\theta_i |^{-1}$, we 
find that \eqref{eq:A.25} equals
\begin{equation}\label{eq:A.26}
\prod_i \frac{\prod_{a \in R_i \cap (\Delta /\theta)} |a \cap \Delta \cap \Phi_i|  \,
\prod_{(a,\epsilon) \in N : a \in R_i} m_{\epsilon,i} (\gamma_a)}
{\prod_{(a,\epsilon) \in P : a \in R_i} m_{\epsilon,i} (\gamma_a)} .
\end{equation}
The formulas \eqref{eq:A.27} and 
\eqref{eq:A.28} entail that each of the factors in \eqref{eq:A.26} is the length of 
an orbit of an automorphism of a connected Dynkin diagram. That is: they are 1, 2 or 3,
where 3 can only occur for an exceptional automorphism of $D_4$. Hence \eqref{eq:A.26} is 
of the form $2^{n_2} 3^{n_3}$ with $n_2, n_3 \in \Z$. We insert this into \eqref{eq:A.24}, 
and we obtain the claimed formula for the adjoint $\gamma$-factor.

When $\mc G$ is an almost direct product of restrictions of scalars of split groups, all 
the factors in \eqref{eq:A.26} are one. In the special case where $\mc G$ is $K$-split,
also $n_1 = 1$ so that \eqref{eq:A.24} becomes the desired expression with $d = \pm 1$. 
\end{proof}

We conclude this appendix by showing that adjoint $\gamma$-factors of bounded unramified
L-parameters have real values. Notice that every such L-parameter arises from a discrete
unramified L-parameter for a Levi subgroup $M \subset G$, via an inclusion ${}^L M \to {}^L G$.

\begin{lem}\label{lem:A.5}
\enuma{
\item In the notations from Theorem \ref{thm:A.2}, 
$\gamma (0,\mr{Ad}_{G^\vee} \circ \phi, \psi) \in \R^\times$.
\item Suppose that $\phi_M \in \Phi (M)$ is a discrete bounded unramified L-parameter 
and that $z \in X_\unr (M)$. Then 
$\gamma (0, \mr{Ad}_{G^\vee} |_{\hat{\mf g} / \hat{\mf m}} \circ z\phi_M, \psi) \in \R$.
} 
\end{lem}
\begin{proof}
(a) For $t \in T^\vee$ we define $\overline t \in T^\vee$ by $x (\overline t) = \overline{x(t)}$ 
for all $x \in X^* (T^\vee)$. From \eqref{eq:A.34} we see that $\overline{r}^{-1} =
\overline{sc}^{-1} = s c^{-1}$ is conjugate to $sc$ by the element
$w_c := \phi' \matje{0}{1}{-1}{0}$. We note that $w_c$ commutes with $s$ and with $\theta$,
and that it normalizes $T^\vee$. Hence it defines an element of $W(G^\vee,T^\vee)^\theta$.

Since $\gamma (0,\mr{Ad}_{G^\vee} \circ \phi, \psi) = 
\gamma (0,\mr{Ad}_{G^\vee} \circ w_c \phi w_c^{-1}, \psi)$, the expression \eqref{eq:A.32}
does not change if we replace $r \theta$ by $w_c r \theta w_c^{-1} = \overline{r}^{-1} \theta$.
As the product in \eqref{eq:A.32} runs over all roots (both positive and negative),
we may further replace $\overline{r}^{-1} \theta$ by $\overline{r} \theta$ without changing
the value. Continuing the calculation from the proof of Theorem \ref{thm:A.2} with 
$\overline{r} \theta$ we end up with $\gamma (0,\mr{Ad}_{G^\vee} \circ \phi, \psi) =
d \, m_{G / Z(G)_s}^{(\{\overline r\})}$, which is exactly the complex conjugate of
$d \, m_{G / Z(G)_s}^{(\{r\})} = \gamma (0,\mr{Ad}_{G^\vee} \circ \phi, \psi)$.\\
(b) As observed before, we may assume that $r_M \in \hat{T}^{\theta,\circ}$.
Replacing $\phi_M$ by $t \phi_M$ (and $z$ by $z t^{-1}$) for a suitable 
$t \in X_\unr (M)$, we can further achieve that 
$r_M \in \hat{T}^{\theta,\circ} \cap M^\vee_\der$. In the proof of part
(a) we showed that $r_M$ is conjugate to $\overline{r_M}^{-1}$ by an element 
$w_c \in N_{M^\vee} (T^\vee)^\theta$. As $z = \overline{z}^{-1} \in Z(M^\vee)^{\theta,\circ}$ 
is fixed by $w_c$, we have $w_c z r_M \theta w_c^{-1} = \overline{z r_M}^{-1} \theta$.

Since $W(M^\vee,T^\vee)^\theta$ acts on $\hat{\mf g} / \hat{\mf m}$, it is clear that \eqref{eq:A.20}
does not change if we conjugate $z \phi_M$ and $z r \theta$ by $w_c$. Further, from
\eqref{eq:typeI} and \eqref{eq:typeII} we see that \eqref{eq:A.20} is invariant under
replacing $z r_M$ by $(z r_M)^{-1}$. From \eqref{eq:A.20} with $\overline{z r_M}$ instead
of $z r_M$ we obtain 
\[
\gamma (0, \mr{Ad}_{G^\vee} |_{\hat{\mf g} / \hat{\mf m}} \circ z\phi_M, \psi) = 
\pm m^M (\overline{z r_M}) .
\]
From \eqref{eq:A.33} we see that this the complex conjugate of $\pm m^M (z r_M)$.
In combination with Lemma \ref{lem:A.4} means that it is a real number, or $\infty$
if $z r_M$ happens to be a pole. 

However, the latter can not happen, and that can be seen with the residual cosets from
\cite[Appendix A]{Opd-Sp}. Namely, if $z r_M$ were a pole of $m^M$, the tempered residual 
coset $X_\unr (M) r_M$ for $m_{G / Z(G)_s}$ would contain a tempered residual coset 
(with the point $z r_M$) of smaller dimension. But that is excluded by 
\cite[Theorem A.17]{Opd-Sp}.
\end{proof}


\begin{thebibliography}{99}



\bibitem[AMS1]{AMS1} A.-M. Aubert, A. Moussaoui, M. Solleveld,
``Generalizations of the Springer correspondence and cuspidal Langlands parameters",
Manus. Math. {\bf 157} (2018), 121--192


\bibitem[AMS2]{AMS3} A.-M. Aubert, A. Moussaoui, M. Solleveld,
``Affine Hecke algebras for Langlands parameters"
arXiv:1701.03593, 2017


\bibitem[BrTi]{BrTi2} F. Bruhat, J. Tits,
``Groupes r\'eductifs sur un corps local: II. Sch\'emas en groupes.
Existence d'une donn\'ee radicielle valu\'ee'',
Publ. Math. Inst. Hautes \'Etudes Sci. {\bf 60} (1984), 5--184

\bibitem[Bor1]{Bor1} A. Borel, 
``Admissible representations of a semi-simple group over a local field with vectors fixed under an Iwahori subgroup'', 
Inventiones Math. {\bf 35} (1976) 233--259

\bibitem[Bor2]{Bor} A. Borel, 
``Automorphic L-functions", 
Proc. Symp. Pure Math {\bf 33.2} (1979), 27--61

\bibitem[BHK]{BHK} J. Bushnell, G. Henniart, P.C. Kutzko,
``Types and explicit Plancherel formulae for reductive $p$-adic groups",
pp. 55--80 in: \emph{On certain L-functions},
Clay Math. Proc. {\bf 13}, American Mathematical Society, 2011

\bibitem[BuKu]{BuKu} C.J. Bushnell, P.C. Kutzko,  
``Smooth representations of reductive $p$-adic groups: structure theory via types",
Proc. London Math. Soc. {\bf 77.3} (1998), 582--634

\bibitem[Car1]{Car1} R.W. Carter,
\emph{Simple groups of Lie type},
Pure and Applied Mathematics {\bf 28},
John Wiley \& Sons, 1972
	
\bibitem[Car2]{Car2} R.W. Carter,
\emph{Finite groups of Lie type. Conjugacy classes and complex characters},
Pure and Applied Mathematics,
John Wiley \& Sons, 1985

\bibitem[Cas]{Cas} W. Casselman,
``Introduction to the theory of admissible representations of $p$-adic reductive groups",
preprint, 1995

\bibitem[CiOp]{CiOp} D. Ciubotaru, E.M. Opdam,
``A uniform classification of discrete series representations of affine Hecke algebras",
Algebra Number Theory {\bf 11.5} (2017), 1089--1134

\bibitem[DeRe]{DeRe} S. DeBacker, M. Reeder,
``Depth-zero supercuspidal $L$-packets and their stability'',
Ann. Math. \textbf{169} (2009), 795--901


\bibitem[Dix]{Dix} J. Dixmier,
\emph{Les C*-alg\`ebres et leurs representations}
Cahiers Scientifiques {\bf 29},
Gauthier-Villars \'Editeur, Paris, 1969

\bibitem[FeOp]{FeOp} Y. Feng, E. Opdam,
``On a uniqueness property of cuspidal unipotent representations", 
arXiv:1504.03458, 2017

\bibitem[FOS]{FOS} Y. Feng, E. Opdam, M. Solleveld,
``Supercuspidal unipotent representations: L-packets and formal degrees", 
arXiv:1805.01888, 2018

\bibitem[GaGr]{GaGr}  W.T. Gan, B.H. Gross,
``Haar measure and the Artin conductor'',
Trans. Amer. Math. Soc. {\bf 351.4} (1999), 1691--1704

\bibitem[GeMa]{GeMa} M. Geck, G. Malle,
``Reductive groups and the Steinberg map'',
arXiv:1608.01156, 2016

\bibitem[GeKn]{GeKn} S.S. Gelbart, A.W. Knapp,
``L-indistinguishability and R groups for the special linear group",
Adv. in Math. {\bf 43} (1982), 101--121

\bibitem[GrRe]{GrRe} B.H. Gross, M. Reeder,
``Arithmetic invariants of discrete Langlands parameters'',
Duke Math. J. {\bf 154.3} (2010), 431--508


\bibitem[HII]{HII} K. Hiraga, A. Ichino, T. Ikeda,
``Formal degrees and adjoint $\gamma$-factors'',
J. Amer. Math. Soc. {\bf 21.1} (2008), 283--304
and correction J. Amer. Math. Soc. {\bf 21.4} (2008), 1211--1213

\bibitem[HiSa]{HiSa} K. Hiraga, H. Saito, 
``On L-packets for inner forms of $SL_n$",
Mem. Amer. Math. Soc. {\bf 1013}, Vol. {\bf 215} (2012)

\bibitem[Jac]{Jac} H. Jacquet,
``Principal L-functions of the linear group",
Proc. Symp. Pure Math {\bf 33.2} (1979), 63--86

\bibitem[KaLu]{KL} Kazhdan, D., Lusztig, G., 
``Proof of the Deligne-Langlands conjecture for Hecke algebras'',
Invent. Math. {\bf 87} (1987), 153--215



\bibitem[Lus1]{Lus-Che} G. Lusztig,
\emph{Representations of finite Chevalley groups},
Regional conference series in mathematics {\bf 39}
American Mathematical Society, 1978



\bibitem[Lus2]{LusUni1} G. Lusztig,
``Classification of unipotent representations of simple $p$-adic groups",
Int. Math. Res. Notices {\bf 11} (1995), 517--589

\bibitem[Lus3]{LusUni2} G. Lusztig,
``Classification of unipotent representations of simple $p$-adic groups II",
Represent. Theory {\bf 6} (2002), 243--289

\bibitem[Mac]{Mac} I.G. Macdonald, 
"Spherical functions on a group of $p$-adic type", 
University of Madras, 1971

\bibitem[Mor1]{Mor1} L. Morris,
``Tamely ramified intertwining algebras",
Invent. Math. {\bf 114.1} (1993), 1--54

\bibitem[Mor2]{Mor2} L. Morris,
``Level zero G-types",
Compositio Math. {\bf 118.2} (1999), 135--157

\bibitem[MoPr]{MoPr2} A. Moy, G. Prasad,
``Jacquet functors and unrefined minimal K-types",
Comment. Math. Helvetici {\bf 71} (1996), 98--121

\bibitem[Opd1]{Opd-Sp} E.M. Opdam,
``On the spectral decomposition of affine Hecke algebras",
J. Inst. Math. Jussieu {\bf 3.4} (2004), 531--648

\bibitem[Opd2]{Opd1} E. Opdam,
``Spectral correspondences for affine Hecke algebras",	
Adv. Math. {\bf 286} (2016), 912--957

\bibitem[Opd3]{Opd2} E.M. Opdam, 
``Spectral transfer morphisms for unipotent affine Hecke algebras",
Selecta Math. \textbf{22.4} (2016), 2143--2207

\bibitem[Opd4]{Opd18} E.M. Opdam,
``Affine Hecke algebras and the conjectures of Hiraga, Ichino and Ikeda on the Plancherel density",
arXiv:1807.10232, 2018

\bibitem[OpSo]{OpSo} E.M. Opdam, M. Solleveld,
``Discrete series characters for affine Hecke algebras and their formal dimensions",
Acta Mathematica {\bf 205} (2010), 105--187

\bibitem[PaRa]{HR} G. Pappas, M. Rapoport,
``Twisted loop groups and their affine flag varieties'', 
with an appendix by T. Haines and M. Rapoport, 
Adv. Math. \textbf{219.1} (2008), 118--198

\bibitem[RaRa]{RaRa} A. Ram, J. Rammage,
``Affine Hecke algebras, cyclotomic Hecke algebras and Clifford theory",
pp. 428--466 in: \emph{A tribute to C.S. Seshadri (Chennai 2002)}, 
Trends in Mathematics, Birkh\"auser, 2003

\bibitem[Ree1]{Re1} M. Reeder, 
``Formal degrees and L-packets of unipotent discrete series representations 
of exceptional $p$-adic groups. With an appendix by Frank L\"ubeck",
J. reine angew. Math. {\bf 520} (2000), 37--93

\bibitem[Ree2]{Re} M. Reeder, 
``Torsion automorphisms of simple Lie algebras'', 
L'Enseignement Mathematique \textbf{56}(2) (2010), 3--47


\bibitem[Sat]{Sat} I. Satake, 
``Theory of spherical functions on reductive algebraic groups over $p$-adic fields'', 
Publications Math\'ematiques de l'IH\'ES {\bf 18} (1963),  5--69

\bibitem[Sil1]{Sil1} A.J. Silberger
``The Knapp-Stein dimension theorem for $p$-adic groups",
Proc. Amer. Math. Soc. {\bf 68.2} (1978), 243--246 and correction
Proc. Amer. Math. Soc. {\bf 76.1} (1979), 169--170

\bibitem[Sil2]{Sil} A.J. Silberger,
``Isogeny restrictions of irreducible admissible representations are
finite direct sums of irreducible admissible representations",
Proc. Amer. Math. Soc. {\bf 73.2} (1979), 263--264



\bibitem[Sol1]{SolAHA} M. Solleveld,
``On the classification of irreducible representations of 
affine Hecke algebras with unequal parameters",
Representation Theory {\bf 16} (2012), 1--87

\bibitem[Sol2]{SolComp} M. Solleveld,
``On completions of Hecke algebras'',
pp. 207--262 in: \emph{Representations of Reductive p-adic Groups}, 
A.-M. Aubert, M. Mishra, A. Roche, S. Spallone (eds.),
Progress in Mathematics {\bf 328}, Birkh\"auser, 2019

\bibitem[Sol3]{SolLLCunip} M. Solleveld,
``A local Langlands correspondence for unipotent representations",
arXiv:1806.11357, 2018

\bibitem[Sol4]{SolFunct} M. Solleveld,
``Langlands parameters, functoriality and Hecke algebras",
arXiv:1810.12693, 2018 (to appear in Pacific J. Math.)

\bibitem[Spr]{Spr} T.A. Springer,
\emph{Linear algebraic groups 2nd ed.},
Progress in Mathematics {\bf 9}, Birkh\"auser, 1998

\bibitem[Tad]{Tad} M. Tadi\'c,
``Notes on representations of non-archimedean $SL (n)$",
Pacific J. Math. {\bf 152.2} (1992), 375--396

\bibitem[Tat]{Tat} J. Tate,
``Number theoretic background",
Proc. Symp. Pure Math {\bf 33.2} (1979), 3--26

\bibitem[Tit]{Tit} J. Tits,
``Reductive groups over local fields",
pp. 29--69 in: \emph{Automorphic forms, representations and L-functions Part I},
Proc. Sympos. Pure Math. {\bf 33}, American Mathematical Society, 1979

\bibitem[Wal]{Wal} J.-L. Waldspurger,
``La formule de Plancherel pour les groupes p-adiques (d'apr\`es Harish-Chandra)",
J. Inst. Math. Jussieu {\bf 2.2} (2003), 235--333

\end{thebibliography}
\end{document}